\documentclass[twoside,a4paper,12pt,centertags]{amsart}
\usepackage{amsmath,amssymb,verbatim,vmargin, amscd}
\usepackage[colorlinks,citecolor=red,
hypertexnames=false]{hyperref}

\theoremstyle{plain}
\newtheorem{thm}{Theorem}[section]
\newtheorem{lem}[thm]{Lemma}
\newtheorem{cor}[thm]{Corollary}
\newtheorem{prop}[thm]{Proposition}

\theoremstyle{definition}
\newtheorem{defn}[thm]{Definition}
\newtheorem{rem}[thm]{Remark}

\mathchardef\semic="303B
\newcommand{\dirac}{{D}}

\newcommand{\R}{{\mathbb R}}
\newcommand{\Q}{{\mathbb Q}}
\newcommand{\N}{{\mathbb N}}
\newcommand{\C}{{\mathbb C}}

\newcommand{\D}{{\mathbb D}}

\newcommand{\IL}{{\mathbb L}}
\newcommand{\IP}{{\mathbb P}}
\newcommand{\IH}{\mathbb H}
\newcommand{\bS}{{\mathbb S}}

\newcommand{\W}{{\mathbb W}}
\newcommand{\BMO}{\mathrm{BMO}}

\newcommand{\mH}{{\mathcal H}}

\newcommand{\mL}{{\mathcal L}}

\newcommand{\mN}{{\mathcal N}}
\newcommand{\mE}{{\mathcal E}}
\newcommand{\mD}{{\mathcal D}}

\newcommand{\mT}{{\mathcal T}}
\newcommand{\mS}{{\mathcal S}}

\DeclareMathOperator{\re}{Re}

\newcommand{\brac}[1]{\langle #1 \rangle}
\newcommand{\supp}{\text{{\rm supp}}\,}
\newcommand{\dist}{\text{{\rm dist}}\,}

\newcommand{\nul}{\textsf{N}}
\newcommand{\ran}{\textsf{R}}
\newcommand{\dom}{\textsf{D}}

\newcommand{\clos}[1]{\overline{#1}}

\newcommand{\conj}[1]{\overline{#1}}

\newcommand{\barint}{\mbox{$ave \int$}}
\newcommand{\divv}{{\text{{\rm div}}}}
\newcommand{\curl}{{\text{{\rm curl}}}}
\newcommand{\esssup}{\text{{\rm ess sup}}}

\newcommand{\tdd}[2]{\tfrac{\partial #1}{\partial #2}}

\newcommand{\wt}{\widetilde}

\newcommand{\ta}{{\scriptscriptstyle \parallel}}
\newcommand{\no}{{\scriptscriptstyle\perp}}
\newcommand{\pd}{\partial}

\newcommand{\loc}{\text{{\rm loc}}}
\newcommand{\tN}{\widetilde N_*}

\newcommand{\bx}{{\bf x}}
\newcommand{\by}{{\bf y}}

\newcommand{\reu}{\mathbb{R}^{1+n}_+}
\newcommand{\ree}{\mathbb{R}^{1+n}}

\newcommand{\pair}[2]{\langle #1,#2 \rangle}
\newcommand{\paire}[2]{( #1 \cdot #2 )}
\newcommand{\bpaire}[2]{\big( #1 \cdot #2 \big)}

\newcommand{\modz}{[z]}
\newcommand{\MM}{{\mathbf M}}
\newcommand{\SF}{S}
\newcommand{\Qpsi}[2]{\Q_{#1,#2}}
\newcommand{\Tpsi}[2]{\bS_{#1,#2}}

\def\barint_#1{\mathchoice
            {\mathop{\vrule width 6pt
height 3 pt depth -2.5pt
                    \kern -8.8pt
\intop \kern -4pt}\nolimits_{#1}}%
            {\mathop{\vrule width 5pt height
3 pt depth -2.6pt
                    \kern -6.5pt
\intop \kern -4pt}\nolimits_{#1}}%
            {\mathop{\vrule width 5pt height
3 pt depth -2.6pt
                    \kern -6pt
\intop \kern -4pt}\nolimits_{#1}}%
            {\mathop{\vrule width 5pt height
3 pt depth -2.6pt
          \kern -6pt \intop \kern -4pt}\nolimits_{#1}}}
          
          \def\bariint_#1{\mathchoice
            {\mathop{\vrule width 10pt
height 3 pt depth -2.5pt
                    \kern -12.8pt
\intop \kern -10pt\intop \kern -4pt}\nolimits_{#1}}%
            {\mathop{\vrule width 9pt height
3 pt depth -2.6pt
                    \kern -10.5pt
\intop \kern -10pt\intop \kern -4pt}\nolimits_{#1}}%
            {\mathop{\vrule width 9pt height
3 pt depth -2.6pt
                    \kern -10pt
\intop \kern -10pt\intop \kern -4pt}\nolimits_{#1}}%
            {\mathop{\vrule width 9pt height
3 pt depth -2.6pt
          \kern -10pt \intop \kern -10pt\intop \kern -4pt}
      \nolimits_{  #1}}}
          
\renewcommand{\iint}{\int \kern -10pt\int}

\usepackage{color}

\definecolor{gr}{rgb}   {0.,   0.8,   0. } 
\definecolor{bl}{rgb}   {0.,   0.5,   1. } 
\definecolor{mg}{rgb}   {0.7,  0.,    0.7}

\makeatletter
\@namedef{subjclassname@2010}{
  \textup{2010} Mathematics Subject Classification}
\makeatother

\title[Representation and uniqueness]{Representation and uniqueness for  boundary value elliptic problems via first order systems}
\author{Pascal Auscher}
\address{Laboratoire de Math\'ematiques d'Orsay, Univ. Paris-Sud, CNRS, Universit\'e Paris-Saclay, 91405 Orsay, France.} 
\email{pascal.auscher@math.u-psud.fr}

\author{Mihalis Mourgoglou}
\address{Institut des Hautes Etudes Scientifiques, 35 Route de Chartres, F-91440 Bures-sur-Yvette} 
\email{mihalis.mourgoglou@ihes.fr}

\date{November 4, 2015}

\begin{document}



\begin{abstract}  Given any elliptic system with $t$-independent coefficients in the upper-half space, we obtain  representation and trace for the conormal gradient of solutions in the natural classes for the boundary value problems of Dirichlet and Neumann types with area integral control or non-tangential maximal control. The trace spaces are  obtained in a natural range of boundary spaces which is parametrized by properties of some Hardy spaces.  This implies a complete picture of uniqueness vs solvability and well-posedness. 
\end{abstract}

\maketitle

\begin{center}
\textsl{In memory of B. Dahlberg}
\end{center}

\tableofcontents

\section{Introduction} 

  The goal of this article is to classify \textbf{all} weak solutions in the natural classes for the boundary value problems of  $t$-independent elliptic systems \eqref{eq:divform} in the upper half space.   The classification is obtained via use of first order systems and consists in  a  semigroup representation for the conormal gradients of such solutions. As a consequence, this will settle  a number of issues concerning relationships between  solvability, uniqueness and well-posedness. 

This classification will be done independently of any solvability issue, which seems  a surprising assertion. To understand this, let   us recall  the situation for  harmonic functions,  found for example in Chapter III, Appendix 4 of Stein's book \cite{Stein},  but presented in a form that suits our goals.   Let  $u$ be a harmonic function on the upper half-space, $t$ will be the vertical variable and $x$ the horizontal one. We want to know about the trace of the gradient of $u$, $\nabla u=(\partial_{t}u, \partial_{x_{1}}u, \ldots, \partial_{x_{n}}u)$, at the boundary and whether $\nabla u(t,x)$ can be recovered from its trace.    For $\frac{n}{n+1}<p<\infty$, the non-tangential maximal control  $\|(\nabla u )^*\|_{p}<\infty$  is equivalent to  $\nabla u $ at the boundary being  in $L^p$ if $p>1$ and in the  Hardy space $H^p$ if $p\le 1$, and  $\nabla u(t,x)$ can be written as the (vector-valued) Poisson extension of its trace. 
If $1<p<\infty$ (we restrict the range for convenience), the area integral control  $\|\SF(t\nabla u)\|_{p}<\infty$, together with a mild control at $\infty$, is equivalent to $\nabla u$ belongs to the Sobolev space $\dot W^{-1,p}$   on the boundary and $\nabla u(t,x)$ is given by the Poisson extension of its trace in this topology.  If one replaces area integral by a Carleson measure,   again assuming  a mild control of $u$ at $\infty$,  this is equivalent  to $\nabla u\in BMO^{-1}$ at the boundary and $\nabla u(t,x)$ is the Poisson extension of its trace. There is even  a H\"older space version of this.

In the  familiar situation of the Laplace equation,   knowledge of $u$ or of $\pd_{t}u$ at the boundary suffices to  find the harmonic function $u$ inside by a Poisson extension: this is a solvability property of this equation. With the above formulation,  it seems one does not need to know solvability of the Dirichlet or Neumann problem in any sense when one works with the complete gradient; however, this is not clear from the available proofs because solvability  could be implicit (use of Fatou type results, commutation of operators...). This observation is of particular interest when dealing with the more general systems \eqref{eq:divform} below,  as one may not \textit{a priori} know solvability (but would like to have it), which is usually  difficult  to establish;  
still there is room for proving existence of a trace and representation  of the gradient as a first step.  Solving the boundary value problem, that is,  constructing a solution  from the knowledge of one component of the  gradient at the boundary  amounts to  an inverse problem on the boundary in a second step,  and this is a different (albeit related) question.

Let us introduce our notation. We denote points in $\ree$ by boldface letters $\bx,\by,\ldots$ and in coordinates in $\R \times \R^n$ by $(t,x)$ etc. We set $\R^{1+n}_+=(0,\infty)\times \R^n$. 
Consider the system of $m$  equations  given by
\begin{equation}  \label{eq:divform}
  \sum_{i,j=0}^n\sum_{\beta= 1}^m \pd_i\big( A_{i,j}^{\alpha, \beta}(x) \pd_j u^{\beta}(\bx)\big) =0,\qquad \alpha=1,\ldots, m
\end{equation}
in $\R^{1+n}_+$,
where $\pd_0= \tdd{}{t}$ and $\pd_i= \tdd{}{x_{i}}$ if $i=1,\ldots,n$.  For short, we write $Lu=-\divv A \nabla u=0$ to mean \eqref{eq:divform}, where we always assume that the matrix   \begin{equation}   \label{eq:boundedmatrix}
  A(x)=(A_{i,j}^{\alpha,\beta}(x))_{i,j=0,\ldots, n}^{\alpha,\beta= 1,\ldots,m}\in L^\infty(\R^n;\mL(\C^{m(1+n)})),
\end{equation} is bounded and measurable, independent of $t$, and satisfies  the strict accretivity condition on  the subspace $\mH$  of $ L^2(\R^n;\C^{m(1+n)})$ defined by $(f_{j}^\alpha)_{j=1,\ldots,n}$ is curl free in $\R^n$ for all $\alpha$,  that is,
for some $\lambda>0$ 
\begin{equation}   \label{eq:accrassumption}
   \int_{\R^n} \re (A(x)f(x)\cdot  \conj{f(x)}) \,  dx\ge \lambda 
   \sum_{i=0}^n\sum_{\alpha=1}^m \int_{\R^n} |f_i^\alpha(x)|^2dx, \ \forall  f\in \mH.
\end{equation}
The system \eqref{eq:divform} is always considered in the sense of distributions with weak solutions, that is  $H^1_{loc}(\R^{1+n}_{+};\C^m)=W^{1,2}_{loc}(\R^{1+n}_{+};\C^m)$ solutions. 
We remark that the equation \eqref{eq:divform} is intrinsic in the sense that it does not depend  on the many choices that can be taken for $A$ to represent $L$. Any such $A$ with the required properties will be convenient. Our method will be in some sense $A$ dependent. We will come back to this when stating the well-posedness results. 

It was proved in \cite{AA1} that weak solutions of $Lu=0$  in the classes $$\mE_{0}=\{u;\|\tN(\nabla u )\|_{2}<\infty\}$$ or $$\mE_{-1}=\{u;\|\SF(t\nabla u)\|_{2}<\infty\}$$ (where $\tN(f)$ and
$S(f)$ stand for a non-tangential maximal function  and square function: definitions will be given later) have a  semigroup representation in their conormal gradient
\begin{equation}
\label{eq:conormal}
\nabla\!_{A} u(t,x):=  \begin{bmatrix} \pd_{\nu_A}u(t,x)\\ \nabla_x u(t,x) \end{bmatrix}.
\end{equation}
More precisely, one has
\begin{equation}
\label{eq:representation}
\nabla\!_{A} u(t,\, .\,)= S(t) ( \nabla\!_{A} u|_{t=0})
\end{equation}
for a certain semigroup $S(t)$ acting on the subspace $\mH$ of $L^2$ in the first case and in the corresponding subspace in $\dot H^{-1}$, where $\dot H^s$ is  the homogeneous Sobolev space of order $s$, in the second case. Actually, another equivalent representation was obtained for $u\in \mE_{-1}$ and this one was only explicitly  derived in  subsequent works (\cite{AMcM, R2}) provided one defines the conormal gradient at the boundary in this subspace of $\dot H^{-1}$. In \cite{R2}, the semigroup representation was extended to 
intermediate classes of solutions  defined by $\mE_{s}=\{u; \|\SF(t^{-s}\nabla u)\|_{2}<\infty\}$ for $-1<s<0$ and the semigroup representation holds in $\dot H^{s}$. In particular, for $s=-1/2$, this is the class of energy solutions used in \cite{AMcM,AM} (other ``energy'' classes were defined in \cite{KR} and used in \cite{HKMP2}). So this allows one to deal with any problem involving energy solutions using this first order method, but this is restrictive.

In \cite{AS}, a number of  \textit{a priori} estimates was proved concerning the solutions enjoying the representation \eqref{eq:representation}  for general systems \eqref{eq:divform}. In  certain ranges of $p$,  $ \nabla\!_{A} u|_{t=0}$ belongs to an identified boundary space and its norm is equivalent to one of these interior controls. Thus, it remained to eliminate this \textit{a priori} information. 
This is  what we do here by  showing existence of the trace and  semigroup  representation for conormal gradients of  solutions in the classes 
$\|\tN(\nabla u )\|_{p}<\infty$ or $\|\SF(t\nabla u)\|_{p}<\infty$ for $p\ne 2$ (and more) and for the ranges of $p$ described in \cite{AS}.   No other assumption than $t$-independence and ellipticity is required.  Hence,  what we are  after  here are  uniqueness results for the \textbf{initial value problem} of the first order equation  \eqref{eq:rep} below. The case $p=2$  in both cases was done in  \cite{AA1}. 

To formulate the results, we need to recall the main discovery of \cite{AAMc} that the system \eqref{eq:divform} is in correspondence with a first order system of Cauchy-Riemann type
\begin{equation}
\label{eq:rep}
\pd_{t}\nabla\!_{A} u+DB \nabla\!_{A} u=0,
\end{equation}
where
 \begin{equation}
\label{eq:dirac}
 \dirac:= 
    \begin{bmatrix} 0 & \divv_{x} \\ 
     -\nabla_{x} & 0 \end{bmatrix},
\end{equation}
and
 \begin{equation}
\label{eq:hat}
B= \hat A:=  \begin{bmatrix} 1 & 0  \\ 
    c & d \end{bmatrix}\begin{bmatrix} a & b  \\ 
    0 & 1 \end{bmatrix}^{-1}=  \begin{bmatrix} a^{-1} & -a^{-1} b  \\ 
    ca^{-1} & d-ca^{-1}b \end{bmatrix}
\end{equation}
whenever we write
\begin{equation}
\label{eq:A}
A= \begin{bmatrix} a & b  \\ 
    c & d \end{bmatrix}
\end{equation}
    and $L$  in the form
    \begin{equation}
\label{eq:L}
L=-\begin{bmatrix}   \pd_{t}& \nabla_x  \end{bmatrix}
  \begin{bmatrix} a & b  \\ 
    c & d \end{bmatrix} \begin{bmatrix}   \pd_{t}\\ \nabla_x  \end{bmatrix}.
\end{equation}
The operators  $D$ and $B$  satisfy the necessary requirements so that $DB$, a perturbed Dirac type operator, is a bisectorial operator and, by a result in \cite{AKMc},  has bounded holomorphic functional calculus on $L^2$:   the semigroup $S(t)$ is built from an extension of $e^{-t|DB|}$, the semigroup generated by $-|DB|$ on $L^2$, from the  closure of the range of $DB$ in $L^2$. We note that for $L^*$, the associated system is not given by $B^*$ but by $\wt B= NB^*N$ where $N= \begin{bmatrix} I & 0  \\ 
    0 & -I \end{bmatrix}$.

In \cite{AS},  the  Hardy space $H^p_{DB}$ associated to $DB$ was exploited: on this space the semigroup has a  bounded extension by construction. For the boundary value problems  the two natural spectral subspaces $H^{p,+}_{DB}$ and $H^{p, -}_{DB}$ obtained as the ranges of  the (extensions of the) bounded projections $\chi^+(DB)$ and $\chi^-(DB)$ respectively come into play.  Formally, a solution  to  $Lu=0$ on the upper half-space can be constructed from $\nabla\!_{A} u(t,\, .\,)= e^{-t DB}\chi^+(DB) F_{0}$ for some  $F_{0}\in H^p_{DB}$ and a solution to $Lu=0$ on the lower half-space from $\nabla\!_{A} u(t,\, .\,)= e^{t DB}\chi^-(DB) F_{0}$  for some  $F_{0}\in H^p_{DB}$. In \cite{R1},  these operators are called Cauchy extension operators because this is exactly what is obtained for the $L=-\Delta$ in two dimensions: the  formula with $\chi^+(DB)$ gives the  analytic extension  of functions on the real line to the upper half-space and the one with $\chi^-(DB)$ the analytic extension to the lower half-space.

However, the Hardy spaces  are by definition abstract completions.  To relate this to classical boundary spaces, one needs to have information on these Hardy spaces.   It was the main thesis of \cite{AS} to  obtain a range of $p$, an open interval called $I_{L}=(a,p_{+}(DB))\subset (\frac{n}{n+1},\infty)$, for which   $H^p_{DB}=H^p_{D}$, a closed and complemented subspace of $L^p$ if $p>1$ and $H^p$, the real Hardy space, if $p\le 1$. The number $p_{+}(DB)$ has a certain meaning there. In fact,  the notation $I_{L}$ could be misleading as it depends on the choice of $B$ thus of $A$, but we use it for convenience.   This allowed to obtain comparisons between trace estimates at the boundary and interior control in classical function spaces at the boundary and in the interior. Remark that for any $p$,      the method in \cite{AS}  still furnishes weak solutions of the system \eqref{eq:divform} (this is  not explicitly written there) for data in a subspace of $H^p_{DB}$. However, for $p\notin I_{L}$, $H^p_{DB}$ is not a space of distributions, so we are unable to compare with 
classical situations (because limits, taken in different ambient spaces,  cannot be identified).

Our goal here is to go backward: begin with an arbitrary solution in some class and prove the desired representation in the restricted range of exponents imposed by the Hardy space theory. In all, this gives two classification results. 

\begin{thm}\label{thm:main1} Let $n\ge 1$, $m\ge 1$. Let $\frac{n}{n+1}<p<p_{+}(DB)$ be such that $H^p_{DB}=H^p_{D}$ with equivalence of norms. Then, for any weak solution $u$ to $Lu=0$ on $\reu$, the following are equivalent: 
\begin{enumerate}
  \item[(i)] $\|\tN(\nabla u )\|_{p}<\infty$. 
  \item[(ii)] $\|\SF(t\partial_{t}\nabla u)\|_{p}<\infty$  and  $\nabla\!_{A} u(t,\cdot)$ converges to 0 in the sense of distributions as $t\to \infty$. 
    \item[(iii)] $ \exists! F_{0}\in  H^{p,+}_{DB}$, called the conormal gradient of $u$ at $t=0$ and denoted by $ \nabla\!_{A} u|_{t=0}$, such that
  $\nabla\!_{A} u(t,\, .\,)= S_{p}(t)( \nabla\!_{A} u|_{t=0})$ for all $t\ge 0$. 
   \item[(iv)] $ \exists F_{0}\in  H^{p}_{D}$  such that
  $\nabla\!_{A} u(t,\, .\,)=  S_{p}^{+}(t) F_{0}$ for all $t\ge 0$.  
\end{enumerate}
Here,  $S_{p}(t)$ is the bounded extension to $H^p_{D}$ of the semigroup $e^{-t|DB|}$ originally defined on $H^2_{D} $, $S_{p}^{+}(t) $ is the extension to $H^p_{D}$ of $e^{-t DB}\chi^+(DB)$ and both agree on $H^{p,+}_{DB}$. 
If any of the conditions above hold, then
\begin{equation}
\label{eq:comparison}
\|\tN(\nabla u )\|_{p} \sim \|\SF(t\partial_{t}\nabla u)\|_{p} \sim \| \nabla\!_{A} u|_{t=0}\|_{H^p}\sim \| \chi_{p}^+ F_{0}\|_{H^p},
\end{equation}
where $\chi_{p}^+$ is the continuous extension of $\chi^+(DB)$ on $H^p_{D}$. 

\end{thm}

This has the following corollary.

\begin{cor}\label{cor:main1} Let $p$ under the conditions of Theorem \ref{thm:main1} and $u$ be a weak solution to $Lu=0$ on $\reu$ with $\|\tN(\nabla u )\|_{p}<\infty$. Then, we have the following regularity properties: $t\mapsto \nabla\!_{A} u(t,\, .\,) \in C_{0}([0,\infty);H^{p,+}_{DB}) \cap C^\infty(0,\infty;H^{p,+}_{DB}) $ and 
\begin{equation}
\label{eq:suphp}
 \| \nabla\!_{A} u|_{t=0}\|_{H^p}  \sim \sup_{t\ge 0}\|\nabla\!_{A} u(t,\, .\,)\|_{H^p}.
\end{equation}
Moreover, if $p<n$, $u=\tilde u+c$, where $t\mapsto  \tilde u(t,\, .\,) \in C_{0}([0,\infty); \dot H^{1,p} \cap L^{p^*}) \cap C^\infty(0,\infty;\dot H^{1,p}\cap L^{p^*})$ and $c\in \C^m$, and  
\begin{equation}
\label{eq:suplp*}
 \sup_{t\ge 0}\| \tilde u(t,\, .\,)\|_{p^*}    \lesssim \| \nabla\!_{A} u|_{t=0}\|_{H^p}.
\end{equation}
If $p\ge n$, then $t\mapsto   u(t,\, .\,) \in C_{0}([0,\infty);\dot H^{1,p}\cap \dot\Lambda^s)\cap C^\infty(0,\infty;\dot H^{1,p}\cap \dot\Lambda^s)$ with $s=1-\frac{n}{p}$ and  
\begin{equation}
\label{eq:suplambdas}
 \sup_{t\ge 0}\|  u(t,\, .\,)\|_{\dot\Lambda^s}   \lesssim \| \nabla\!_{A} u|_{t=0}\|_{L^p}.
\end{equation}
In addition,  we have the almost everywhere limits when $p\ge 1$,
\begin{equation}
\label{eq:CVaegradA}
\lim_{t\to 0}\ \bariint_{W(t,x)} \nabla\!_{A} u(s,y) \, dsdy  = \lim_{t\to 0}\ \barint_{B(x,t)} \nabla\!_{A} u(t,y)  \, dy= \nabla\!_{A} u|_{t=0}(x)
\end{equation}
and similarly for the time derivatives $\pd_{t}u$, and the  almost everywhere  limit for $u$ whatever $p$,
\begin{equation}
\label{eq:CVaeu1}
\lim_{t\to 0}\ \bariint_{W(t,x)}  u(s,y) \, dsdy  = \lim_{t\to 0}\ \barint_{B(x,t)}  u(t,y)  \, dy = u|_{t=0}(x).
\end{equation}
\end{cor}

 Here, we adopt the convention that $H^p=L^p$ if $p>1$. The notation $C_{0}$ stands for continuous functions that vanish at $\infty$. The exponent $p^*=\frac{np}{n-p}$ is the Sobolev exponent. The Whitney regions will be defined in Section \ref{sec:tent}.

 If (iii) holds, then (iv) holds clearly and conversely (iv) implies (iii)  with $\nabla\!_{A} u|_{t=0}= \chi^+_{p} F_{0}$. 
 The second condition in (ii) is  mild and only meant to control the growth of the gradient at infinity. It cannot  be avoided. 
  It follows from the results in  \cite{AS},  that (iii) implies (i) and (iii) implies (ii) and in this case \eqref{eq:comparison} holds.  Here, we prove the converses:    existence of the semigroup equation and trace in (iii) for the indicated topology.

  We  remark that the estimates \eqref{eq:suphp}, \eqref{eq:suplp*} and \eqref{eq:suplambdas} come \textit{a posteriori} in Corollary \ref{cor:main1}: they are regularity results for the class of solutions in (i) or (ii).  We are not sure we could run the argument taking the condition  \eqref{eq:suphp}  or even the weaker one  $\sup_{t\ge 0}\|\barint_{\,[t,2t]}\nabla\!_{A} u(s,\, .\,)\,ds\|_{H^p}<\infty$ as a starting point, except if $p=2$ (this is observed in \cite{AA1}) or $p$ near 2. We shall not attempt to prove this.

 A second theorem requires the use of negative order Sobolev and H\"older spaces. We will recall the definitions later. 
 
 \begin{thm}\label{thm:main2} Let $n\ge 1$, $m\ge 1$. Let $\frac{n}{n+1}<q<p_{+}(D\wt B)$ be such that $H^q_{D\wt B}=H^q_{D}$ with equivalence of norms.  
Let  $u$ be a weak solution to $Lu=0$ on $\reu$. 

First when $q>1$ and $p=q'$, the following are equivalent:
 \begin{enumerate}
\item[($\alpha$)]  $\|\SF(t\nabla u)\|_{p}<\infty$ and $u(t,\cdot)$ converges to $0$ in $\mD'$ modulo constants as $t\to \infty$ if $p\ge  2^*$. 
  \item[($\beta$)] $ \exists! F_{0}\in  \dot W^{-1,p,+}_{DB}$, called the conormal gradient of $u$ at $t=0$ and denoted by $ \nabla\!_{A} u|_{t=0}$, such that
  $\nabla\!_{A} u(t,\, .\,)= \wt S_{p}(t) ( \nabla\!_{A} u|_{t=0})$ for all $t\ge 0$.
   \item[($\gamma$)] $ \exists F_{0}\in   \dot W^{-1,p}_{D}$,  such that
  $\nabla\!_{A} u(t,\, .\,)= \wt S_{p}^+(t)F_{0}$ for all $t\ge 0$.  
\end{enumerate}
Here,  $\wt S_{p}^+(t)$ is the extension of $e^{-t DB}\chi^+(DB)$ to $\dot W^{-1,p}_{D}$  which agrees with the extension  $\wt S_{p}(t)$  of $e^{-t |DB|}$ on $\dot W^{-1,p,+}_{DB}$.
If any of the conditions above hold, then
\begin{equation}
\label{eq:main2}
 \|\SF(t\nabla u)\|_{p} \sim \| \nabla\!_{A} u|_{t=0}\|_{\dot W^{-1,p}}\sim \| \wt \chi_{p}^+ F_{0}\|_{\dot W^{-1,p}},
\end{equation}
 where $\wt \chi_{p}^+$ is the bounded extension of $\chi^+(DB)$ on $\dot W^{-1,p}_{D}$.

Second when $q\le 1$ and $\alpha=n(\frac{1}{q}-1)\in [0,1)$,  the following are equivalent:
 \begin{enumerate}
\item[(a)]  $\|t\nabla u\|_{T^\infty_{2,\alpha}}<\infty$ and $u(t,\cdot)$ converges to $0$ in $\mD'$ modulo constants as $t\to \infty$. 

  \item[(b)] $ \exists! F_{0}\in  \dot  \Lambda^{\alpha-1,+}_{DB}$, called the conormal gradient of $u$ at $t=0$ and denoted by $ \nabla\!_{A} u|_{t=0}$, such that
  $\nabla\!_{A} u(t,\, .\,)= \wt S_{\alpha}(t) ( \nabla\!_{A} u|_{t=0})$ for all $t\ge 0$.
   \item[(c)] $ \exists F_{0}\in   \dot  \Lambda^{\alpha-1}_{D}$,  such that
  $\nabla\!_{A} u(t,\, .\,)=\wt S_{\alpha}^+(t) F_{0}$ for all $t\ge 0$,  
\end{enumerate}
where $\wt S_{\alpha}^+(t)$ is the extension (defined by  weak-star duality) of $e^{-t DB}\chi^+(DB)$ to $ \dot  \Lambda^{\alpha-1}_{D}$  which agrees with the extension (also defined by  weak-star duality)  $\wt S_{p}(t)$  of $e^{-t |DB|}$ on $\dot  \Lambda^{\alpha-1,+}_{DB}$.
If any of the conditions above hold, then
\begin{equation}
\label{eq:t2alpha}
\|t\nabla u\|_{T^\infty_{2,\alpha}} \sim \| \nabla\!_{A} u|_{t=0}\|_{\dot  \Lambda^{\alpha-1}}\sim \| \wt \chi_{\alpha}^+ F_{0}\|_{\dot  \Lambda^{\alpha-1}},
\end{equation}
where $\wt \chi_{\alpha}^+$ is the bounded extension of $\chi^+(DB)$ on $\dot \Lambda^{\alpha-1}_{D}$.

\end{thm}

Let us mention that in the case $\alpha=0$,  $\|t\nabla u\|_{T^\infty_{2,\alpha}} <\infty$ means that $|t\nabla u(t,x)|^2 \, \frac{dtdx}{t} $ is a  Carleson measure and  $\dot  \Lambda^{-1}$ is the space $\BMO^{-1}$.  

The condition at $\infty$ in ($\alpha$)  is used to eliminate some constant solutions in $t$. When $p< 2^*= \frac{2n}{n-2}$, it  follows from $\|\SF(t\nabla u)\|_{p}<\infty$ and thus is redundant.  Our statement is therefore  in agreement with the $p=2$ result of \cite{AA1}. 
Again, ($\beta$) is equivalent to ($\gamma$) and ($\beta$) implies ($\alpha$) is proved in \cite{AS}.   Similarly, ($b$) is equivalent to ($c$) and ($b$) implies ($a$) is proved in \cite{AS}.  
We show here the converses. 

\begin{cor}\label{cor:main2}
In the first case of Theorem \ref{thm:main2}, 
$$t\mapsto \nabla\!_{A} u(t,\, .\,) \in C_{0}([0,\infty);\dot W^{-1,p,+}_{DB})\cap C^\infty(0,\infty; \dot W^{-1,p,+}_{DB})$$ and 
\begin{equation}
\label{eq:supw-1p}
 \| \nabla\!_{A} u|_{t=0}\|_{\dot W^{-1,p}}  \sim \sup_{t\ge 0}\|\nabla\!_{A} u(t,\, .\,)\|_{\dot W^{-1,p}},
\end{equation}
and $u=\tilde u+c$, where $t\mapsto  \tilde u(t,\, .\,) \in C_{0}([0,\infty);L^{p})\cap C^{\infty}(0,\infty; L^p)$ and $c\in \C^m$ and  
\begin{equation}
\label{eq:suplp}
 \sup_{t\ge 0}\| \tilde u(t,\, .\,)\|_{p} \lesssim \| \nabla\!_{A} u|_{t=0}\|_{\dot W^{-1,p}}.
\end{equation}
In addition, we have the non-tangential maximal estimate
\begin{equation}
\label{eq:ntmax}
\|\tN (\tilde u)\|_{p}\lesssim \|S(t\nabla u)\|_{p}
\end{equation}
and the  almost everywhere limit 
\begin{equation}
\label{eq:CVaeu}
\lim_{t\to 0}\ \bariint_{W(t,x)}  u(s,y) \, dsdy  = \lim_{t\to 0}\ \barint_{B(x,t)}  u(t,y)  \, dy = u|_{t=0}(x).
\end{equation}

In the second case of Theorem \ref{thm:main2}, 
$$t\mapsto \nabla\!_{A} u(t,\, .\,) \in C_{0}([0,\infty);\dot  \Lambda^{\alpha-1,+}_{DB})\cap C^\infty(0,\infty; \dot  \Lambda^{\alpha-1,+}_{DB}),$$ where  $\dot\Lambda^{\alpha-1}$ is equipped with weak-star topology, and 
\begin{equation}
\label{eq:supgradualpha}
\| \nabla\!_{A} u|_{t=0}\|_{\dot  \Lambda^{\alpha-1}}  \sim \sup_{t\ge 0}\|\nabla\!_{A} u(t,\, .\,)\|_{\dot  \Lambda^{\alpha-1}}.
\end{equation}
Next $t\mapsto   u(t,\, .\,) \in C_{0}([0,\infty);\dot\Lambda^\alpha)\cap C^\infty(0,\infty; \dot  \Lambda^{\alpha}) $, where  $\dot \Lambda^\alpha$ is equipped with the weak-star topology, and  
\begin{equation}
\label{eq:suplambdasdir}
 \sup_{t\ge 0}\|  u(t,\, .\,)\|_{\dot\Lambda^\alpha} \lesssim  \| \nabla\!_{A} u|_{t=0}\|_{\dot  \Lambda^{\alpha-1}}. \end{equation}
 Moreover,  $u\in\dot\Lambda^\alpha(\overline{\reu})$ with
 \begin{equation}
\label{eq:global}
\|u\|_{\dot\Lambda^\alpha(\overline{\reu})} \lesssim \| \nabla\!_{A} u|_{t=0}\|_{\dot  \Lambda^{\alpha-1}}.
\end{equation}
For $\alpha=0$, this is a $BMO$ estimate on $\reu$.
\end{cor}

Remark that the estimate \eqref{eq:ntmax} holds for any weak solution for which the right hand side is finite (for those $p$), and nothing else. 

Note that we have defined a number of apparently different notions of conormal gradients at the boundary. They are all consistent. Namely, if a solution satisfies several of the conditions  in the range of applicability of our results  then the conormal gradients at the boundary associated to each of these conditions are the same.  This is again because everything happens in the ambient space of Schwartz distributions. 

\

Let us turn to   boundary value problems  for solutions of $Lu=0$ or $L^*u=0$ and formulate four such problems:
\begin{enumerate}
  \item $(D)_{Y}^{L^*}$:  $L^*u=0$,  $u|_{t=0}\in Y$, $t\nabla u\in \wt \mT$. 
  \item $(R)_{X}^L$:  $Lu=0$, $\nabla_{x}u|_{t=0}\in X$, $\tN(\nabla u)\in \mN$.
  \item $(N)_{Y^{-1}}^{L^*}$:  $L^*u=0$,  $\partial_{\nu_{A^*}}u|_{t=0}\in \dot Y^{-1}$, $t\nabla u\in \wt \mT$.
   \item $(N)_{X}^L$: $Lu=0$, $\partial_{\nu_{A}}u|_{t=0}\in X$, $\tN(\nabla u)\in \mN$.
\end{enumerate}

Here we restrict ourselves to   $q\in I_{L}$. Then  $\mN=L^q$,  $X=L^q$ if $q>1$ and $X=H^q$ if $q\le 1$. Next, 
   $Y$ is the dual space   $X$ (we are ignoring whether functions are scalar or vector-valued; context is imposing it)  and $\dot Y^{-1}= \divv_{x}(Y^n)$ with the quotient topology, equivalently  $\dot Y^{-1}$ is the dual of $\dot X^{1}$ defined by $\nabla f\in X$. 
   Finally,  $t\nabla u\in \wt \mT$ means that $t\nabla u$ belongs  the tent space $\mT$ equal to $T^{q'}_{2}$ if $q > 1 $, to the weighted Carleson measure space $T^\infty_{2,n(\frac{1}{q}-1)}$ if $q\le 1$, and $u(t,\cdot)$ converges to 0 in $\mD'$ modulo constants if $t\to \infty$.

In each case, solving means finding a solution  with control from the data. For example, for $(D)_{Y}^{L^*}$ we want $\|t\nabla u\|_{\mT} \lesssim \|u|_{t=0}\|_{Y}$. If one can do this for all data then the open mapping theorem furnishes the implicit constant. The behavior at the boundary is the strong or weak-star convergence specified by our corollaries above (almost everywhere convergence of Whitney averages is available as a bonus).  Uniqueness means that there is at most one solution (modulo constants for $(R)_{X}^L$ and $(N)_{X}^L$) in the specified class for a given boundary data.    Well-posedness is the conjunction of both existence of a solution for all data and uniqueness. No additional restriction is imposed.

Recall that  the trace $h$ of a conormal gradient    contains two terms:  the first one is called the scalar part  and denoted by $h_{\no}$,  and the second one called the tangential part and denoted by $h_{\ta}$, which has the gradient structure, \textit{i.e.,} $h=[h_{\no},h_{\ta}]^T$.  The two maps  
$N_{\no}: h\mapsto  h_{\no}$ and $N_{\ta}: h\mapsto h_{\ta}$ are of importance in this context because they carry the well-posedness.   Denote by $X_{\no}$ and $X_{\ta}$ the corresponding parts of the space $X_{D}$ defined as the image of $X$ under the (extension to $X$) of the orthogonal projection on the $L^2$ range of the Dirac operator $D$. Do similarly for $\dot Y^{-1}_{D}$.

\begin{thm}\label{thm:wpequiv} We have the following assertions for $q\in I_{L}$. 
\begin{enumerate}
  
 \item $(D)_{Y}^{L^*}$ is well-posed if and only if $N_{\ta}: \dot Y^{-1,+}_{D\wt B} \to \dot Y^{-1}_{\ta}$ is an isomorphism. 
  \item $(R)_{X}^L$ is well-posed if and only if $N_{\ta}: X^{+}_{DB} \to X_{\ta}$  is an isomorphism.
  \item $(N)_{Y^{-1}}^{L^*}$  is well-posed if and only if $N_{\no}: \dot Y^{-1,+}_{D\wt B} \to \dot Y^{-1}_{\no}$   is an isomorphism.
   \item $(N)_{X}^L$ is well-posed if and only if $N_{\no}: X^{+}_{DB} \to X_{\no}$  is an isomorphism. 
   \end{enumerate}
   In each case, ontoness is equivalent to  existence and injectivity  to uniqueness. 
\end{thm}

As said earlier, our method is $A$ dependent in the sense that it builds a representation on the conormal gradient for $\nabla\!_{A}u$, hence the trace space depends on $A$.  However, the well-posedness of the Dirichlet and regularity problems are intrinsic: it does not depend on which 
$A$ is chosen to represent the operator $L$, thus $N_{\ta}$ is invertible for any possible choice of  $A$. It is tempting to think that there is a better choice of $A$  than other which could lead to invertibility of $N_{\ta}$. 
The Neumann problem is of course $A$ dependent because we impose the conormal derivative. For example, for $Lu=0$ being a real and symmetric equation, then one chooses $A$ real and symmetric,  which is the only possible choice with this property. Then (2) and (4) are well-posed when $X=L^2$ (existence follows from   \cite{KP} and uniqueness, together with an other proof of existence via layer potentials, follows from  \cite{AAAHK}. It is even extendable to all systems with $A=A^*$ in complex sense combining \cite{AAMc} and \cite{AA1}.

We continue with a discussion on duality principles first studied in \cite{KP,KP2},  obtaining the sharpest possible version (in our context) compared to earlier results (\cite{KP, KP2, Shen2, Shen1, May, HKMP2, AM}).

\begin{thm}\label{cor:dualityprinciple} Let $n\ge 1$, $m\ge 1$. Let $q\in I_{L}$. If $q>1$, 
\begin{enumerate}
  \item $(D)_{Y}^{L^*}$ is well-posed if and only if  $(R)_{X}^L$ is well-posed.  
    \item $(N)_{Y^{-1}}^{L^*}$ is well-posed if and only if $(N)_{X}^L$ is well-posed.
          \end{enumerate}
          If $q\le1$, then the `if' direction holds in both cases.  
 
\end{thm}

In fact, these   duality principles are best seen from   extensions of  some abstract results on pairs of projections  in \cite{AAH}. In concrete terms, this corresponds to  the Green's formula; we will do a direct proof where in the way, we  use a easy consequence of our method:   any solution in  $\mN$ or $\wt \mT$ can be approximated by an energy solution in the topology $\mN$ or $\wt \mT$
respectively.  This is possible by approximating  the conormal gradient at the boundary, not just the Neumann or Dirichlet data.

\begin{thm}\label{thm:green} Let $n\ge 1$, $m\ge 1$. Let $q\in I_{L}$. Then, for any weak solution $u$ to $Lu=0$ on $\reu$ with $\tN(\nabla u)\in \mN$
 and any weak solution $w$ to $L^*w=0$ on $\reu$ with  $t\nabla w\in \mT$ and $w(t,\cdot)$ converging to 0 in $\mD'$ modulo constants as $t\to \infty$,
 \begin{equation}
\label{eq:green}
\pair {\partial_{\nu_{A}}u|_{t=0}} {w|_{t=0}} = \pair {u|_{t=0}}{\partial_{\nu_{A^*}}w|_{t=0}} .
\end{equation}
Here the first pairing is the $\pair X Y$ (sesquilinear) duality while the second one is the 
$\pair{\dot X^{1}}{ \dot Y^{-1}}$ (sesquilinear) duality. 
\end{thm}

There is a notion of solvability attached to energy solutions which was the one originally introduced in \cite{KP}. Indeed, the ``classical'' Dirichlet and Neumann problems for solutions with 
$\|\nabla u\|_{2}<\infty$ are well-posed by Lax-Milgram theorem: for all ``energy data'', one obtains a unique (modulo constants) solution in this class.  The $DB$ method furnishes an energy solution for any  ``energy data'' (see \cite{R2} or \cite{AMcM} or \cite{AM}).  It has to be the same solution as the Lax-Milgram solution by uniqueness.  The boundary value problem is said to be solvable for the energy class if there is a constant $C$ such that any energy solution   satisfies the interior estimate $\|\tN(\nabla u)\|_{\mN} \le C \|data\|_{X}$   or $\|t\nabla u\|_{ \mT} \le C\|data\|_{\dot Y^{-1}}$  with  ``data'' being   the Neumann data  or the gradient of the Dirichlet data depending on the  problem. 
 In \cite{AS} it was shown for each of the four boundary value problems that solvability for the energy class implies existence of a solution for all data. We can improve this result as follows. 
We can  define a notion of \textbf{compatible well-posedness} (a terminology taken from \cite{BM}) which is that there is well-posedness of the boundary value problem \textbf{and} that  when the Dirichlet or Neumann data is also an ``energy data'', then the solution coincides with the energy solution obtained by Lax-Milgram.   This restriction is unavoidable. Examples from \cite{Ax, KR, KKPT} show that existence for all data does not imply solvability for the energy class. That is, there are examples with smooth Dirichlet  data where a solution exists and belongs to the appropriate class $\mN$ (for the regularity problem) but  the energy solution  does not belong to this class. Hence this is strictly stronger. It is clear that compatible well-posedness implies solvability for the energy class. The converse holds. 

\begin{thm}\label{thm:solvimplieswp} For each of the four boundary value problems with $q\in I_{L}$,  solvability for the energy class implies, hence is equivalent to,  compatible well-posedness. \end{thm}

In particular, combining this theorem with the extrapolation results  in \cite{AM} (which have extra hypotheses) gives extrapolation of compatible well-posedness under the assumptions there.

The Green's formula has a useful corollary as far as solvability vs uniqueness of the dual problem is concerned. This improves also the known duality principles of \cite{HKMP2} and \cite{AM}. 

\begin{thm}\label{thm:solvvsuniq} Let $n\ge 1$, $m\ge 1$. Let $q\in I_{L}$. Then
\begin{enumerate}
  \item If  $(R)_{X}^L$ is solvable for the energy class, then $(D)_{Y}^{L^*}$ is compatibly well-posed. 
    \item If $(D)_{Y}^{L^*}$ is solvable  for the energy class and $q\ge 1$, then $(R)_{X}^L$  is compatibly well-posed.  
     \item If  $(N)_{X}^L$ is solvable for the energy class, then $(N)_{Y}^{L^*}$ is compatibly well-posed. 
    \item  If $(N)_{Y}^{L^*}$ is solvable  for the energy class and $q\ge 1$, then $(N)_{X}^L$  is compatibly well-posed.
     \end{enumerate}
In the items (2) and (4),  uniqueness holds in the conclusion when $q<1$.  
\end{thm}

Note that we include the case $q=1$ in (2) and (4), while the corresponding directions is Theorem 
\ref{cor:dualityprinciple} are unclear. It has to do with the fact that  there always exists an energy solution with given  ``energy data''; we do not have to ``build'' it. 

It follows from  our results that any method of solvability will lead to optimal results in the range of exponents governed by $I_{L}$. Of course, what we do not address here is how to prove solvability.  We mention that \cite{AS} describes this range in a number of cases covering the case of systems with De Giorgi-Nash-Moser interior estimates (in fact, less is needed) for which $I_{L}= (1-\varepsilon, 2+\varepsilon')$. 
Another interesting situation is the case of boundary dimension $n=1$: then $I_{L}= (\frac{1}{2},\infty)$ so that we are able to describe all possible situations in this class of boundary value problems. Another situation is constant coefficients in arbitrary dimension:  $I_{L}= (\frac{n}{n+1},\infty)$. This will imply   compatible well-posedness of each  of the four boundary value problems in every possible class considered here using results in\cite{AAMc}. For general systems, $I_{L}$ is an open interval and contains $[\frac{2n}{n+2}, 2]$. 

In \cite{HKMP1}, it is proved, and this is a beautiful result, that for real equations $L^*u=0$, $L^*$-harmonic measure is $A_{\infty}$ of the Lebesgue measure on the boundary  and  that $(D)_{L^p}^{L^*}$ is solvable for the energy class (in the sense of our definition). In \cite{DKP}, it is shown (this is in bounded domains, but there is no difference, and for $t$-dependent coefficients as well) that this  $A_{\infty}$ condition implies $BMO$ solvability of the Dirichlet problem for the energy class, \textit{i.e.}, $(D)_{BMO}^{L^*}$. In \cite{AS}, it is shown that $(R)_{H^1}^{L}$ solvability for the energy class  implies $N_{\ta}$ is an isomorphism for neighboring spaces, which by our result above means well-posedness and we just assume solvability with our improvement. Thus well-posedness of the regularity problem can be shown for spaces near $H^1$. We have a little more precise result, completing the ones in \cite{HKMP2} for dimensions $1+n\ge 3$ and in \cite{KR,B} for dimension $1+n=2$. There is a result of this type in \cite{BM}.  

\begin{cor}\label{cor:real} Consider real equations of the form \eqref{eq:divform}.  The  regularity problem is compatibly well-posed   for $L^q$ for some range of $q>1$, $H^1$, and $H^q$ for some range of $q<1$.  This is the same for the Neumann problem when $n=1$. 
The Dirichlet problem is compatibly well-posed on $L^p$ for some $p<\infty$, $BMO$ and $\dot \Lambda^\alpha$ for some $\alpha>0$. In both situations, well-posedness remains for perturbations in $L^\infty$ of the matrix coefficients.\end{cor}

It is well-known that Neumann problems and regularity problems are the same when $1+n=2$ using conjugates, hence the result is valid for the Neumann problem in two dimensions.  Solvability of the Neumann problem seems to be harder in dimensions $1+n\ge 3$.

We mentioned the perturbation case for completeness of this statement but this is a mere consequence of the  results in \cite{AS}, Section 14.2, at this stage. Such results imply that well-posedness is stable under $L^\infty$ perturbation of the coefficients. But it is a question, even with our improvements here, whether compatible well-posedness is stable under $L^\infty$ perturbation of the coefficients. 

We note that this result, as any other one so far,  only takes care of solvability issues on the upper-half space.  Results on the lower half-space can be formulated using the negative spectral spaces $X^-_{DB}$, $\dot Y^{-1,-}_{DB}$. We leave them to the reader. In \cite{HKMP2}, the solvability on both half-spaces (this holds for real equations) is also used to show invertibility of the single layer potential in this range. The converse holds. This can be made a formal statement in our context for all four problems, extending the result in \cite{HKMP2} for the single layer potential. We shall explain the notation when doing the proof. 

\begin{thm}\label{thm:layerpot}  Fix $q\in I_{L}$. Let $X=L^q$ or $H^q$,  $\dot X^1$ be the homogeneous Sobolev space $\dot W^{1,q}$ or $\dot H^{1,q}$, $Y$ be the dual space of $X$ and $\dot Y^{-1}$ be the dual space of $\dot X^1$.
\begin{enumerate}
  \item  $(R)_{X}^L$ is well-posed on both   half-spaces if and only if  $\mS_{0}^{L}$ is invertible 
 from $X$ to $\dot X^1$.
  \item $(D)_{Y}^{L^*}$ is well-posed on both  half-spaces if and only if $\mS_{0}^{L^*}$ is invertible 
 from $\dot Y^{-1}$ to $Y$.
  \item  $(N)_{X}^L$ is well-posed on both   half-spaces  if and only if   $\partial_{\nu_A} \mD_{0}^{L}$ is invertible from $\dot X^1$ to $X$.
  \item $(N)_{\dot Y^{-1}}^{L^*}$ is well-posed on both   half-spaces if and only if    $\partial_{\nu_{A^*} }\mD_{0}^{L^*}$ is invertible from $Y$ to $\dot Y^{-1}$.
\end{enumerate}
  \end{thm}

We remark that the invertibility of the layer potentials $ \mD^{L}_{0\pm}$ on $L^p$ is sufficient  for the solvability of the Dirichlet problem $(D)_{L^p}^L$ (not $L^*$) on both half-spaces. This can also be seen from our method but it is a well-known fact. This invertiblity  is what is proved  on $L^2$ for the Laplace equation on special Lipschitz domains  in \cite{V}. More recently, \cite{AAAHK} establishes the same invertibility property on $L^2$  for real and symmetric equations \eqref{eq:divform}. It is however not clear whether it is necessary.

\

The main bulk of the paper is the proof of  the classification described in Theorems \ref{thm:main1} and \ref{thm:main2}. The other statements concerning solvability, although important of course in the theory,  are functional analytic consequences of our classification. We shall follow the setup introduced in \cite{AA1} which proves the case $p=2$ of our statements. This is, nevertheless, far more involved. 
First, there is an approximation procedure to pass from the weak formulation to a semigroup equation for conormal gradients which should look like
$$
\nabla\!_{A} u(t+\tau,\, .\,)= e^{-\tau|DB|} \nabla\!_{A} u(t,\, .\,),  \quad t>0,\tau\ge 0.
$$
 In classical situations, one  tests the equation against the adjoint fundamental solution and goes to the limit in appropriate sense. Here, we only have at our disposal an adjoint fundamental operator. Thus we argue more at operator level and do not use any integral kernel representation.  One difficulty is that we do not know how to interpret this identity at first. This is especially true when  $p>2$. Some weak version against appropriate test functions, which we will build,  is first proved. Then one constructs an element $f_{t}$ which satisfies this equation and has a trace in the appropriate  topology. Then the goal is to prove that $\nabla\!_{A} u(t,\, .\,)$ can be identified to such an  $f_{t}$. This uses behavior at $\infty$ in the statement to eliminate residual terms in some asymptotic expansion. 

The main difficulty here is that we have to reconcile two different calculi: the one on Schwartz distributions $\mS'$ and the one using the functional calculus of $DB$ via the Hardy spaces. It is thus crucial that these spaces, and any other one in the process, are embedded in $\mS'$ and this is the purpose of this interval $I_{L}$ that we carry all the way through. The article  \cite{AS} will be most important for that and one cannot read this article without having the other one at hand, as one finds a lot of estimates used along the way. Outside this interval of exponents, our arguments collapse.

A word on history to finish. The recent papers contain historical background as concerns solvability methods (\textit{e.g.}, \cite{AAAHK, HKMP2, AM, HMiMo}). We isolate a few points and refer to those articles for more complete quotes.  We feel that the starting point of the study of such boundary value problems with non-smooth coefficients  is the breakthrough work of Dahlberg on the Dirichlet problem for the Laplace equation  on Lipschitz domains \cite{Da}, then followed by  the ones  of Dahlberg, Fabes, Jerison, Kenig, Pipher, etc., setting the theory in the classical context of real symmetric equations: see the book by Carlos Kenig \cite{Ke} and references therein for  a good overview of the techniques based on potential theory and Green's function. Layer potential techniques based on fundamental  solutions (extending the ones known and used (see \cite{Br,V}) for the Laplace equation on Lipschitz domains) were developed in the context of real equations and their perturbations under the impetus of Steve Hofmann (see \cite{HK, AAAHK, HMiMo, BM}).  The solution of the Kato problem for second order operators and systems \cite{CMcM, HMc, AHLMcT, AHMcT} and its extension in \cite{AKMc} to $DB$ operators brought a wealth of new estimates and techniques. The estimates for square roots can be used directly in the second order setup (see \cite{HKMP1, HKMP2, HMiMo}). The article \cite{AAH} in a first step and, more importantly, \cite{AAMc} made the explicit link between this class of second order equations and the first order Cauchy-Riemann type  system \eqref{eq:rep} for boundary value problems with $L^2$ data.  This and  the Hardy space theory  were exploited to build  solutions with $L^p$ data of \eqref{eq:divform}  in \cite{AS}: this also gave new results for the boundary layer potentials. The work \cite{AA1} is the first one where a converse on representation for conormal gradients is proved for $p=2$.  As for uniqueness results concerning the boundary value problems,  one can find many statements in the literature in our context (\cite{DaK, Br, AAAHK, KS, HMiMo, BM}...) but always assuming the De Giorgi-Nash-Moser regularity for solutions. Ours are without such an assumption. Moreover, had we even assumed such regularity hypotheses, our results improve the existing ones for $t$-independent equations by being less greedy on assumptions: no superfluous \textit{a priori} assumption is taken and they apply to each boundary value problem individually, and by proposing  possible representation and uniqueness for $p\le 1$ for regularity and Neumann problems.

Our results are bi-Lipschitz invariant in standard fashion. We do not insist, but this does cover the case of the Laplace equation above Lipschitz graphs. It is likely that a similar theory can be developed on the unit ball as in \cite{KP}, adapting the setup of \cite{AR}. This deserves some adaptation though. 

We do not address here the perturbation of $t$-independent operators by some $t$-dependent ones and leave the corresponding issues as the ones in this article open for $p\ne 2$; for $p=2$, they are treated in \cite{AA1}, and recently  \cite{HMaMo} addresses the method of layer potentials in this context under De Giorgi-Nash-Moser assumptions.

\

Acknowledgements:  This article would not   exist without the earlier collaborations of the first author  with  Alan McIntosh and Andreas R\'osen on functional calculus and first order theory, and  of both authors with Steve Hofmann on second order PDE's. It is a pleasure to acknowledge their deep influence.   The  authors were partially supported by the ANR project ``Harmonic analysis at its boundaries'' ANR-12-BS01-0013-01.  Finally, we thank Alex Amenta for a careful reading leading to an improved second version.

\section{Technical lemmas in tent spaces}\label{sec:tent}

For $0<q\le \infty$, $T^{q}_{2}$ is the  tent space   of \cite{CMS}. For $0<q<\infty$, this is the space of $L^2_{\loc}(\reu)$ functions  $F$ such that 
 $$
 \|F\|_{T^q_{2}} = \|\SF F\|_{q} <\infty
 $$
 with for all $x\in \R^n$,
 \begin{equation}
\label{eq:sfdef}
 (\SF F)(x): = \left( \iint_{\Gamma_{a}(x)}  |F(t,y)|^2\, \frac{dtdy}{t^{n+1}}\right)^{1/2}, 
\end{equation}
 where $a>0$ is a fixed number called the aperture of the cone $\Gamma_{a}(x)= \{(t,y); t>0, |x-y|<at\} $. Two different apertures give equivalent $T^q_{2}$ norms.

  For $q=\infty$, $T^\infty_{2}$ is defined via Carleson measures  by  $\|F\|_{T^\infty_{2}}<\infty$, where $\|F\|_{T^\infty_{2}}$ is the smallest positive constant $C$ in 
$$
\iint_{T_{x,r}}   |F(t,y)|^2\, \frac{dtdy}{t} \le C^2 |B(x,r)|
$$
for all   open balls $B(x,r)$  in $\R^n$ and $T_{x,r}=(0,r)\times B(x,r)$. 
For $0<\alpha<\infty$,   $T^\infty_{2,\alpha}$ is defined by $\|F\|_{T^\infty_{2,\alpha}} <\infty$, where 
$\|F\|_{T^\infty_{2,\alpha}}$ is the smallest positive constant $C$ in 
$$
\iint_{T_{x,r}}   |F(t,y)|^2\, \frac{dtdy}{t} \le C^2 r^{2\alpha} |B(x,r)|
$$
for all   open balls $B(x,r)$  in $\R^n$. For convenience, we set $T^\infty_{2,0}=T^\infty_{2}$. Introduce also 
the Carleson function $C_{\alpha}F$ by
$$
C_{\alpha}F(x):= \sup \bigg(\frac {1}{r^{2\alpha}  |B(y,r)|}\iint_{T_{y,r}}   |F(t,z)|^2\, \frac{dtdz}{t}\bigg)^{1/2},
$$
taken over all open balls $B(y,r)$ containing $x$, 
so that $\|F\|_{T^\infty_{2,\alpha}}= \|C_{\alpha}F\|_{\infty}$. 

For $1\le q<\infty$ and $p$ the conjugate exponent to $q$,  $T^{p}_{2}$ is the dual of $ T^q_{2}$  for the duality
$$
(F,G):=\iint_{\reu} F(t,y) \overline{G(t,y)} \, \frac{dtdy}{t}.
$$
For $0<q\le 1$ and $\alpha= n (\frac{1}{q}-1)$, $T^\infty_{2,\alpha}$ is the dual of $T^q_{2}$ for the same duality form. Although not done explicitly there, it suffices to adapt the proof of \cite[Theorem 1]{CMS}.

For $0<p<\infty$, we also introduce the space $N^p_{2}$ as the space of  $L^2_{loc}(\reu)$ functions such that $\tN F\in L^p(\R^n)$, where
 \begin{equation}
\label{eq:KP}
 \tN F(x):= \sup_{t>0}  \bigg(\bariint_{W(t,x)} |F(s,y)|^2\, {dsdy}\bigg)^{1/2}, \qquad x\in \R^n,
\end{equation}
with \begin{equation}
\label{eq:whitney}
W(t,x):= (c_0^{-1}t,c_0t)\times B(x,c_1t),
\end{equation} for some fixed parameters $c_0>1$, $c_{1}>0$. Changing the parameters yields equivalent $N^p_{2}$ norms. 

\begin{lem}\label{lem:trunc} Let $0<p<\infty$. If $F\in T^p_{2}$ and $0<a<b<Ka$ with fixed $K>1$.  Then $F1_{a<t<b} \in N^p_{2}$ uniformly in $a,b$. 
If $F\in L^\infty(0,\infty; L^2)$ and $0<a<b<\infty$, then $F1_{a<t<b} \in T^\infty_{2,\alpha}$ for all $\alpha\ge 0$. 
\end{lem}

These estimates are trivially verified.  
\begin{lem}
\label{lem:HMiMo} If $0<p<r\le 2$ and $F\in N^p_{2}$, then  
$$
\bigg(\iint_{\reu} |F(t,x)|^r\,  t^{n(\frac{r}{p}-1)}\, \frac{dtdx}{t}\bigg)^{1/r} \lesssim \|\tN F\|_{p}.
$$
\end{lem}

\begin{proof} This statement for $1<r=\frac{p(n+1)}{n}\le 2$ is explicitly in \cite{HMiMo}. 
For the other cases,    as $r\le 2$, we easily obtain
$$
\iint_{\reu} |F(t,x)|^r \, t^{n(\frac{r}{p}-1)}\, \frac{dtdx}{t} \lesssim \iint_{\reu}\bigg( \bariint_{\wt W(s,y)}  |F|^2\bigg)^{r/2}  \, s^{n(\frac{r}{p}-1)}\, \frac{dsdy}{s},
$$
where $\wt W(s,y)$ is some slightly smaller Whitney region contained in $W(s,y)$.  We can apply the inequalities of  \cite{CT}, Theorem 2.6,   for the pointwise non-tangential maximal function, adjusting the aperture of the cone at vertex $x$ containing $(s,y)$ so that the pointwise non-tangential maximal function of the expression within parentheses is controlled by $\tN F(x)$.
\end{proof}

\begin{lem}\label{lem:p}If $1<p\le 2$ and $F\in N^p_{2}$ then $F\in L^p_{loc}(0,\infty; L^p)$ and 
$$
\sup_{a>0 }\barint_{[a,2a]} \|F_{t}\|_{p}^p\, dt \lesssim \|\tN F\|_{p}^p.
$$
\end{lem}

\begin{proof} Do as above with $r=p$ and use   that the $t$-integral is between $a$ and $2a$. 
\end{proof}

\begin{prop}\label{prop:approx} Let $0<p<\infty$. Suppose $0<a<b<\infty$ and $F\in T^p_{2}$ with support in $[a,b]\times \R^n$. Let $\rho_{k}$ be a standard mollifier in $\R^n$: $\rho\in C^\infty(\R^n; [0,1])$, $\supp\rho\in B(0,1)$, $\int \rho=1$  and $\rho_{k}(x)= k^n\rho (kx)$ for $k\ge 1$. Then, $F_{k}(t,x)= F_{t}\star_{\R^n} \rho_{k}(x)$ belongs to $T^p_{2}$ uniformly in $k$ and converges to $F$ in $T^p_{2}$. 
\end{prop}

\begin{proof} Let  $\|S_{a}F\|_{p}$ be  the $T^p_{2}$ norm on cones of aperture a. Let $x\in \R^n$ and $(t,y)\in \reu$ with $|x-y|<t$. Using $|F_{t}\star_{\R^n} \rho_{k}(y)|^2 \le |F_{t}|^2\star_{\R^n} \rho_{k}(y)$, the supports of $F$ and $\rho$, and Fubini's theorem,  we have
\begin{align*}
  \iint_{\substack{a \le t\le b\\ |x-y|<t}}  |F_{t}\star_{\R^n} \rho_{k}(y)|^2\, \frac{dtdy}{t^{n+1}}  &  \le 
   \iint_{\substack{a \le t\le b\\ |x-z|<t+1/k}}  |F(t,z)|^2 \int \rho_{k}(y-z)\, dy \, \frac{dtdz}{t^{n+1}}   \\
    & \le  \iint_{\substack{a \le t\le b\\ |x-z|<t(1+1/a)}}  |F(t,z)|^2 \ \, \frac{dtdz}{t^{n+1}}.
\end{align*}
Thus $S_{1}F_{k}\le S_{1+1/a}F$ for all $k\ge 1$. Also clearly, $S_{1}F_{k}(x)$ converges to $S_{1}F(x)$, because $F_{t}\star_{\R^n} \rho_{k}(y)$ converges to $F(t,y)$ in $L^2$ of any compact set in $\reu$. The conclusion follows by dominated convergence.
\end{proof}

\section{Slice-spaces}\label{sec:slice}

We now introduce spaces adapted to $T^p_{2}$ and $T^\infty_{2,\alpha}$. We shall call them slice-spaces while,  when they are Banach spaces, they are particular cases of Wiener amalgams (see for example the survey  \cite{Fei} and also \cite{Hei}) first introduced by Wiener and further developed in relation with time-frequency analysis and sampling theory. Our terminology comes from the heuristic image of slicing   the tent space norm at a fixed height. This relation to easily cover the quasi-Banach case that we need later on. 

For $p\in (0, \infty]$ and $t>0$, the slice-space $E^p_{t}$ is  the subspace of $L^2_{loc}(\R^n)$ functions $f$  with 
$$ 
\|f\|_{E^p_{t}}= \left( \int_{\R^n} \bigg(\barint_{B(x,t)} |f(y)|^2\, dy\bigg)^{p/2}\, dx\right)^{1/p}<\infty,
$$
with the usual modification taking the $\esssup$ norm when $p=\infty$ (as averages on balls are continuous with respect to the center, one can take $\sup$).  For  $p<1$, this is  only a quasi-normed space. 
Also, for $\alpha \in [0,1)$ and $t>0$, the slice-space $E^{\infty,\alpha}_t$ is the subspace of $L^2_{loc}(\R^n)$ functions $g$ such that
$$ 
\|g\|_{E^{\infty,\alpha}_t}=  \mathop{\sup_{x\in \R^n, r \ge t}} \frac{1}{r^\alpha} \bigg( \barint_{B(x,r)} |g(y)|^2\,dy \bigg)^{1/2} <\infty.
$$
 It was pointed out to us by A. Amenta that standard coverings  of balls with radii $r$ by roughly $(r/t)^n$ balls with radii $t$ when $r\ge t$ show that  
 $$ 
\|g\|_{E^{\infty,\alpha}_t}\sim  \mathop{\sup_{x\in \R^n}} \frac{1}{t^\alpha} \bigg( \barint_{B(x,t)} |g(y)|^2\,dy \bigg)^{1/2}.
$$
Hence, we see that $E^{\infty,\alpha}_t= E^{\infty}_t$ with equivalent norms $\|g\|_{E^{\infty,\alpha}_t} \sim t^{-\alpha}\|g\|_{E^{\infty}_t}$.
Averaging   
\begin{equation}
\label{eq:averaging}
 \int_{\R^n} g(y)\, dy= \int_{\R^n}  \barint_{B(x,t)} g(y)\, dy\ dx
\end{equation}
and using H\"{o}lder's inequality with exponent $2/p$, we obtain that $E^p_{t} \subset L^p(\R^n)$ for $p \in(0, 2]$,
with 
$$\|f\|_{p} \le \|f\|_{E^p_{t}}.
$$
Similarly,  using H\"{o}lder's inequality with exponent $p/2$ and then \eqref{eq:averaging},   $L^p(\R^n) \subset E^p_{t}$ for $p \in [2, \infty)$, with 
 $$ \|f\|_{E^p_{t}}\le \|f\|_{p}. $$
Note that the  constant is one, thus uniform in $t$, and $E^2_t=L^2(\R^n)$ for any $t>0$, isometrically.  

We have trivial embedding and projection that allow us to carry the properties of  tent spaces to the slice-spaces by retraction. Let $f:\R^n\to \C$ and set 
$$
\iota(f)(s,x)=f(x) 1_{[t,et]}(s).
$$
For $G:\reu \to \C$ we set 
$$
\pi(G)(x)= \int_{t}^{et} G(s,x)\, \frac{ds}{s}.
$$ 
Clearly, for suitable $f$ and $G$, 
\begin{equation}\label{slice:tent-embed}
\iint_{\reu} \iota(f)(s,x) G(s,x) \, \frac{dsdx}{s}= \int_{\R^n} f(x) \pi(G)(x)\, dx,
\end{equation}
that is, $\pi$ is the dual of $\iota$ in some sense, and 
\begin{equation}\label{slice:retract}
\pi \circ \iota(f) = f.
\end{equation}
Let $t>0$. Then for $p\in (0,\infty)$ and $\alpha \in [0,1)$, it is easy to see that $\iota: E^p_{t}/E^{\infty,\alpha}_t \to T^p_{2}/T^\infty_{2,\alpha} $ is bounded and the norm is uniform in $t$, where $X/Y$ means $X$ or $Y$ respectively. Next, $\pi: T^p_{2}/T^\infty_{2,\alpha} \to E^p_{t}/E^{\infty,\alpha}_t$ is bounded also with uniform bound in $t$ with the same norm on $T^p_{2}$. This yields that $E^p_{t}$ and $E^{\infty,\alpha}_t$ are retracts of $T^p_{2}$ and $T^\infty_{2,\alpha}$ respectively. 
\begin{rem} For $E^p_t$ we may use instead rescaled $\iota$ and $\pi$ (consider $s\in [1, e]$) and adapt the aperture of the cone in the norm on $T^p_{2}$ (consider cones of aperture $t$). Both methods are equivalent.
\end{rem}
In the following lemmas we summarize the consequences of the retraction property on slice-spaces.

\begin{lem}[Duality]\label{lem:slice_duality} Fix $t>0$. In the pairing $ \int_{\R^n} f(x) g(x)\, dx$, 
\begin{itemize} 
\item[1)] for $p \in (1, \infty)$, $E^{p'}_t$ is the dual space of $E^p_t$, where $\frac{1}{p}+\frac{1}{p'}=1$.
\item[2)] for $p \in (0,1]$, $E^{\infty,\alpha}_t$ is the dual space of $E^p_t$, where $\alpha= n(\frac{1}{p}-1)$.
\end{itemize}
\end{lem}

\begin{proof}1) We first show that $E^{p'}_t \subseteq (E^p_t)'$. If $g \in E^{p'}_t$, then $f \mapsto \int_{\R^n} f g$ induces a bounded linear functional on $E^p_t$. Indeed, using \eqref{slice:retract}, we have
$$\int_{\R^n}| f(x) g(x)|\, dx  = \iint_{\R^{n+1}_+} |\iota(f)(s,x) \iota(g)(s,x)| \frac{dsdx}{s},$$
for all $f \in E^p_t$. Therefore, by tent space duality, there holds that
$$\Big|\int_{\R^n} f(x) g(x)\, dx \Big| \lesssim \|\iota(f)\|_{T^p_2} \|\iota(g)\|_{T^{p'}_2} \lesssim \|f\|_{E^p_{t}} \|g\|_{E^{p'}_t},$$
with the implicit constants  uniform in $t$. (Note that one could obtain this inequality with constant 1 by applying \eqref{eq:averaging}.)  
 
We now prove the converse inclusion, \textit{i.e.}, $ (E^p_t)'\subseteq E^{p'}_t$. Suppose that $\ell$ is a bounded linear functional on $E^p_t$. Then $\tilde \ell=\ell \circ \pi$ is a bounded linear functional on $T^p_2$ and by tent space duality, there exists $G \in T^{p'}_2$ so that 
$$ 
\tilde \ell(F)=\iint_{\R^{n+1}_+} F(x,s) G(x,s) \frac{dx ds}{s},
$$
for all $F \in T^p_2$. Then, if we set $F= \iota(f)$, one can easily see that $\ell(f)=   \int_{\R^n} f(x) g(x)\, dx$ with  $g=\pi (G)$, which is an element of $E^{p'}_t$.

2) The proof follows from a simple modification of 1) using that the dual of $T^p_{2}$ is $T^{\infty,\alpha}_{2}$ and we omit it.
\end{proof}

At this point we introduce the notion of $E^p_t$-atoms and $E^p_t$-molecules, and then we prove that any function in $E^p_t$ has an atomic decomposition if $p\le 1$.
\begin{defn}
We fix $t>0$ and we let $p \in (0, 1]$. A function $a \in E^p_t$ is said to be an $E^p_t$-atom if it is supported in a ball $B_r$ of radius $r \ge t$ and satisfies 
\begin{equation}\label{eq:atomic estimate}
\| a \|_{L^2(\R^n)} \le r^{n(1/2-1/p)}.
\end{equation}

A function $m\in E^p_t$ is said to be an $E^p_t$-molecule adapted to a ball $B_r$ of radius $r \ge t$ if it satisfies 
\begin{equation}\label{eq:molecule_local estimate}
\| m \|_{L^2(8B_r)} \le r^{n(1/2-1/p)},
\end{equation}
and if there exists $\epsilon>0$ such that
\begin{equation}\label{eq:molecule_annular estimate}
\| m \|_{L^2(2^{k+1} B_r \setminus 2^k B_r)} \le 2^{-\epsilon k} (2^k r)^{n(1/2-1/p)}, \quad k \ge 3.
\end{equation}
\end{defn}

\begin{lem}[Atomic decomposition]\label{lem:slice_atom.decomp.} Let $p \in (0, 1]$ and $f \in  E^p_t$. Then there exist a sequence of numbers $\{\lambda_j\}_{j \ge 1} \subset \ell^p$ and a sequence of $ E^p_t$-atoms $\{a_j\}_{j \ge 1}$ so that $f=\sum_j \lambda_j a_j$, with convergence in $E^p_t$ and $\|f\|_{E^p_{t}}\lesssim \|\{\lambda_j\}\|_{\ell^p}$.  If $f \in  E^p_t \cap E^2_t$ then it converges also in  $E^2_t$. Conversely, given  $\{\lambda_j\}_{j \ge 1} \subset \ell^p$ and  $ E^p_t$-atoms $\{a_j\}_{j \ge 1}$, the  series $\sum_j \lambda_j a_j$  converges in $E^p_{t}$ and  defines a function in $E^p_{t}$ with norm controlled by  $\|\{\lambda_j\}\|_{\ell^p}$. 
Any $E^p_t$-molecule belongs to $E^p_{t}$ with uniform norm. 
\end{lem}
\begin{proof}
 Suppose that $f \in E^p_t$. Then $\iota(f)$ has an atomic decomposition in $T^p_2$, that is, $\iota(f)= \sum_j \lambda_j A_j$. Each $A_j$ is supported in a tent region $\widehat B_j=\{(t,y)\, ; \, B(y, t)\subset B_{j}\}$ and satisfies the bound $\iint_{\widehat B_j } |A_j|^2dxdt/t \le |B_j|^{1-2/p}$. But the support of $\iota(f)$ is contained in the strip $[t, et] \times \R^n$ and thus, if $r_j$ is the radius of the ball $B_j$, we necessarily have that $r_j \ge t$. We now set $a_j=\pi(A_j)$, which  is an $E^p_t$-atom as one easily shows. The convergence of $\iota(f)= \sum_j \lambda_j A_j$ is both in $T^p_{2}$ by the atomic decomposition and also $T^2_{2}$ using the support of $\iota(f)$. The convergence of $f=\sum_{j}\lambda_{j}a_{j}$ in $E^p_t \cap E^2_t$   follows from the boundedness of $\pi : T^q_2 \to E^q_t$ for all $q \in (0, \infty)$.
 
 The converse is similar. As $\iota(a_{j})$ is a $T^p_{2}$ atom,  $F= \sum_j \lambda_j \iota(a_j)$ converges in  $T^p_{2}$ and $\pi(F)= \sum_j \lambda_j a_j$ is an element on $E^p_{t}$.

An $E^p_t$-molecule has an atomic decomposition using the annuli of its definition. Thus, it belongs to $E^p_{t}$ with uniform norm.  
\end{proof}

The next lemma shows that if a function is in $E^p_t$ for some $t>0$, then it belongs to $E^p_s$ for all $s>0$. 

\begin{lem}[Change of norms]\label{lem:slice_changenorms} If $0<s, t<\infty$ and $p \in (0, \infty)$,  then
$E^p_{t}=E^p_{s}$ with 
\begin{equation}
\label{eq:equivept}
 \min (1, (t/s)^{n/2-n/p})  \| f\|_{E^p_{t}} \lesssim_{n,p} \| f \|_{E^p_{s}} \lesssim_{n,p}   \max (1, (t/s)^{n/2-n/p})  \| f\|_{E^p_{t}} .
\end{equation} 
\end{lem}
\begin{proof}
To see this we use that $E^{p}_{t}$ norms are comparable to $T^p_{2}$ norms with aperture $t$ and use the precise comparison of $T^p_{2}$ norms  under change of aperture obtained in \cite{AuCR} (some of these bounds were already in  Torchinsky's book \cite{Tor}). 
\end{proof}

We define $M_b$ to be the operator of multiplication with a function $b \in L^\infty(\R^n)$ and $\mathcal{C}_\phi$ to be the convolution operator with an integrable function $\phi$ with bounded support. We now show some stability, density and embedding properties of slice-spaces.

\begin{lem} Fix $t>0$ and $ p \in (0, \infty]$. Then the following hold.
\begin{enumerate}
  \item[1)]  $M_b: E^p_{t} \to E^p_{t}$.
  \item[2)] $\mathcal{C}_\phi: E^p_{t} \to E^p_{t}$.
  \item[3)] $\mD(\R^n)$ is dense in $E^p_{t}$ when $p<\infty$. 
  \item[4)] $E^p_{t}$  embeds in the space of Schwartz distributions $\mathcal{S}'$.
\end{enumerate}
\end{lem}
\begin{proof}
The first point  is obvious.   The second point follows easily using Lemma \ref{lem:slice_changenorms}: we may assume that $\phi$ is supported in the ball $B(0,R)$ and set $t=R$. In this case, we have
$$
\int_{B(x,R)} |\phi \star f|^2 \le \int_{\R^n} |\phi \star 1_{B(x,2R)}f|^2  \le  \|\phi\|_{1}^2  \int_{B(x,2R)} | f|^2$$
and it suffices to integrate $p/2$ powers or to take the sup norm when $p=\infty$. 
 To prove 3), we use the usual truncation and mollification arguments along with 1) and 2). To show 4), we write $\int_{\R^n} f(x) \varphi(x)\, dx= (\iota(f), \iota(\varphi))$ and observe that $\iota(\varphi)$ belongs to the dual of  $T^p_{2}$ when $\varphi\in \mS(\R^n)$.
\end{proof}

\begin{rem}
Since $\mD(\R^n) \subset L^2(\R^n)$, 3) yields that $E^p_t \cap E^2_t$ is dense in $E^p_t$ when $p<\infty$.
\end{rem}

In the following lemma the derivatives are taken in the $W^{1,2}_{loc}$ sense. We shall need to use several times this unconventional integration by parts. 

\begin{lem}[Integration by parts in slice-spaces]\label{lem:sliceIBP} Let $p \in (0,\infty)$ and suppose that $\partial$ is a first order differential operator with constant coefficients and $\partial^*$ is its adjoint operator. If $f, \partial f\in E^p_{t}$ and $g, \partial^*g \in (E^{p}_{t})'$, then $\int_{\R^n} \partial f(x) \overline{g(x)}\, dx= \int_{\R^n}  f(x) \overline{\partial^*g(x)}\, dx$. 
\end{lem}

\begin{proof}
Take a smooth cut-off function $\chi_{R}$ which is 1 on $B(0,R)$, $0$ outside $B(0,2R)$ and $\|\nabla\chi_{R}\|_{\infty}\le CR^{-1}$. Then using integration by parts for $W^{1,2}$ functions with bounded support
$$
\int_{\R^n} \chi_{R}(x)\partial f(x) \overline{g(x)}\, dx= \int_{\R^n} \chi_{R}(x) f(x) \overline{\partial^*g(x)}\, dx - \int_{\R^n} \partial \chi_{R}(x) f(x) \overline{g(x)}\, dx. 
$$
It remains to let $R\to \infty$ by using dominated convergence for each integral. 
\end{proof}

Finally, we state that $E^p_{t}$ are real and complex interpolation spaces. 

\begin{lem}[Interpolation]\label{lem:slice-intepolation} For fixed $t>0$, suppose that $0 < p_0 < p < p_1 \leq \infty$ and $1/p=(1-\theta)/p_0 + \theta/p_1$. Then
\begin{enumerate}
  \item $[E^{p_0}_t,E^{p_1}_t]_{\theta} =E^p_t$ (complex method) with equivalent norms uniformly with respect to $t$.
  \item $(E^{p_0}_t,E^{p_1}_t)_{\theta, q} =E^p_t$, $q=p$ (real method) with equivalent norms uniformly with respect to $t$.
\end{enumerate}
\end{lem}
\begin{proof}
This follows from the fact that $E^p_{t}$ is a retract of $T^p_2$ and the results in \cite{CMS}, completed and corrected in \cite{Be}.
\end{proof}

All of this extends to $\C^N$-valued functions. 

\section{Operators with off-diagonal decay on slice-spaces}

In this section, we investigate the boundedness on slice-spaces of operators with $L^2$ off-diagonal decay and prove some convergence results in those spaces.
\begin{defn}
A family of  operators $(T_s)_{s>0}$ is said to have $L^2$ off-diagonal decay of order $N \in \N$ if there exists $C_N>0$ such that
\begin{equation}\label{eq:off-diag.decay}
\| 1_E T_s f\|_2 \leq C_N \langle \dist(E,F)/s\rangle^{-N} \|f\|_2,
\end{equation}
for all $s > 0$, whenever $E,F \subset \R^n$ are closed sets and $f \in L^2$ with $\supp f \subset F$. We have set $\langle x \rangle:=1+|x|$ and $\dist(E,F):= \inf\{ |x-y| : x \in E, y \in F \}$.
\end{defn}

\begin{prop}\label{prop:slice-Ts bound} Let $p \in (0,\infty]$. If $(T_s)_{s>0}$ is a family of linear operators with $L^2$ off-diagonal decay of order $N>\inf( n|1/p-1/2|, n/2)$, then 
$$
T_s: E^p_t \to E^p_t,  \quad  \text{uniformly in}\,\, 0< s \le t .
$$
\end{prop}

\begin{proof}
\textbf{Case} $0< p \le 1$. To prove boundedness in $E^p_t$ we claim that it suffices to prove that $T_s$ maps $E^p_t$-atoms to $E^p_t$-molecules. Indeed, let $f \in E^p_t \cap E^2_t$. Then, by Lemma \ref{lem:slice_atom.decomp.}, $f$ has an atomic decomposition with convergence in both $E^p_t$ and $E^2_t$. Since $T_s: E^2_t \to E^2_t$ and $f=\sum_j \lambda_j a_j$ in $E^2_t$, we have $T_{s}f=\sum_j \lambda_j T_{s}a_j$  in $E^2_t$, hence $|(T_s f)(x)| \leq \sum_j |\lambda_j| |(T_s a_j)(x)|$ for almost every $x \in \R^n$. Taking $E^p_t$ quasi-norms in both sides and using the claim we obtain that $T_s: E^p_t \cap E^2_t \to E^p_t$. Since $E^p_t \cap E^2_t $ is dense in $E^p_t$, we extend $T_s$ by continuity to a bounded operator $\tilde T_s : E^p_t \to E^p_t$, with bounds uniform in $0<s \le t$. 

We shall now prove our claim. Suppose that $a$ is an $E^p_t$-atom supported in a ball $B_r$ of radius $r \ge t$. Let also $C_k(B_r)=2^{k+1}B_r \setminus 2^kB_r$ and apply \eqref{eq:off-diag.decay} with $E=C_k(B_r)$, $F=B_r$ and $f=a$. Then we have that
\begin{align}
\| T_s a\|_{L^2(C_k)} &\lesssim \frac{s^N}{(2^k r)^N} \|a\|_{L^2} \lesssim \frac{r^N}{(2^k r)^N} r^{n(1/2-1/p)} \notag\\
&= 2^{-k(N-n(1/p-1/2))} (2^k r)^{n(1/2-1/p)}, \label{eq:Ts-molecule-annular}
\end{align}
where in the second inequality we used that $s \le t \le r$ and \eqref{eq:atomic estimate}. The fact that $\| T_s a\|_{L^2(8 B_r)} \lesssim r^{n(1/2-1/p)}$ is immediate from \eqref{eq:off-diag.decay} and \eqref{eq:atomic estimate}, which in conjunction with \eqref{eq:Ts-molecule-annular} implies that $T_s a$ is a $E^p_t$-molecule.

\textbf{Case} $p =\infty$. As $ E^{\infty}_t=  E^{\infty, 0}_t$ with equivalent norms uniformly in $t$, it is enough to prove $T_{s}: E^{\infty, 0}_t \to E^{\infty, 0}_t$. Suppose that $g \in E^{\infty,0}_t$ and let us fix an arbitrary ball $B_r$ of radius $r \ge t$ that contains $x$. We decompose $g$ so that $g:=g_0+\sum_k g_k = g 1_{8B_r} + \sum_k g 1_{C_k(B_{r})}$, where $C_k(B_r)=2^{k+1}B_r \setminus 2^k B_r$, $k \ge 3$. We utilize the $L^2$ boundedness of $T_s$ (coming from \eqref{eq:off-diag.decay}) to get
\begin{equation}\label{eq:Ts-local-infinity}
 \Big( \barint_{B_r} |T_s g_0|^2 \Big)^{1/2} \lesssim  \Big( \barint_{8 B_r} |g|^2\Big)^{1/2} \le \|g\|_{E^{\infty, 0}_t}
\end{equation}
In view of \eqref{eq:off-diag.decay} and $s\le t$, we have that
\begin{equation}\label{eq:Ts-annular-infinity}
 \Big( \barint_{B_r} |T_s g_k|^2 \Big)^{1/2} \lesssim 2^{-k(N-n/2)}  \Big( \barint_{2^{k+1}B_r} |g|^2\Big)^{1/2} \le 2^{-k(N-n/2)} \|g\|_{E^{\infty, 0}_t}
.
\end{equation}
 Combining \eqref{eq:Ts-local-infinity} and \eqref{eq:Ts-annular-infinity} and $N>n/2$,  we obtain that $T_sg \in E^{\infty}_t$, uniformly in $0<s\le t$.\\

\textbf{Case} $1<p<\infty$. In light of Lemma \ref{lem:slice-intepolation}, we conclude by interpolation that $T_s$ extends to a bounded operator from $E^p_t$ to $E^p_t$ for all $1<p < \infty$. Indeed, we assume off-diagonal decay of order $N>n/2$, which allows us to apply the  $p=1$ and $p=\infty$ cases. 
\end{proof}

\begin{rem}
One can probably obtain a sharper lower bound in $N$ for a given fixed $p$, but this suffices for our needs. Also notice that for all $s\ge t$, one obtains that $T_{s}: E^p_{t}\to E^p_{t}$ but the norm may not be uniform any longer. This comes from the fact that $E^p_{s}=E^p_{t}$ with equivalent norms. 
\end{rem}

\begin{prop}\label{prop:cv}  Fix $t>0$. 
Let $I$ be the identity operator and $(T_s)_{s>0}$ be a family of linear operators which has $L^2$ off-diagonal decay of order $N>\inf (n/p,n/2)$.  If $T_s \to I$ strongly in $L^2$ for $s\to 0$, then for $0<p<\infty$, $T_s \to I$ strongly in $E^p_{t}$ for $s\to 0$.

\end{prop}
\begin{proof} As we have a uniform estimate for $0<s<t$ from Proposition \ref{prop:slice-Ts bound}, it suffices to check this for $f$ being in a dense class. Thus, take $f\in L^2$ supported in a ball and we may assume without loss of generality that it has  radius $R\ge 2t$. Let $C_{j}(B_{R})= 2^{j+1}B_{R}\setminus 2^jB_{R}$ for $j\ge 1$.
Using the support condition of $f$ and \eqref{eq:off-diag.decay}, we have
\begin{align*}
 \int_{(2B_R)^c} \left( \barint_{B(x,t)} | T_sf- f|^2 \right)^{p/2} \, dx &= \int_{(2B_R)^c} \left( \barint_{B(x,t)} | T_s f|^2 \right)^{p/2}\, dx \\
&=   \sum_{j\ge 1} \int_{C_j(B_R)} \left( \barint_{B(x,t)} |T_s f|^2 \right)^{p/2} \, dx \\
& \lesssim  \sum_{j\ge 1} t^{-np/2} (2^jR)^{n} \left( \frac{s}{2^jR}\right)^{Np} \|f\|_{2}^p
&\lesssim_{t,R,f} s^{Np}
\end{align*}
using $N>n/p$. 
Next, we have
$$
\int_{2B_R} \left( \barint_{B(x,t)} | T_sf- f|^2 \right)^{p/2} \, dx \lesssim   t^{-np/2} R^N \|T_sf-f\|_{2}^p.
$$
Using the strong convergence of $T_{s}\to I$ in $L^2$, this proves that $\|T_sf-f\|_{E^p_{t}}\to 0$ as $s\to 0$. 
\end{proof}

\section{Some properties of weak solutions}

Throughout, we assume without mention that the coefficients of \eqref{eq:divform} satisfy the ellipticity conditions \eqref{eq:boundedmatrix} and \eqref{eq:accrassumption}.

\begin{lem}[Caccioppoli] Any  weak solution $u$ of \eqref{eq:divform}  in a ball $B=B(\bx,r)$ with $ B \subset \reu$  enjoys the Caccioppoli inequality for any $0<\alpha<\beta<1$  and some $C$ depending on the ellipticity constants, $n,m$,  $\alpha$ and $\beta$,  
 \begin{equation}
\label{eq:caccio}
{}\iint_{\alpha B} |\nabla u|^2 \le Cr^{-2} {}\iint_{\beta B}  |u|^2.
\end{equation}
\end{lem}

This is well-known. Note that the ball $B$ can be taken with respect to any norm in $\R^{1+n}$. Thus $\alpha B$ can be a Whitney region $W(t,x)$ as in \eqref{eq:whitney} and $\beta B$ is a slightly enlarged region of the same type. 

\begin{lem}\label{lem:udu}
If $u$ is a weak solution of \eqref{eq:divform}, then so is $\partial_{t}u$.  
\end{lem}

\begin{proof}
This is a consequence of the fact that $L$ has $t$-independent coefficients. 
\end{proof}

\begin{lem}\label{lem:regular}
Suppose $u$ is a weak solution of \eqref{eq:divform}, then $t\mapsto \nabla\!_{A}u(t, \, \cdot\,)$ is $C^\infty$ from $t>0$ into $L^2_{\loc}(\R^n)$ and  for all $t>0$, 
\begin{equation}
\label{eq: sliceL2}
\barint_{B(x,c_{1}t)}  |\nabla\!_{A}u(t, x)|^2\, dx \lesssim  \bariint_{W(t,x)}  |\nabla\!_{A}u(s, y)|^2\, dsdy.
\end{equation}
\end{lem}  

\begin{proof} Let $K$ be a compact set in $\R^n$ and $[a,b]$ a compact interval in $(0,\infty)$. 
Using the inequality $|f(t)-\barint_{\, [a,b]}f(s)\, ds|^2\le (b-a)\int_{[a,b]} |f'(s)|^2\, ds$ for all $t\in [a,b]$,  we have 
$$\int_{K}|\nabla\!_{A}u(t, x)|^2\, dx \lesssim_{a,b} \iint_{[a,b]\times K}  |\nabla\!_{A}u(s, x)|^2\, dsdx +
\iint_{[a,b]\times K}  |\partial_{s}\nabla\!_{A}u(s, x)|^2\, dsdx.
$$
and we conclude that $\nabla\!_{A}u(t, \, \cdot\,) \in L^2_{\loc}(\R^n)$  using $\partial_{s}\nabla\!_{A}u(s, x)=\nabla\!_{A}\partial_{s}u(s, x)$ and  Lemma \ref{lem:udu}.
Applying this to all $t$-derivatives of $u$ implies the same conclusion  for all $t$-derivatives of the conormal gradient. 

The inequality \eqref{eq: sliceL2} follows from applying the above inequality and Caccioppoli inequality  with $[a,b]= [\tilde c_{0}^{-1}t, \tilde c_{0}t]$ and $K=\overline B(x,\tilde c_{1}t)$, the closed ball, where $1<\tilde c_{0}<c_{0}$,  $\tilde c_{1}<c_{1}$ and $c_{0},c_{1}$ are the parameters in the Whitney region $W(t,x)$. 
\end{proof}

\begin{cor} \label{cor:repeatCaccio} Suppose $u$ is a weak solution of \eqref{eq:divform}. Let  $k\in \N$. For $0<p<\infty$, 
 we have that $\|t^k\nabla \pd_{t}^ku\|_{N^p_{2}} \lesssim \|\nabla u\|_{N^p_{2}}$,  $\|t^{k+1}\nabla \pd_{t}^ku\|_{T^p_{2}} \lesssim \|t\nabla u\|_{T^p_{2}}$, and for $\alpha\ge 0 $
 $\|t^{k+1}\nabla \pd_{t}^ku\|_{T^\infty_{2,\alpha}} \lesssim \|t\nabla u\|_{T^\infty_{2,\alpha}} $.
 \end{cor}
 
 \begin{proof} This follows from a repeated use of Caccioppoli inequality:  starting from $\nabla \pd_{t}^ku$, eliminate $\nabla$ and then control $\pd_{t}^ku$ by $\nabla \pd_{t}^{k-1}u$, and iterate. Details are easy and we omit them.  \end{proof} 

Let us observe that we can truncate in $t$ in these inequalities provided we enlarge the truncation in the right hand side. 

\begin{cor}\label{cor:uniformept}  Suppose $u$ is a weak solution of \eqref{eq:divform}.  Then  
\begin{equation}
\label{eq:ept1}
\sup_{t,t'>0, t/t'\sim 1} \|\nabla\!_{A}u(t, \, \cdot\,)\|_{E^p_{t'}} \lesssim \|\nabla u\|_{N^p_{2}}.
\end{equation}
Similarly, 
\begin{equation}
\label{eq:ept2}
\sup_{t,t'>0, t/t'\sim 1}  t \|\nabla\!_{A}u(t, \, \cdot\,)\|_{E^p_{t'}} \lesssim \|t\nabla u\|_{T^p_{2}}.
\end{equation}
In particular, if one of the right hand sides is finite, then for any $\delta >0$,   $t \mapsto \nabla\!_{A}u(t,\cdot)$ is $C^\infty$ valued in $E^p_{\delta }$.
\end{cor}

\begin{proof} The inequalities follow right away with $t'=c_{1}t$ from \eqref{eq: sliceL2} together with Lemma \ref{lem:trunc} for the second one. Then use the equivalence \eqref{eq:equivept}.
For the regularity, Corollary \ref{cor:repeatCaccio} tells us we have that $\nabla u$ is infinitely differentiable with respect to $t$ in  $N^p_{2}$ topology  or   $T^p_{2}$ topology. Using \eqref{eq:ept1} or \eqref{eq:ept2}, 
 and the fact that $E^p_{\delta }=E^p_{t}$ for all $t$ with equivalent norms, we can obtain that  $t \mapsto \nabla\!_{A}u(t,\cdot)$ is $C^\infty$ valued in $E^p_{\delta }$, using Lebesgue differentiation theorem and induction. Details are standard and skipped. 
\end{proof}

\section{Review of  basic material on $DB$ and $BD$}\label{sec:basic}
 
 All the material below can be found in  \cite{AS}.
With $D$ and $B$ given in \eqref{eq:dirac} and \eqref{eq:hat},  the operators $ T=DB$ and $BD$  with natural domains $B^{-1}\dom_{2}(D)$ and $\dom_{2}(D)$ are bisectorial operators with bounded holomorphic functional calculus on $L^2=L^2(\R^n; \C^N)$.  Their restrictions to their closed ranges are injective. 

An important operator here is the orthogonal projection $\IP: L^2 \to \clos{\ran_{2}(D)}= \clos{\ran_{2}(DB)}$.  It is a Calder\'on-Zygmund  operator, thus extends to a bounded operator  on $L^p $, $1<p<\infty$, $H^p$, $0<p\le 1$, etc. 

Recall that for $0<p<\infty$, $\IH^p_{DB}$ is defined as the subspace  of  $\clos{\ran_{2}(DB)}$ with  $\|\psi(tDB)h\|_{T^p_{2}}<\infty$ for an allowable $\psi$ and $H^p_{DB}$ is its completion for this  norm (or quasi-norm). Allowable means that the choice of $\psi$ does not affect the set and the norm, up to equivalence. 
  Similarly 
 $\IH^p_{BD}$ is defined as the subspace of  $\clos{\ran_{2}(BD)}$ with  $\|\psi(tBD)h\|_{T^p_{2}}<\infty$ with an allowable $\psi$ and $H^p_{BD}$ is its completion for this norm (or quasi-norm). 
Also $H^2_{DB}= \IH^2_{DB}= \clos{\ran_{2}(DB)}$ and similarly for $BD$.  The projection $\IP_{BD}$ onto $\clos{\ran_{2}(BD)}$ along $\nul_{2}(BD)=\nul_{2}(D)$ will play an important role in some proofs. 

The semigroup   $(e^{-t|DB|})_{t\ge 0}$ on $H^2_{DB}$   extends to a bounded, strongly continuous semigroup on $H^p_{DB}$ for all $0<p<\infty$, which is denoted by $(S_{p}(t))_{t\ge 0}$. As a matter of fact the $H^\infty$-calculus extends: for any $b\in H^\infty(S_{\mu})$, $b(DB)$, defined on $H^2_{DB}$, extends to a bounded operator  on $H^p_{DB}$. In particular, we have two spectral subspaces $H^{p,\pm}_{DB}$ of $H^{p}_{DB}$ obtained as completions of the images $\IH^{p, \pm}_{DB}$ of $\IH^{p}_{DB}$ under  $\chi^\pm(DB)$ where $\chi^\pm= 1_{\{\pm \re z >0\}}$. The restrictions $S^\pm_{p}(t)$ of $S_{p}(t)$ to $H^{p,\pm}_{DB}$ are the respective extensions of $e^{-tDB}\chi^+(DB)$ and $e^{tDB}\chi^- (DB)$. 

A similar discussion can be made for $BD$. There is also  a notion of H\"older space adapted to $BD$ which is useful here. However, $H^p_{BD}$ and its H\"older version  may not be spaces of measurable functions.  
 
  The space 
$H^p_{D}$ agrees with  $\IP(L^p)=\clos{\ran_{p}(D)}$ if $p>1$ and $\IP(H^p)$ if $\frac{n}{n+1}<p\le 1$. Thus it is a closed and complemented subspace of $L^p$ if $p>1$ and $H^p$ if $p\le 1$. 
When $1<p<\infty$, $H^p_{D}$ and $H^{p'}_{D}$  are dual spaces for the standard inner product $\pair f g = \int_{\R^n} (f(x), g(x))\, dx$. The pair $(u,v)$ inside the integral is the standard complex inner product on $\C^N$ where our functions take their values.  For $p\le 1$, the dual space of $H^p_{D} $ is $\dot \Lambda^\alpha_{D}$, the image of the H\"older space $\dot \Lambda^\alpha$ 
if $\alpha=n(\frac{1}{p}-1)$ or of $BMO= \dot \Lambda^0$ if $\alpha=0$ under $\IP$. 

We use  the notation $v= \begin{bmatrix} v_{\no}\\ v_{\ta}\end{bmatrix}$ for vectors  in $\C^{N}$,  $N=m(1+n)$, where $v_{\no}\in \C^m$ is called the scalar part and $v_{\ta}\in \C^{mn}$ the tangential part of $v$.  The elements in a space $X_{D}$ can be characterized as the elements $f$ in $X$ such that $\curl_{x} f_{\ta}=0$ in distribution sense.

We note that $ \int_{\R^n} (f(x), g(x))\, dx$ makes sense for pre-Hardy spaces $\IH^p_{DB}, \IH^{p'}_{B^*D}$ when $1<p<\infty$ and these spaces are in duality. Similarly, under this pairing,  $\IH^{p,\pm}_{DB}, \IH^{p', \pm}_{B^*D}$ are in duality $1<p<\infty$.  When $p\le 1$, one replaces $\IH^{p'}_{B^*D}$ by its H\"older version. 
When taking completions, the pairing becomes an abstract one and the dualities extend.

 \begin{thm}\label{thm:IL}
 There is an open interval $I_{L}$ in $(\frac{n}{n+1},\infty)$ such that for $p\in I_{L}$,  $H^p_{DB}=H^p_{D}$ with equivalence of norms. For $p\in I_{L}$, we also have  $\IH^p_{DB}=H^p_{DB}\cap L^2=H^p_{DB}\cap H^2_{DB} $ and 
 $\IH^p_{DB}=\IH^p_{D}$ with equivalence of norms.  
 
 The semigroup $S_{p}(t)$ regularizes in the scale of Hardy spaces. More precisely, for $0<p\le q<\infty$ both in the interval $I_{L}$,   $S_{p}(t)$ maps $H^p_{D}$ to $H^q_{D}$   with norm  bounded by $Ct^{-(\frac{n}{p}-\frac{n}{q})}$.  If, moreover, $p\le 2$, then  $S_{p}(t)$ maps $H^p_{D}$ to $\IH^q_{D}= H^q_{D}\cap H^2_{D}$ equipped with the same norm as $H^q_{D}$.
\end{thm}
 
 This interval always contains $[\frac{2n}{n+2}, 2]$ but it could be larger. This is the case  for an equation (instead of a system) in \eqref{eq:divform} with real coefficients: $I_{L}$ contains $[1, 2]$. 
 
 Based on the observation  that the semigroup $S_{p}(t)$ allows one to construct solutions of \eqref{eq:divform},  the thesis of \cite{AS} was to obtain estimates on such solutions in terms of their trace at time $t=0$.  
 We recall that our goal here is to show  that all solutions with such estimates have a trace and are given from the semigroup  acting on this trace.  
   
More results related to $DB$ and $BD$, and in particular the needed estimates,   will be given along the way.

\section{Preparation}

For a function $(t,x)\mapsto f(t,x)$, we use the notation $f_{t}$ or $f(t,\,\cdot\,)$ to designate the map $x\mapsto f(t,x)$. 

\begin{lem}[\cite{AAMc, AA1}] If $u$ is a weak solution  to \eqref{eq:divform} on $\reu$, then
$F=\nabla\!_{A} u$ is an $L^{2}_{loc}(\reu; \C^N)$ solution of 
\begin{equation}
\label{eq:curlfreesystem}
\begin{cases}
    \curl_x F_\ta &=0,  \\
    \partial_{t}F + DBF&=0.\end{cases}
\end{equation}
The  equations are interpreted in the sense of distributions in $\reu$, and $D$ and $ B$ are defined in \eqref{eq:dirac} and \eqref{eq:hat}. For the second one, it reads
\begin{equation}    \label{eq:tstfcnforinteq}
  \iint_{\reu}  \bpaire {\pd_s \varphi(s,x)}{F(s,x)} \, dsdx= \iint_{\reu} \bpaire {B^*(x)D\varphi(s,x)}{  F(s,x)}\big) dsdx\end{equation}
for all $\varphi\in C^\infty_0(\R^{1+n}_+; \C^N)$ and the integrals exist in the Lebesgue sense with $|\pd_s \varphi||F|$ and $|B^*D\varphi||F|$ integrable.
Conversely, if $F$ is an $L^{2}_{loc}(\reu; \C^N)$ solution of \eqref{eq:curlfreesystem} in $\reu$, then there exists, unique up to a constant in $\C^m$,  a weak solution $u$ to $Lu=0$ in $\reu$ such that $F= \nabla\!_{A} u$. 
\end{lem}

We  will use the integral notation when it makes sense, that is having verified integrability. In general, we will be careful about justifying use of integrals or duality pairings. 

\begin{rem} By Lemma \ref{lem:regular},  $t\mapsto F_{t}$  belongs to  $C^\infty(0,\infty; L^2_{loc}(\R^n;\C^N))$. Thus 
$\partial_{t}F_{t} \in L^2_{loc}(\R^n;\C^N)$ and the equality 
 \begin{equation}
\label{eq:L2loc}
\partial_{t}F_{t} =- DBF_{t}
\end{equation}  holds in $ L^2_{loc}(\R^n;\C^N)$ for all $t>0$  by standard arguments.   Similarly, we have $\curl_{x} (F_t)_{\ta}=0$ in  $\mD'(\R^n;\C^N)$ for all $t>0$.
\end{rem}  

\begin{rem}\label{rem:test} It suffices to test for $\varphi$ with $\varphi, \partial_{t}\varphi, D\varphi \in L^2$  and compact support in $\reu$. Indeed, one can regularize by convolution in $(t,x)$ and obtain a test function in $C^\infty_0(\R^{1+n}_+; \C^N)$ to which \eqref{eq:tstfcnforinteq} applies and then pass to the limit using $F \in C^\infty(0,\infty; L^2_{loc}(\R^n;\C^N))$. 
\end{rem}

\begin{lem} Assume $F\in L^{2}_{loc}(\reu; \C^N)$ is  a solution of \eqref{eq:curlfreesystem} in $\reu$.  Let $\phi_{0}\in \IH^2_{B^*D}$. Fix $t\in \R_{+}$ and set $\R_{+,t}=\R_{+}\setminus \{t\}.$ 
Let $\varphi(s,x)=\varphi_{s}(x)$ with 
\begin{equation}
\label{eq:phis}
\varphi_s:= \begin{cases}
       e^{-(t-s) |B^*D|}\chi^+(B^*D)  \phi_0 = e^{-(t-s) B^*D}\chi^+(B^*D)  \phi_0, & \text{if\ } s<t ,\\
      -e^{-(s-t) |B^*D|}\chi^-(B^*D)  \phi_0 =  - e^{(s-t) B^*D}\chi^-(B^*D)  \phi_0, & \text{if\ } s>t.
\end{cases} 
\end{equation}
Let $\eta\in Lip(\R_{+})$ with compact support  in $\R_{+,t}$ and $\chi\in Lip(\R^n)$, with compact support,  real-valued.  Then, the functions $$(s,x) \mapsto |\eta'(s)\chi(x) B^*(x)D \varphi(s,x)||F(s,x)|$$ and $$(s,x)\mapsto |\eta(s)B^*(x)D_{\chi}(x)\partial_{s}\varphi(s,x)||F(s,x)|$$ are integrable in $(s,x)$ and one has the identity
\begin{align}    \label{eq:step1}
  \iint_{\reu}  &\bpaire {\eta'(s)\chi(x) B^*(x)D \varphi(s,x)}{F(s,x)} dsdx
  \\
\nonumber &  = \iint_{\reu} \bpaire {\eta(s)B^*(x)D_{\chi}(x)\partial_{s}\varphi(s,x)}{F(s,x)} dsdx,\end{align}
where $D_{\chi}$ is a bounded, supported on $\supp(\chi)$,  matrix-valued function with $\|D_{\chi}\|_{\infty}\lesssim \|\nabla \chi\|_{\infty}.$  \end{lem} 

\begin{proof} Observe that $\varphi \in C^\infty(\R_{+,t}; \dom_{2}(B^*D))$ and that 
\begin{equation}
\label{eq:varphis}
\partial_{s}\varphi_{s}=B^*D\varphi_{s}, \quad s\ne t.
\end{equation} 
For $(s,x)\in \reu$, set
$$
\varphi^{\eta, \chi}(s,x):= \eta(s)\chi(x) \varphi_{s}(x).
$$
Since the support of $\eta$ does not contain $t$, it makes sense even if $\varphi_{t}$ is not defined. 
 The functions $\varphi^{\eta, \chi}$, $\partial_{s}\varphi^{\eta, \chi}$ and $D\varphi^{\eta, \chi}$  are well-defined compactly supported $L^2$ functions  with 
  \begin{align*}
 \partial_{s}\varphi^{\eta, \chi}(s,x)   & = \eta'(s)\chi(x) \varphi_{s}(x) +  \eta(s)\chi(x) \partial_{s}\varphi_{s}(x)    \\
B^*(x)D\varphi^{\eta, \chi}(s,x)   & = \eta(s)B^*(x)[D,\chi] \varphi_{s}(x) +  \eta(s)\chi(x)B^*(x)D \varphi_{s}(x)      
\end{align*}
where $[D,\chi]$ is the commutator between $D$ and multiplication by $\chi$: it is a multiplication with the function $D_{\chi}$ of the statement. It follows that $\varphi^{\eta, \chi} $  is a test function for \eqref{eq:tstfcnforinteq} according to Remark \ref{rem:test}.
Plugging the expressions into \eqref{eq:tstfcnforinteq} and using 
 \eqref{eq:varphis}, we obtain 
   \begin{align}    \label{eq:step1'} 
  \iint_{\reu} &\bpaire {\eta'(s)\chi(x) \varphi_s(x)}{F(s,x)} dsdx\\ 
  \nonumber & = \iint_{\reu} \bpaire {\eta(s)B^*(x)D_{\chi}(x)\varphi_s(x)}{F(s,x)} dsdx.\end{align}
  Now, we remark that this applies to $\partial_{s}F$ in place of $F$.  On the one hand, using $\partial_{s}F_{s}=-DBF_{s}$ for all $s>0$, and integrating by parts in $x$ (first use  Fubini's theorem and $D=D^*$ in the integration by parts)
 \begin{align*}
     \iint_{\reu}  & \bpaire {\eta'(s)\chi(x) \varphi_s(x)}{\partial_{s}F(s,x)} dsdx \\ 
   &=  - \iint_{\reu}  \bpaire {\eta'(s)B^*(x)D_{\chi}(x)\varphi_s(x)}{F(s,x)} dsdx \\
   &\qquad  - \iint_{\reu}  \bpaire {\eta'(s)\chi(x) B^*(x)D\varphi_s(x)}{F(s,x)} dsdx.
     \end{align*}

On the other hand, integrating by parts in the $s$ variable, 
 \begin{align*}
   \iint_{\reu} & \bpaire {\eta(s)B^*(x)D_{\chi}(x)\varphi_s(x)}{\partial_{s}F(s,x)}\, dsdx \\ 
  &= - \iint_{\reu}  \bpaire {\eta'(s)B^*(x)D_{\chi}(x)\varphi_s(x)}{F(s,x)} dsdx  \\
 & \qquad    -\iint_{\reu}  \bpaire {\eta(s)B^*(x)D_{\chi}(x)\partial_{s}\varphi_s(x)}{F(s,x)}\, dsdx. 
   \end{align*}
Combining all this yields  \eqref{eq:step1}.
 \end{proof}

\section{Proof of Theorem \ref{thm:main1}: (i) implies (iii)}

We assume (i) :  $\|\tN F\|_{p}<\infty$ where $F$ is a solution to \eqref{eq:curlfreesystem} and $p\in I_{L}$, that is, $H^p_{DB}=H^p_{D}$. 

\

\paragraph{\textbf{Step 1}}   Finding the semigroup equation. 

\

This will be achieved  by  taking limits in  \eqref{eq:step1} by selecting $\chi$, $\eta$ and $\phi_{0}$ in the definition of $\varphi_{s}$ in \eqref{eq:phis}. 

\

\paragraph{\textbf{Step 1a}} Limit in space.   We show that if $\phi_{0}\in \clos{\ran_{2}(B^*D)}$, with in addition $\phi_{0}\in \IH^{p'}_{B^*D}$ (or, equivalently, $\IP\phi_{0}\in 
\IH^{p'}_{D}$) when $p>1$, then \begin{equation}    \label{eq:step2(i)}
  \iint_{\reu} \bpaire {\eta'(s) B^*(x)D\varphi_s(x)}{F(s,x)}\,   dsdx = 0\end{equation} 
 where the integral is defined in the Lebesgue sense.

\

We replace  $\chi$ in \eqref{eq:step1}  by $\chi_{R}$ with $\chi_{R}(x)=\chi(x/R)$ where $\chi\equiv 1$ in the unit ball $B(0,1)$, has compact support in the ball $B(0,2)$ and let $R\to \infty$.  As $\chi_{R}$ tends to 1 and $D_{\chi_{R}}$ to 0, it suffices by dominated convergence to show that $|\eta'(s)B^*D \varphi_s  F|$ and $|\eta(s)\partial_{s}\varphi_{s} F|$ are integrable on $\reu$.  As $1_{\supp \eta}F\in T^p_{2}$, it is enough to have that $\eta'(s)B^*D \varphi_s$ and $\eta(s)\partial_{s}\varphi_{s}$ belong to $(T^p_{2})'$.  As $\partial_{s}\varphi_{s}=B^*D\varphi_{s}$ on $\supp \eta$, it suffices to invoke the following lemma.

\begin{lem}\label{lem:step1} Assume $\phi_{0}\in \clos{\ran_{2}(B^*D)}=\IH^2_{B^*D}$ and, in addition if $p>1$,  $\phi_{0}\in \IH^{p'}_{B^*D}$. Then $\eta(s)B^*D\varphi_{s}\in (T^p_{2})'$ for all $\eta$ bounded and compactly  supported  away from $t$. 

\end{lem} 

\begin{proof} We begin with the case $p\le 1$.  Set  $G_{s}:= |\eta(s)B^*D\varphi_{s}|$ and $G(s,x)=G_{s}(x)$. The definition of $\varphi_{s}$ and the properties of $\varphi_{s}$ show that   $G_{s}\in L^2(\R^n)$ uniformly in $s\in \R_{+}$. As $s\mapsto G_{s}$ also has compact support in $\R_{+}$, it is easy to see that $C_{\alpha}G\in L^\infty$ (the bound depends on $\eta$) for any $\alpha\ge 0$.  In particular, $G \in (T^p_{2})'$. 

We continue with the case $p>1$ and we use the functional calculus for $B^*D$.  We have 
$$
\eta(s)B^*D\varphi_{s}=\begin{cases}
 \frac{\eta(s)}{t-s} \psi_{+}((t-s)B^*D)\phi_{0}     & \text{if } s<t,  \\
      \frac{\eta(s)}{s-t} \psi_{-}((s-t)B^*D)\phi_{0}     & \text{if } t<s, \
\end{cases} 
$$
with $\psi_{\pm}$ the bounded homomorphic functions defined by $\psi_{\pm}(z)=\pm ze^{\mp z}\chi^\pm(z)$.
By geometric considerations as  $t-s$ is bounded away from 0, we see that if $(s,y)$ belongs to a cone $\Gamma_{x}$ with $s\in \supp(\eta) \cap (0,t)$ then $(t-s,y)$ belongs to a  cone $\wt \Gamma_{x}$ with (bad) finite aperture depending on the support of $\eta$. Thus, setting $t-s=\sigma$ and using also that $s$ is bounded below and $t-s$ bounded above on $\supp(\eta)$,  we obtain
$$
\iint_{\Gamma_{x}, s<t} |\eta(s)B^*(y)D\varphi_{s}(y)|^2 \,\frac{dsdy}{s^{n+1}} \lesssim_{\eta} \iint_{\wt \Gamma_{x}} |\psi_{+}(\sigma B^*D)\phi_{0}(y)|^2 \,\frac{d\sigma dy}{\sigma^{n+1}}.
$$
The change of aperture  allows us to use $\wt \Gamma_{x}$ in estimating tent space norms.  By Theorem 5.7 of \cite{AS},  we have that 
$$
\|\psi_{+}(\sigma B^*D)\phi_{0}\|_{T^{p'}_{2}}\lesssim  \|\IP\phi_{0}\|_{p'},
$$
where $\IP$ is the orthogonal projection of $L^2$ onto $ \clos{\ran_{2}(D)}$. By the assumption on $p$,  $\IP\phi_{0}\in \IH^{p'}_{D}$, and in particular, $\IP\phi_{0}\in L^{p'}$ with $  \|\IP\phi_{0}\|_{p'}\sim \|\phi_{0}\|_{\IH^{p'}_{B^*D}}$.  

The argument is the same when $t<s$, replacing $\psi_{+}$ by $\psi_{-}$.  \end{proof}

  \

\paragraph {\textbf{Step 1b}}  Limit in time.  We assume the condition of Step 1a on $\phi_{0}$ hold.

\

Now we select appropriate functions $\eta$ in \eqref{eq:step2(i)} depending on $t$, which is still a fixed positive real number. We follow \cite{AA1} for the choices but the limit process requires more care.  

We let $0<\varepsilon<\inf(t/4, 1/4, 1/t)$ and pick $\eta$ piecewise linear and continuous with
$\eta (s)= 0$ if $s<t+\varepsilon$,  $\eta (s)=1$ if $t+2\varepsilon\le s\le t+ \frac{1}{2\varepsilon}$ and $\eta(s)=0$ if $s>t+\frac{1}{\varepsilon}$. 
Plugging this choice into  \eqref{eq:step2(i)} and reorganizing we obtain
\begin{align}   \label{eq:limit-(i)} 
  \frac{1}{\varepsilon} \iint_{[t+\varepsilon, t+2\varepsilon]\times \R^n}& \bpaire { B^*(x)D\varphi_s(x)}{F(s,x)}\,   dsdx \\\nonumber & = 2\varepsilon \iint_{[t+ \frac{1}{2\varepsilon}, t+\frac{1}{\varepsilon}]\times \R^n} \bpaire { B^*(x)D\varphi_s(x)}{F(s,x)} \, dsdx.
  \end{align}

We make  a second choice of $\eta$, again piecewise linear and continuous, with
$\eta (s)= 0$ if $s<\varepsilon$,  $\eta (s)=1$ if $2\varepsilon\le s\le t- {2\varepsilon}$ and $\eta(s)=0$ if $s>t-{\varepsilon}$.  Plugging this choice into  \eqref{eq:step2(i)} with $4\varepsilon<t$,  we obtain
\begin{align}\label{eq:limit+(i)}    
  \frac{1}{\varepsilon} \iint_{[t-2\varepsilon, t-\varepsilon]\times \R^n}& \bpaire { B^*(x)D\varphi_s(x)}{F(s,x)}\,   dsdx \\\nonumber &=\frac{1}{\varepsilon} \iint_{[{\varepsilon}, 2{\varepsilon}]\times \R^n} \bpaire { B^*(x)D\varphi_s(x)}{F(s,x)} \, dsdx.
  \end{align}

As said our goal is to pass to the limit in both formulas.

  The second integral in \eqref{eq:limit-(i)} converges to 0 as $\varepsilon\to 0$ for fixed $t$ since $1_{[\varepsilon,2\varepsilon]}F\in T^p_{2}$. Indeed, 
using the function $\psi_{-}$ defined earlier, set $G(s,x)= G_{s}(x)= 1_{[t+ \frac{1}{2\varepsilon}, t+\frac{1}{\varepsilon}]}(s)(s-t)^{-1} \psi_{-}((s-t)B^*D)  \phi_0(x)$, under the conditions on $\phi_{0}$ in Step 1a.   As $t\varepsilon<1$, then $s\in [t+ \frac{1}{2\varepsilon}, t+\frac{1}{\varepsilon}]$ implies $s\in [ \frac{1}{2\varepsilon}, \frac{2}{\varepsilon}]$. Remark that this integral is bounded by  
  $$
  2\varepsilon \|1_{[ \frac{1}{2\varepsilon}, \frac{2}{\varepsilon}]}F\|_{T^p_{2}}\|sG\|_{(T^p_{2})'} \lesssim 2\varepsilon \|\tN F\|_{p}\|sG\|_{(T^p_{2})'}. 
  $$
 It remains to show $\|sG\|_{(T^p_{2})'}<\infty$.  Assume first $p\le 1$ and let $\alpha=n(\frac{1}{p}-1)$.  Since $\|G_{s}\|_{2}\lesssim \varepsilon\|\phi_{0}\|_{2}$ for those $s$, we have
  $\|C_{\alpha}(sG)\|_{\infty}\lesssim \varepsilon^{\alpha+\frac{n}{2}}\|\phi_{0}\|_{2}$. 
  
  Next, consider $p>1$. Then one sees that $\|sG\|_{T^{p'}_{2}} \lesssim  \|\psi_{-}(\sigma B^*D)  \phi_0\|_{T^{p'}_{2}} \lesssim  \|\IP \phi_{0}\|_{p'}\sim \|\phi_{0}\|_{\IH^{p'}_{B^*D}}$. 
  
  Therefore the left hand side of \eqref{eq:limit-(i)} converges to 0 as well. However, the way to interpret this limit and also how we can pass to the limit in \eqref{eq:limit+(i)} depend on whether $p\le 2$ or $p>2$. 
  
  We break the continuation of this step in the two cases $p\le 2$ and $p>2$.

\

\subparagraph{\textbf{Case $\mathbf{p\le 2}$}}

Changing $s$ to $t+s$ in the first integral, this shows that 
  \begin{equation}
\label{eq:limit--}
\lim_{\varepsilon\to 0}\frac{1}{\varepsilon} \iint_{[\varepsilon, 2\varepsilon]\times \R^n} \bpaire { (B^*De^{sB^*D}\chi^-(B^*D)\phi_{0})(x)}{F(t+s,x)}\,   dsdx=0.
\end{equation}

Before we take this limit, we reinterpret this integral.  We begin with

\begin{lem}\label{lem:q} For $q=p$ if $1<p\le 2$ and $q=p\frac{n+1}{n}\in (1,2]$ if $p\le 1$, we have $F\in C^\infty(0,\infty; H^q_{D}).$
\end{lem}

\begin{proof} Assume $p\le 1$. By Lemma \ref{lem:HMiMo} applied to $F$, we obtain $F\in L^q_{loc}(0,\infty; L^q)$. As one can apply the same argument replacing $F$ by its $t$-derivatives, we have $F\in C^\infty(0,\infty; L^q).$ Now, $\curl_{x} (F_{t})_{\ta}= 0$ in $\mD'(\R^n)$ as remarked before.  As $F_{t}\in L^q$, this means that 
$F_{t}\in \clos{\ran_{q}(D)}=H^q_{D}$ for all $t>0$ (see Section \ref{sec:basic}). 

Next, assume $1<p\le 2$. We know that $F\in L^p_{loc}(0,\infty; L^p)$. The rest of the argument is the same. 
\end{proof}

\begin{lem}\label{lem:duality(i)} Let $q$ be as Lemma \ref{lem:q}. For $\phi_{0}\in \IH^{q'}_{B^*D}$, $\varepsilon>0$, $t>0$, we have
\begin{align*}
\frac{1}{\varepsilon} \iint_{[\varepsilon, 2\varepsilon]\times \R^n}& \bpaire { (B^*De^{sB^*D}\chi^-(B^*D)\phi_{0})(x)}{F(t+s,x)}\,   dsdx \\ & = \barint_{[\varepsilon, 2\varepsilon]} \pair { B^*De^{s B^*D}\chi^-(B^*D)  \phi_0}{F_{t+s}}\,   ds
\end{align*}
where the pairing is the (sesquilinear) duality between $H^{q'}_{B^*D}$ and $H^{q}_{DB}$. 
\end{lem}

\begin{proof} Note that $t,\varepsilon>0$ are fixed so estimates are allowed to depend on them. 
Set $G(s,x)=(sB^*De^{sB^*D}\chi^-(B^*D)\phi_{0})(x) $ and $H(s,x)= 1_{[\varepsilon,2\varepsilon]}(s){F(t+s,x)}$. Because of the truncation in $s$, we know that $H\in T^p_{2}$ (with norm depending on $t,\varepsilon$)  and $G\in (T^p_{2})'$ so that the integral on the left hand side is a Lebesgue integral. We may apply Proposition \ref{prop:approx}, so that this integral is a limit of the integral where $H$ has been mollified by convolution in the $x$-variable, that is,  $H(s,x)$ is replaced by $H_{k}(s,x)=H_{k,s}(x)= H_{s}\star_{\R^n}\rho_{k}(x)$. But recall that $H_{s}\in H^q_{D}=\clos{\ran_{q}(D)}$ so that $\IP H_{s}=H_{s}$ for all $s>0$. This condition is preserved by convolution since it commutes with $\IP$ because $\IP$ is also given by convolution . Also such mollifications map $L^q$ into $L^2$, thus $H_{k,s}\in \IH^2_{D}$ as well so that $H_{k,s}\in \IH^q_{D}=\IH^q_{DB}$ (as $p\le q \le 2$, thus $q\in I_{L}$). As $G_{s}\in \IH^{q'}_{B^*D}$ (because $\phi_{0}\in \IH^{q'}_{B^*D}$ and this space is preserved by $H^\infty(S_{\mu})$ functions of $B^*D$) which is in duality with $\IH^q_{DB}$ for the standard $L^2$ duality (see Section \ref{sec:basic}), 
$$
\int_{\R^n} \bpaire { G(s,x)}{H_{k}(s,x)}\,   dx= \pair { G_{s}}{ H_{k,s}}
$$
for any fixed $s$, where the pairing is  the  duality  extended to $H^{q'}_{B^*D}, H^q_{DB}$. But $H_{k,s}$ converges to $H_{s}$ in $L^q$, hence in $H^q_{D}=H^q_{DB}$, as $k\to \infty$ and this is uniform in $s$. Thus 
$\pair { G_{s}}{ H_{k,s}}$ converges to $\pair { G_{s}}{ H_{s}}$ uniformly in $s$. It remains to integrate the equality above in $s$ and pass to the limit.  
\end{proof}

We continue with  Step 1b in this case ($p\le 2$) and conclude for the limit of \eqref{eq:limit-(i)}.  Pick $\delta >0$, replace $\phi_{0}$ by $e^{-\delta |B^*D|}\phi_{0}$ in \eqref{eq:limit--} and use Lemma \ref{lem:q} and \ref{lem:duality(i)}, to obtain
$$
\lim_{\varepsilon\to 0} \barint_{[\varepsilon, 2\varepsilon]} \pair { B^*De^{(s+\delta) B^*D}\chi^-(B^*D)  \phi_0}{F_{t+s}}\,   ds = 0.
$$
Now, the map $s\mapsto F_{t+s}$ is continuous at $s=0$ into $H^q_{D}=H^q_{DB}$ and  the map $s\mapsto B^*De^{(s+\delta) B^*D}\chi^-(B^*D)  \phi_0$ is continuous at $s=0$ into 
$\IH^{q'}_{B^*D}$.  For the last point, this is because we have the continuity of the semigroup in 
$\IH^{q'}_{B^*D}$ (\cite{AS}, Proposition 4.5).  It follows that the limit is the value at $s=0$ of the integrand and we have obtained for all $\phi_{0}\in \IH^{q'}_{B^*D}$, all
$\delta >0$ and $t>0$  that
$$
\pair { B^*De^{\delta B^*D}\chi^-(B^*D)  \phi_0}{F_{t}} = 0.
$$
We deduce from this equation that $F_{t} \in H^{q,+}_{DB}$. Indeed, 
 using semigroup theory,  the vector space containing $\{B^*De^{\delta B^*D}\chi^-(B^*D)  \phi_0, \phi_{0}\in \IH^{q'}_{B^*D}, \delta >0\}$ forms a dense subspace of $\IH^{q', -}_{B^*D}$ [$\chi^-(B^*D)  \phi_0 = \lim_{t\to 0}\int_{0}^t B^*De^{\delta B^*D}\chi^-(B^*D)  \phi_0\, d\delta $ and approximate each integral by Riemann sums]   which is dense in $H^{q',-}_{B^*D}$. As $F_{t}\in H^q_{DB}$, this shows that $F_{t} \in H^{q,+}_{DB}$ (See Section \ref{sec:basic}).

We turn to taking limits in \eqref{eq:limit+(i)} still in the case $p\le 2$. Reinterpreting the $dx$ integrals using the duality pairing between $H^{q'}_{B^*D}$ and $H^q_{DB}$ as in Lemma \ref{lem:duality(i)} and  reorganizing we obtain a second equation for fixed $t>0 $ and $\phi_{0}\in \IH^{q'}_{B^*D}$, 
\begin{align}
\label{eq:limit++(i)}
\barint_{[\varepsilon, 2\varepsilon]} \pair { B^*De^{ -s B^*D}& \chi^+(B^*D)  \phi_0}{F_{t-s}} \,   ds& \\  \nonumber =& \barint_{[\varepsilon, 2\varepsilon]} \pair { B^*De^{ -(t-s)  B^*D}\chi^+(B^*D)  \phi_0}{F_{s}}\,   ds.
\end{align}
Next, if we replace $\phi_{0}$ by $e^{-\delta |B^*D|} \phi_{0}$ with $\delta >0$, we can pass to the limit as $\varepsilon\to 0$ as above and obtain for the integral in the left hand side of \eqref{eq:limit++(i)}, 
$$
\pair { B^*De^{-\delta B^*D}\chi^+(B^*D)  \phi_0}{F_{t}}.
$$
Set 
$$
I_{t}^{\varepsilon,\delta }= \barint_{[\varepsilon, 2\varepsilon]} \pair { B^*De^{ -(t-s) B^*D}\chi^+(B^*D)  e^{-\delta |B^*D|}\phi_0}{F_{s}}\,   ds.
$$
We thus have  shown 
$$
\lim_{\varepsilon\to 0} I_{t}^{\varepsilon,\delta }= \pair { B^*De^{-\delta B^*D}\chi^+(B^*D)  \phi_0}{F_{t}}
$$
for all $t,\delta >0$. 
We shall use this information together that $F_{t}\in H^{q,+}_{DB}$  to prove that  for all $t>0$,  $F_{t} \in \IH^{q,+}_{DB}$ and for all $t>0$ and $\tau \ge 0$, 
$$
F_{t+\tau}= e^{-\tau DB}\chi^+(DB)F_{t}= e^{-\tau |DB|}F_{t}.
  $$
So we fix $t,\delta , \tau$. Observe that   $I_{t+\tau}^{\varepsilon,\delta }= I_{t}^{\varepsilon,\delta +\tau }$. Thus, we obtain
$$
 \pair { B^*De^{-(\tau+\delta) B^*D}\chi^+(B^*D)  \phi_0}{F_{t}}=  \pair { B^*De^{-\delta B^*D}\chi^+(B^*D)  \phi_0}{F_{t+\tau}}.
 $$
 Because $F_{t}\in H^{q,+}_{DB}$, the  left hand side can be rewritten as  
 $$
  \pair { B^*De^{-\delta B^*D}\chi^+(B^*D)  \phi_0}{S_{q}^+(\tau)F_{t}}
  $$
  where $S_{q}^+(\tau)$ is the bounded extension of  $e^{-\tau DB}\chi^+(DB)$ to  $H^{q,+}_{DB}$ described in Section \ref{sec:basic}.  By density as before, this shows that for all $t>0$ and $\tau\ge 0$ 
  $$
  F_{t+\tau}= S_{q}^+(\tau)F_{t}. 
 $$
 Remark that as we already know that $F_{t} \in H^{q,+}_{DB}$, this is the same as
 $$
   F_{t+\tau}= S_{q}(\tau)F_{t}. 
 $$
 By  Corollary 8.3 in \cite{AS}, the semigroup maps $H^{q}_{DB}$  into $\IH^2_{DB}$.   This proves  that $F_{t+\tau} \in \IH^2_{DB}$ and, therefore, $F_{t} \in \IH^{q,+}_{DB}$ for all $t>0$ so that  the original semigroup acts on $F_{t}$ and we can  write for all $t>0$ and $\tau\ge 0$,
 $$
F_{t+\tau}= e^{-\tau DB}\chi^+(DB)F_{t}= e^{-\tau |DB|}F_{t}.
  $$

 \

 \paragraph{\textbf{Step 2 when} $\mathbf{p\le 2}$}   Finding the trace at $t=0$.

 \

We distinguish two subcases.

 \

 \paragraph{\textbf{Step 2a}: $\mathbf{1<p\le 2}$}

 \
 
 If $1<p\le 2$, recall that $q=p$ from Lemma \ref{lem:q}. From  Lemma \ref{lem:p}, $\barint_{\, \, [t,2t]}\|F_{s}\|_{p}\, ds\lesssim 
 \|\tN F\|_{p} <\infty$. In particular $\wt F_{t}= \barint_{\,\,[t,2t]}F_{s}\, ds$ belongs to $L^p$ with uniform bound with respect to $t$, thus it has  a subsequence $\wt F_{t_{k}}$ converging weakly in $L^p$ with $t_{k}\to 0$. Call  the limit $F_{0}$. Then   $F_{0}\in H^{p,+}_{DB}$ (which is the completion of $\IH^{p,+}_{DB}$). Now 
 $e^{-\tau |DB|}\wt F_{t_{k}}= S_{p}^+(\tau)\wt F_{t_{k}}$ converges weakly to $S_{p}^+(\tau)F_{0}$ by continuity of  the semigroup $S_{p}^+(\tau)$ on $H^{p,+}_{DB}$, while 
 $\barint_{\,\,[t_{k},2t_{k}]}F_{\tau+ s}\, ds$ converges strongly to $F_{\tau}$ and we obtain
 $$
 S_{p}^+(\tau)F_{0}  = F_{\tau}.
 $$
 
 \

 \paragraph{\textbf{Step 2b}: $\mathbf{p\le 1}$}

 \

This case is more delicate.  Let us explain the strategy. We introduce a new Banach space $\wt H^p_{DB}$ of Schwartz distributions which contains $H^p_{DB}$. This means that we will have the containments $ H^p_{DB} \subset  \wt H^p_{DB}\subset  \mS'$ with \textbf{continuous  inclusions}.  We shall obtain the trace and representation of $F_{t}$ in $\wt H^p_{DB}$. Then we shall show that the trace actually belongs to the smaller space $H^p_{DB}$ (a regularity result) and conclude from this for the representation $
 S_{p}^+(\tau)F_{0}  = F_{\tau}
 $  in the space $H^p_{DB}$. 

We begin by building $\wt H^p_{DB}$. Several choices are possible but in a very narrow window to match both the functional calculus of $DB$ and the usual calculus of distributions. Recall that $p\in I_{L}$ and $\frac{n}{n+1}<p\le 1$. Let $p_{0}=\inf I_{L}$ and $s_{0}=n(\frac{1}{p_{0}}-1)$. We fix $n(\frac{1}{p}-1)<s<s_{0}$ and we select $r>1$ so that $n(\frac{1}{p}-\frac{1}{r})<1$ and 
\begin{equation}
\label{eq:s}
\beta=\frac{2}{r'}+ s\bigg(1-\frac{2}{r'}\bigg) > n\bigg(\frac{1}{p}-\frac{1}{r}\bigg)= \alpha.
\end{equation}
The different possibilities are in the choice of this $r$.

Let $\wt \IH^{p}_{DB}=\{ h\in \IH^{2}_{DB}\, ; \, \|h\|_{\wt \IH^p_{DB}}<\infty\} $ with  $$\|h\|_{\wt \IH^p_{DB}}=\bigg(\iint_{\reu} |e^{-t|DB|}h(x)|^r t^{n(\frac{r}{p}-1)}\, \frac{dtdx}{t}\bigg)^{1/r}= \bigg(\int_{0}^\infty  \|e^{-t|DB|}h\|_{r}^r\,  t^{\alpha r}\, \frac{dt}{t}\bigg)^{1/r}.$$ 
Set $ \wt \IH^{p, \pm}_{DB}= \wt \IH^{p}_{DB} \cap  \IH^{2,\pm}_{DB}$. Let $\wt H^{p}_{DB}$ and $ \wt H^{p,\pm}_{DB}$ be the respective  completions with respect to this norm.  

\begin{lem}\label{lem:disthtilde} The space $\wt H^{p}_{DB}$ embeds in the space of Schwartz distributions.
\end{lem}

\begin{proof} Let $Q_{t}= t^2\Delta e^{t^2\Delta}$ where $\Delta$ is the ordinary negative self-adjoint Laplacian on $\R^n$.  
We first claim that for $h\in \IH^2_{DB}$, 
\begin{equation}
\label{eq:embed}
\bigg(\int_{0}^\infty  \|Q_{t}h\|_{r}^r\,  t^{\alpha r}\, \frac{dt}{t}\bigg)^{1/r} \lesssim \bigg(\int_{0}^\infty  \|e^{-t|DB|}h\|_{r}^r\,  t^{\alpha r}\, \frac{dt}{t}\bigg)^{1/r}.
\end{equation}
It is well-known that the norm on the left is equivalent to that of the homogeneous Besov space $\dot B^{-\alpha, r}_{r}$, which, as $-\alpha<0$,  can be identified to a subspace of $\mS'$. We make this identification and consider now this Besov space as contained in $\mS'$ with continuous inclusion.  
To prove this inequality, we use the  Calder\'on reproducing formula in the functional calculus of $DB$.  
Let  $\psi$ be  the holomorphic function $\psi(z)=\pm 2 ze^{\mp z}$ for $z\in S_{\mu\pm}$ so that
$$
\int_{0}^\infty \psi(sz) e^{\mp sz}\, \frac{ds}{s}= 1, \quad \forall z\in S_{\mu},
$$
hence, for all $h\in \IH^2_{DB}$, 
$$
\int_{0}^\infty \psi(sDB) e^{-s|DB|}h\, \frac{ds}{s}= h
$$
with convergence in $L^2$. Applying  $Q_{t}$ yields
$$Q_{t}h= \int_{0}^\infty Q_{t}\psi(sDB) e^{-s|DB|}h\, \frac{ds}{s}$$
with convergence in $L^2$.   
By the standard Schur argument, the conclusion will follow from the inequality 
\begin{equation}
\label{eq:schur}
t^\alpha \|Q_{t}\psi(sDB) e^{-s|DB|}h\|_{r}  \lesssim g(s/t) \|e^{-s|DB|}h\|_{r} s^{\alpha},
\end{equation} 
with $g: (0,\infty) \to (0,\infty)$ independent of $h,s,t$ such that $\int_{0}^\infty g(u)\, \frac{du}{u}<\infty$. 
Because $p<r<2$, we have $r\in I_{L}$, hence by \cite{AS}, Theorem 4.19, we have
$ \|\psi(sDB)e^{-s|DB|}h\|_{r} \lesssim \|e^{-s|DB|}h\|_{r}$ uniformly in $s$. Thus, if $t\le s$,
\eqref{eq:schur} holds 
with $g(u)= u^{-\alpha}$ when $u>1$. In the case $t>s$, observe that $\psi(z)=z\tilde \psi(z)$ on $S_{\mu}$ with $\tilde \psi\in H^\infty(S_{\mu})$. Hence, 
$Q_{t}\psi(sDB)=  s Q_{t}D B \tilde \psi(sDB)$ and observe that   $Q_{t}D B  $  is bounded on $L^r$ with norm bounded by $Ct^{-1}$ ($Q_{t}D$ is convolution with an $L^1$ function and $B$ is bounded multiplication). Thus, using again \cite{AS}, Theorem 4.19,   we obtain
\eqref{eq:schur} 
 with $g(u)=u^{1-\alpha}$ if $u<1$. As  $\alpha<1$, the desired property holds for $g$. 
This proves that $\wt \IH^p_{DB} \subset \dot B^{-\alpha, r}_{r}$ continuously. 

To prove the inclusion  when taking completion, we have to show that if $(h_{\varepsilon})$ is a Cauchy sequence in $\wt \IH^p_{DB}$ that converges to 0 in $\dot B^{-\alpha, r}_{r}$ then it also converges to 0 in $\wt H^p_{DB}$ (in other words, this proves that the extension of the identity map is injective). 
Let $G_{\varepsilon}(t,x)= e^{-t|DB|}h_{\varepsilon}(x)$. As $G_{\varepsilon}$ is a Cauchy sequence in $L^r(\reu; \C^N,  t^{\alpha r - 1}dxdt)$,  it converges to some $G$ in this (complete) space. We have to show that $G=0$ (For the moment,  $\wt H^p_{DB}$ is  defined as the closure in $L^r(\reu; \C^N,  t^{\alpha r - 1}dxdt)$  of elements $G$ with $G(t,x)=e^{-t|DB|}h(x)$, $h \in \wt \IH^p_{DB}$). It suffices to show this in the sense of distributions on $\reu$. Let $\chi\in C^\infty_{0}(\reu; \C^N)$. Then $(G-G_{\varepsilon}, \overline \chi)\to 0$ as $\varepsilon\to 0$. Next, we have
$$
(G_{\varepsilon},\overline \chi)= \int_{0}^\infty \pair {e^{-t|DB|}h_{\varepsilon}}{ \chi_{t}} \, \frac{dt}{t}= \pair {h_{\varepsilon}} \varphi
$$
with $$\varphi= \int_{a}^b \IP e^{-t|B^*D|}\chi_{t} \, \frac{dt}{t}.
$$
We used that $\supp \chi \subset [a,b]\times \R^n$ and $\chi_{t}(x)= \chi(t,x)$. We also  used that $h_{\varepsilon}\in \IH^2_{DB}=\ran (\IP)$ to insert  the orthogonal projection $\IP$. For any $t$ fixed in $[a,b]$, we have 
$$D\IP e^{-t|B^*D|}\chi_{t}=D e^{-t|B^*D|}\chi_{t}=  e^{-t|DB^*|}D\chi_{t} \in L^2$$ and by \cite{AS}, Corollary 4.21,
$\IP e^{-t|B^*D|}\chi_{t} =  \IP e^{-t|B^*D|}\IP \chi_{t}\in \dot \Lambda^s$ as $\IP\chi_{t}\in 
 \dot \Lambda^s \cap L^2$ and $s$ can be taken as the one chosen before the statement. Thus $\varphi \in \IH^2_{D}$, $D\varphi\in L^2$ and $\varphi\in  \dot \Lambda^s$.  The first two conditions imply that $\varphi\in W^{1,2}$,  the usual Sobolev space,  
 and  the first and third that $\varphi\in L^\infty$, and by interpolation $\varphi\in L^{r'}$. 
 Now, for some constant $c>0$, one can use the usual Calder\'on reproducing formula to write
$$
\pair {h_{\varepsilon}} \varphi= c \int_{0}^\infty \pair {Q_{t}h_{\varepsilon}} { Q_{t}\varphi} \, \frac{dt}{t}.
$$
As $Q_{t}h_{\varepsilon}$ converges to 0 in $L^r(\reu; \C^N,  t^{\alpha r - 1}dxdt)$, it is enough to show that $Q_{t}\varphi \in L^{r'}(\reu; \C^N,  t^{-\alpha r' - 1}dxdt)$ to conclude that $\pair {h_{\varepsilon}} \varphi \to 0$ as $\varepsilon\to 0$. 
The part for $t>1$ follows from the boundedness of $Q_{t}$  on $L^{r'}$ and $\varphi\in L^{r'}$ as $\alpha r'>0$. 
For $t\le 1$, we use
 $$
\|Q_{t}\varphi\|_{2} \lesssim t \|\nabla \varphi\|_{2}, \quad \|Q_{t}\varphi\|_{\infty} \lesssim  t^s \|\varphi\|_{\dot \Lambda^s}$$
hence,   
$$
\|Q_{t}\varphi\|_{r'}  \lesssim  \|Q_{t}\varphi\|_{2}^{2/r'} \|Q_{t}\varphi\|_{\infty}^{1-2/r'} \lesssim_{\varphi} t^\beta,
$$
where $\beta$ is the number defined in \eqref{eq:s}.  The convergence when $t\le 1$ follows from $\beta>\alpha$. 
\end{proof}

Having shown the embedding of $\wt H^p_{DB}$ in $\mS'$, we decide to identify $\wt H^p_{DB}$ to a subspace of $\mS'$. Note that as we have identified $H^p_{DB}$ to $H^p_{D}$, which is also a subspace of $\mS'$, we can now compare these two realizations of $\wt H^p_{DB}$ and $ H^p_{DB}$ safely.

\begin{lem}  \label{lem:htilde}
\begin{enumerate}
 \item We have the spectral splitting $\wt H^{p}_{DB}= \wt H^{p, +}_{DB}\oplus \wt H^{p,-}_{DB}$. 
  \item The identity map is an embedding of  $H^{p,\pm}_{DB}$  into $\wt H^{p,\pm}_{DB}$ respectively.
\item Let $S_{p}^+(\tau)$ and $\wt S_{p}^+(\tau)$ be the respective bounded extensions by density of $e^{-\tau|DB|}$ on $H^{p,+}_{DB}$ and $\wt H^{p,+}_{DB}$. If $ h\in H^{p}_{DB} \cap \wt H^{p,+}_{DB}$, then  $  h\in H^{p,+}_{DB}$ and  $\wt S_{p}^+(\tau) h = S_{p}^+(\tau) h$.
  \end{enumerate}   

\end{lem}

\begin{proof} 
For  (1),  let $h\in \wt \IH^{p}_{DB}$. Then we have $h=h^+ +h^-$ where $h^\pm=\chi^\pm(DB)h$ and 
$$e^{-t|DB|}h= e^{-t|DB|}h^+ + e^{-t|DB|}h^-= \chi^+(DB)e^{-t|DB|}h + \chi^-(DB)e^{-t|DB|}h.
$$
As $\chi^\pm(DB)$ are bounded operators on $\IH^r_{DB}$ equipped with $L^r$ norm  (\cite{AS}, Theorem 4.19, because $r\in I_{L}$),  we obtain  $\|\chi^\pm(DB) e^{-t|DB|}h\|_{r}\lesssim \| e^{-t|DB|}h\|_{r}$ for all $t>0$. It follows that $
\|h^\pm\|_{\wt \IH^p_{DB}}\lesssim \|h\|_{\wt \IH^p_{DB}}.
$
The splitting follows in the completion. 

For (2), by \cite{AS}, Theorem 9.1,  since $p\in I_{L}$, $H^{p,+}_{DB}$ is a closed subspace of $H^p_{DB}=H^p_{D}$ and 
$$
\|h\|_{\IH^{p}_{DB}}\sim \|h\|_{H^p}\sim \|\tN( e^{-t|DB|}h)\|_{p}, \quad \forall h\in \IH^{p,+}_{DB}.$$
Let $h\in \IH^{p,+}_{DB}$. Combining this with
Lemma \ref{lem:HMiMo} and the choice of $r$,  $h\in \wt \IH^{p}_{DB}$ and
$$
\|h\|_{\wt \IH^{p}_{DB}}  \lesssim  \|h\|_{\IH^{p}_{DB}}.
$$
By completion, this shows  that the identity map extends to a continuous  map from  $H^{p,+}_{DB}$ to $\wt H^{p,+}_{DB}$.  Now, both completions are embedded in  $\mS'$ so this extended map must be injective: it is an embedding. 
The argument for the inclusion of  $H^{p,-}_{DB}$ in $\wt H^{p,-}_{DB}$ is the same. 

Next, the property (3)  is now an easy exercise in functional analysis. We give it for the sake of completeness.  Let $ h\in \wt H^{p,+}_{DB} \cap H^p_{DB}$. There exists $h_{\varepsilon}\in \IH^{p}_{DB}$  converging to $ h$ in $H^{p}_{DB}$. Write $ h= h^++h^-$ according to the spectral splitting $H^p_{DB}= H^{p,+}_{DB}\oplus H^{p,-}_{DB}$. Thus $h_{\varepsilon}^\pm=\chi^\pm(DB)h_{\varepsilon}\in \IH^{p,\pm}_{DB}$ and converge respectively to $h^\pm$ for the $H^p_{DB}$ topology.  By the embedding in part (2), the convergence is also in $\wt H^{p,\pm}_{DB}$. Thus, we obtain that $ h=h^++h^-$ also in $\wt H^p_{DB}$. Since $ h\in \wt H^{p,+}_{DB}$, the splitting obtained in part  (1) yields that $h^-=0$ and $ h=h^+$ and it follows that $ h\in H^{p,+}_{DB}$. Now, one can assume that the  s $h_{\varepsilon}$, which converge to $h$,  belong to $\IH^{p,+}_{DB}$ to begin with.  For fixed $\varepsilon>0$ and $\tau>0$, by definition of the extensions, 
$S_{p}^+(\tau)h_{\varepsilon} = e^{-\tau|DB|}h_{\varepsilon}= \wt S_{p}^+(\tau)h_{\varepsilon}$. 
If $\varepsilon\to 0$, the first term converges to $S_{p}^+(\tau) h $ in $H^{p,+}_{DB}$, while the last term converges to $\wt S_{p}^+(\tau) h$ in $\wt H^{p,+}_{DB}$. The equality follows again using the the embedding in part (2).  
\end{proof}

We come back to the solution $F_{t}$. Recall that 
 for $t>0$,  $F_{t}\in \IH^{2,+}_{DB}$ and  when   $\tau\ge 0$,  
$$
  e^{-\tau|DB|}F_{t} = F_{t+\tau}, 
  $$
 hence for all $k\in \N$, 
 $$ (-DB)^k F_{t}= \partial_{t}^kF_{t}.
  $$
  
  \begin{lem} $F_{t}$ belongs to $\wt \IH^{p,+}_{DB}$ uniformly in $t>0$ and converges to some $h \in \wt H^{p,+}_{DB}$ as $t\to 0$, so that $F_{t}= \wt S_{p}^+(t)h$ for all $t>0$. 
\end{lem}

\begin{proof} We are going to use another feature of the spaces $\wt \IH^p_{DB}$, which is the possibility of changing the norm (exactly as with the Hardy spaces $\IH^p_{DB}$ given by square functions). Indeed, if $\psi, \tilde \psi\in \Psi_{0}^\tau(S_{{\mu}})$ with $\tau>\alpha$, and $\psi$ is non-degenerate on $S_{\mu}$, then for all $h\in \IH^2_{DB}$, 
\begin{equation}
\label{eq:b}
\bigg(\int_{0}^\infty  \|\tilde \psi(tDB)h\|_{r}^r\,  t^{\alpha r}\, \frac{dt}{t}\bigg)^{1/r} \lesssim_{\psi,\tilde \psi} \bigg(\int_{0}^\infty  \| \psi(tDB)h\|_{r}^r\,  t^{\alpha r}\, \frac{dt}{t}\bigg)^{1/r}.
\end{equation}
The proof is roughly the same as the one of \eqref{eq:embed} but staying entirely within the functional calculus for $DB$. As $\psi$ is non-degenerate, there exists $\theta\in \Psi_{1}^1(S_{\mu})$ such that the Calder\'on reproducing formula 
$$
\int_{0}^\infty \theta(sDB)\psi(sDB) h\, \frac{ds}{s}= h
$$
holds with convergence in $\IH^2_{DB}$, hence
$$\tilde \psi(tDB)h= \int_{0}^\infty \tilde \psi(tDB)\theta(sDB) \psi(sDB)h\, \frac{ds}{s}$$
with convergence in $\IH^2_{DB}$. Now, by \cite{AS}, Theorem 4.19 and functional calculus for $DB$ we have  
$$
t^\alpha \|\tilde \psi(tDB)\theta(sDB) \psi(sDB)h\|_{r} \lesssim   s^\alpha g(s/t) \|\psi(sDB)h\|_{r}
$$
with $g(u)= \inf (u^{-\alpha}, u^{\tau-\alpha})$. The details are similar to the ones above and we skip them. We may apply this to  $\psi(z)= z^k e^{-\modz}$ for any $k\in \N$.  For $k=0$, we recover the defining norm. But here we pick $k$ with $(k+\alpha)r-1>0$ and it gives an equivalent norm. Thus 
we have for fixed $t>0$
\begin{align*}
 \|F_{t}\|_{\wt \IH^{p,+}_{DB}}^r    &  \sim  \int_{0}^\infty  \|\tau^k (DB)^k e^{-\tau|DB|} F_{t}\|_{r}^r\,  \tau^{\alpha r}\, \frac{d\tau}{\tau} \\
    &  = \int_{0}^\infty \| (DB)^k F_{t+\tau} \|_{r}^r \tau^{(k+\alpha)r-1}\, d\tau\\
    & \le \int_{t}^\infty \| (DB)^k F_{\tau} \|_{r}^r (\tau-t)^{(k+\alpha)r-1}\, d\tau\\
    & \le  \int_{0}^\infty \| (DB)^k F_{\tau} \|_{r}^r \tau^{(k+\alpha)r-1}\, d\tau\\
    & =  \int_{0}^\infty \| \tau^k \pd_{\tau}^kF_{\tau} \|_{r}^r \tau^{\alpha r-1}\, d\tau\\
    &\lesssim  \|\tN(\tau^k\partial_{\tau}^kF)\|_{p}^r \\
      &\lesssim  \|\tN F\|_{p}^r.
\end{align*}
We used the change of variable $t+\tau\to \tau$ and $(k+\alpha)r-1>0$ in the fourth line,   Lemma \ref{lem:HMiMo} in the next to last inequality and Corollary \ref{cor:repeatCaccio} in the last inequality. 

Next, for $0<t'<t\le \delta $,  we wish to show that  $\|F_{t}-F_{t'}\|_{\wt \IH^{p,+}_{DB}}$ tends to 0, 
which will imply the existence of the limit    in the completion. First, by Minkowski inequality in $L^r$ and a  computation as above
$$
 \bigg(\int_{0} ^\delta    \|\tau^k (DB)^k e^{-\tau|DB|}(F_{t}-F_{t'})\|_{r}^r\, \tau^{\alpha r-1}\, d\tau\bigg)^{1/r} \le  2 \bigg(\int_{0} ^{2\delta}   \|\tau^k (DB)^k F_{\tau}\|_{r}^r\, \tau^{\alpha r-1} d\tau\bigg)^{1/r} .
 $$
 Secondly,  for $\tau\ge \delta $, by the mean value inequality 
 \begin{align*}
  \|(DB)^k e^{-\tau|DB|}(F_{t}-F_{t'})\|_{r}  & = \| \pd_{\tau}^k( F_{t+\tau}-F_{t'+\tau})\|_{r} \\
  & \le \int_{t'+\tau}^{t+\tau  } \|\partial_{s}^{k+1}F_{s}\|_{r}\, ds  \\
    &  \le \bigg(\int_{t'+\tau}^{t+\tau  } \|s^{k+1}\partial_{s}^{k+1}F_{s}\|_{r}^r\, ds\bigg)^{1/r} |t-t'|^{1/r'}\tau ^{-(k+1)}
    \\
    & \le \bigg(\int_{\tau}^{2\tau  } \|s^{k+1}\partial_{s}^{k+1}F_{s}\|_{r}^r\, ds\bigg)^{1/r} |t-t'|^{1/r'}\tau ^{-(k+1)}.
\end{align*}
Thus, as $\tau\sim s$, changing the order of integration, we get 
\begin{align*}
   \int_\delta^\infty  \|\tau^k (DB)^k e^{-\tau|DB|}(F_{t}-F_{t'})\|_{r}^r\, \tau^{\alpha r-1}\, d\tau & \le |t-t'| ^{r/r'}   \int_{\delta}^\infty \|s^{k+1}\partial_{s}^{k+1}F_{s}\|_{r}^r \,  \frac{s^{\alpha r-1}}{s^{r-1}} ds    \\
    & \lesssim \left(\frac{ |t-t'|}{\delta}\right)^{r/r'} \|\tN(F)\|_{p}^r
    \end{align*}
    arguing as above  in the last inequality.  This gives the desired limit 0 of $\|F_{t}-F_{t'}\|_{\wt \IH^{p,+}_{DB}}$ when $t,t'\to 0$. 
    
    Let $h$ be the limit in   $\wt H^{p,+}_{DB}$ of $F_{t}$ as $t\to 0$.      As $F_{t+\tau}=e^{-t|DB|}F_{\tau}= \wt S_{p}^+(t)F_{\tau}$, taking the strong limit in $\wt H^{p,+}_{DB}$ as $\tau\to 0$ for fixed $t>0$ yields $F_{t}= \wt S_{p}^+(t) h$ for all $t>0$.    \end{proof}
    
    It remains to show that  $ h\in H^{p}_{DB}$ to conclude that $F_{t}=S_{p}^+(t) h$ for all $t>0$ by (3) in Lemma \ref{lem:htilde}, which finishes the proof of (i) implies (iii) in this case.  To do this we follow an idea  in \cite{HMiMo}. 
    
\begin{lem} If  $h$ is the limit in   $\wt H^{p,+}_{DB}$ of $F_{t}$ as $t\to 0$, we have $h\in H^p$ with $\|h\|_{H^p} \lesssim \|\tN F\|_{p}$, and also $h\in H^p_{D}$, which is the same as $h\in H^p_{DB}$. 
\end{lem}    
    
    \begin{proof}  Let $\varphi_{0} \in \mS$.  Let  $\chi$ be a real, $C^1$ function with compact support in $[0,\infty)$ and $\chi(0)=1$. Let $\varphi(s,y)=\varphi_{s}(y)= \chi(s)\varphi_{0}(y)$ for $s\ge 0$ and $y\in \R^n$.
    Observe that $\varphi \in C^1([0,\infty);\mS)$. 
 We  have 
 $\pair h{ \varphi_{0}}= \lim_{\tau\to 0} \pair {F_{\tau}} { \varphi_{\tau}}$. Indeed,
 $$
 \pair h{ \varphi_{0}} - \pair {F_{\tau} }{ \varphi_{\tau}}= \pair {h-F_{\tau}}{\varphi_{0}}+ (1-\chi(\tau)) \pair {F_{\tau}}{ \varphi_{0}}
 $$ and strong convergence of $F_{\tau}$ to $h$ in $\wt H^{p,+}_{DB}$ implies convergence in $\mS'$. 
 Next, 
  the pairing $\pair {F_{\tau}} { \varphi_{\tau}}$ is now taken as the $L^2$ pairing and since $F \in C^1((0,\infty); L^2)$ and $\varphi \in C^1((0,\infty); L^2)$ with bounded support in $s$, we have
 $$
 \pair {F_{\tau}}{ \varphi_{\tau}} =  \int_{\tau}^\infty (- \pair {F_{s}}{\pd_{s}\varphi_{s}} - \pair {\pd_{s}F_{s}}{\varphi_{s}}) \, ds.
 $$
 Using $- \pair {\pd_{s}F_{s}}{\varphi_{s}}=  \pair {DB F_{s}}{\varphi_{s}}= \pair {F_{s}}{B^*D\varphi_{s}}$ and taking the limit as $\tau\to 0$, we obtain that 
 $$ \pair h{ \varphi_{0}}= - \int_{0}^\infty  \pair {F_{s}}{\pd_{s}\varphi_{s}}\, ds + \int_{0}^\infty \pair {F_{s}}{B^*D\varphi_{s}} \, ds= I+II.
$$
The pairings inside the integrals are Lebesgue integrals on $ \R^n$. The convergence of the $s$ integral at 0 is in the sense prescribed above. However, note that we have put the action of $D$ on $\varphi$ in the process and we shall see that this way,  for some choices of $\varphi$, we obtain  \textit{bona fide} Lebesgue integrals on $\reu$ as the next argument shows.

Now choose $\varphi_{0}(y)=  \frac{1}{r^n}\phi(\frac{x-y}{r})$ with $\phi \in C^\infty_{0}$ supported in the ball $B(0,c_{1})$ with mean value 1. We have that $ \pair {h}{\varphi_{0}} = { h\star\frac{1}{r^n}\phi(\frac{.}{r})(x)}$. By the Fefferman-Stein characterisation of $H^p$, we  need to control $\sup_{r>0} |h\star \frac{1}{r^n} \phi(\frac{.}{r})|$ in $L^p$ to conclude that the Schwartz distribution $h$ belongs to $H^p$.  We use for that the integral representation of $\pair {h}{\varphi_{0}}$ above in which we take $\chi(t)$  supported in  $[0, c_{0}r)$ with value 1 on  
$[0, c_{0}^{-1}r]$ and $\|\chi\|_{\infty}+r\|\chi'\|_{\infty}\lesssim 1$. Note that the integrand of $I$  is supported in the Whitney box $W(r, x)$, so that looking at powers of $r$ and applying Cauchy-Schwarz inequality,  this  integral is dominated by $(\tN F)(x)$. For $II$,  using the boundedness of $B$, we obtain 
$$
|  II  | \lesssim  \iint_{T} |F| \|\nabla_{y} \varphi\|_{\infty} \lesssim r^{-n-1} \iint_{T} |F| ,
$$
where $T:=(0, c_{0}r)\times B(x,c_{1}r)$. Then,  using the inequality in Lemma \ref{lem:HMiMo}
$$
\iint_{\reu} |u| \lesssim \| \tN u\|_{{\frac{n}{n+1}}}
$$
with $u=|F|1_{T}$ and by support considerations, we obtain 
$$
r^{-n-1} \iint_{T} |F| \lesssim \left( r^{-n} \int_{(1+c_{0})B(x,c_{1}r)} (\tN F)^{\frac{n}{n+1}} \right)
^{\frac{n+1}{n}} \lesssim  (\MM ((\tN F)^{\frac{n}{n+1}}))^{\frac{n+1}{n}}(x),$$
where $\MM$ is the Hardy-Littlewood maximal operator. As $ \frac{n}{n+1}<p$, we obtain the conclusion  from the maximal theorem and $\tN F \in L^p$.

It remains to prove $h\in H^p_{D}$. Assume  $\varphi_{0}\in \mS$ is such that $D\varphi_{0}=0$.    Consider the extension $\varphi(s,y)= \chi(s)\varphi_{0}(y)$.  Then $D \pd_{s}\varphi_{s}=0$ and  $\pair {F_{s}}{\pd_{s}\varphi_{s}}=0$ because $F_{s}$ is orthogonal to the null space of $D$. Also $D\varphi_{s}=0$ and $\pair {F_{s}}{B^*D\varphi_{s}} =0$. It follows that $\pair h {\varphi_{0}}=0$ by the representation above. As $ h\in H^p$, this means that $h \in H^p_{D}$. 
\end{proof}

\

\subparagraph{\textbf{Case $\mathbf{p> 2}$.}}  

 Here $p>2$ means that $2<p<p_{+}(DB)$ since we impose $H^p_{DB}=H^p_{D}$. For $p$ is this range, $H^{p'}_{B^*D}=\clos{\ran_{p'}(B^*D)}=B^*\clos{\ran_{p'}(D)}$ is thus a closed subspace of $L^{p'}$, so that it is equipped with $L^{p'}$ norm and for all $h\in H^{p'}_{B^*D}$, 
 $$
 \|\IP h\|_{p'}\sim \|h\|_{p'} \sim \|h\|_{\IH^{p'}_{B^*D}}.
 $$
  Recall also that $\IH^{p'}_{B^*D}= H^{p'}_{B^*D}\cap H^{2}_{B^*D}$ and in this space we are able to compute without thinking about completions.

Recall that our goal is to interpret the limits in \eqref{eq:limit-(i)} and \eqref{eq:limit+(i)}.
Here,   we do not \textit{a priori} know that $F_{t}$ belongs to some $L^p$ space but only that  $F_{t}\in E^p_{t}$ uniformly.  In fact, we could suppose that $F_{t}$ belongs to $L^p$ for some $p<p_{0}$ with $p_{0}>2$. This follows from Meyers $W^{1,p}$ inequality for weak solutions. Thus the argument of subcase $p\le 2$ would carry almost without change for $p<p_{0}$. However, we do not know the relation between $p_{0}$ and $p_{+}(DB)$ so that we would have to work in the range $2<p< \inf(p_{0}, p_{+}(DB))$. We decide not to do this, in order to obtain the full range up to $p_{+}(DB)$. 
    We shall use the slice-spaces $E^p_{t}$ more extensively.

We shall rely on three technical lemmas. 

\begin{lem}\label{lemma1} Let $h\in \IH^{p'}_{B^*D}$. For all $\delta >0$, 
$e^{-\delta |B^*D|}h\in E^{p'}_{\delta }$ with uniform bound with respect to $\delta $. More precisely, 
$\sup_{\delta >0}\|e^{-\delta |B^*D|}h\|_{E^{p'}_{\delta }}\lesssim \|h\|_{p'}$. 
\end{lem}

\begin{proof}  Theorem 9.3 of \cite{AS} gives us the non-tangential maximal estimates $$ \|\tN(e^{-t|B^*D|} h) \|_{p'} \sim \|h\|_{p'}$$ for $p'$ in our range.  By Theorem 5.7 in \cite{AS} we have the $T^{p'}_{2}$ estimate, 
$$ \|t|B^*D|e^{-t|B^*D|} h \|_{T^{p'}_{2}} \sim \|h\|_{p'},$$ hence
 $$
 \|\tN(t|B^*D|e^{-t|B^*D|} h) \|_{p'} \lesssim \|h\|_{p'}.
 $$
As   $t|B^*D|e^{-t|B^*D|} h=-t\pd_{t}e^{-t|B^*D|} h$, using a similar argument via a mean value inequality as in the proof of Lemma \ref{lem:regular}, we obtain the desired uniform estimates $e^{-\delta |B^*D|}h\in E^{p'}_{\delta }$.
\end{proof}

\begin{lem}\label{lemma2} Fix $\delta >0$. Let $h\in \IH^{p',\pm}_{B^*D}\cap E^{p'}_{\delta }$.  Then $e^{-s |B^*D|}h$ converges to $h$ in $E^{p'}_{\delta }$ when $s\to 0$. 
\end{lem}

\begin{proof} See Section \ref{sec:technical}.

\end{proof}

\begin{lem}\label{lemma3} Let  $\varphi_{0}\in \mS$ and  $\phi_{0}=\IP_{B^*D}\varphi_{0}$ where the projection $\IP_{B^*D}$ was defined in Section \ref{sec:basic}. Then  
$\phi_{0}\in \IH^{p'}_{B^*D}$, $B^*D\phi_{0}=B^*D\varphi_{0} \in \IH^{p'}_{B^*D}$ and    $\chi^\pm(B^*D)B^*D\phi_{0}\in  \IH^{p',\pm}_{B^*D}\cap E^{p'}_{\delta }$ for all $\delta >0$.
\end{lem}

\begin{proof}
See Section \ref{sec:technical}.
\end{proof}

We begin the argument. Let  $\phi_{0}$ satisfy 
$\phi_{0}, B^*D\phi_{0} \in \IH^{p'}_{B^*D}$ and   $\chi^\pm(B^*D)B^*D\phi_{0}\in  \IH^{p',\pm}_{B^*D}\cap E^{p'}_{\delta }$ for all $\delta >0$.   From $\phi_{0}\in \IH^{p'}_{B^*D}$, we
 have the equalities \eqref{eq:limit-(i)} and \eqref{eq:limit+(i)} and we restart from those.  Recall that one term in 
    \eqref{eq:limit-(i)} tends to 0 as $\varepsilon\to 0$. Thus we need to calculate the limit of the other term in \eqref{eq:limit-(i)} and take the  limit in   \eqref{eq:limit+(i)}. Our first goal is to obtain some identities on $\partial_{t}F_{t}$.

    We begin with computing the limit of the first term in    \eqref{eq:limit-(i)}. 
   As  $\phi_{0}, B^*D\phi_{0}  \in \IH^{p'}_{B^*D}$, for $s>0$ we have  $B^*De^{sB^*D}\chi^-(B^*D)\phi_{0}=e^{sB^*D}\chi^-(B^*D)B^*D\phi_{0}$. It follows from Lemmas  \ref{lemma1} and \ref{lemma2} that  $e^{sB^*D}\chi^-(B^*D)B^*D\phi_{0} \in E^{p'}_{t}$ and converges to $\chi^-(B^*D)B^*D\phi_{0}=B^*D\chi^-(B^*D)\phi_{0} $ in this space as $s\to 0$. Since 
  $s\mapsto  F_{t+s}$ is continuous near $s=0$    into  $E^p_{t}$  by Corollary \ref{cor:uniformept}, we obtain 
    \begin{align*}
\frac{1}{\varepsilon} \iint_{[\varepsilon, 2\varepsilon]\times \R^n} & \bpaire { (B^*De^{sB^*D}\chi^-(B^*D)\phi_{0})(x)}{F(t+s,x)}\,   dsdx \\ 
&= \barint_{[\varepsilon, 2\varepsilon]} \pair { B^*De^{s B^*D}\chi^-(B^*D)  \phi_0}{F_{t+s}}\,   ds
\\
& \to  \pair { B^*D\chi^-(B^*D)  \phi_0}{F_{t}},  \quad \varepsilon\to 0,
\end{align*}
the pairings denoting the $E^{p'}_{t},E^p_{t}$ duality. This is in fact a Lebesgue integral. 
It follows from \eqref{eq:limit--} that 
\begin{equation}
\label{eq:a1}
 \pair { B^*D\chi^-(B^*D)  \phi_0}{F_{t}}= 0
\end{equation}
 for all such $\phi_{0}$ and $t>0$. 
 
 Similarly the first term in  \eqref{eq:limit+(i)} converges to $ \pair { B^*D\chi^+(B^*D)  \phi_0}{F_{t}}$ for all such $\phi_{0}$ and $t>0$. We set
 \begin{align*}
   I_{t,\phi_{0}}^\varepsilon&:=\frac{1}{\varepsilon} \iint_{[{\varepsilon}, 2{\varepsilon}]\times \R^n} \bpaire { B^*(x)D\varphi_s(x)}{F(s,x)} \, dsdx\\ 
    &= \barint_{[\varepsilon, 2\varepsilon]} \pair { B^*De^{ -(t-s) B^*D}\chi^+(B^*D)  \phi_0}{F_{s}}\,   ds,
    \end{align*}
where again the pairing can be interpreted using the $E^{p'}_{t},E^p_{t}$ duality.  Thus we have 
 shown, 
$$
\lim_{\varepsilon\to 0} I_{t,\phi_{0}}^{\varepsilon }= \pair { B^*D\chi^+(B^*D)  \phi_0}{F_{t}}
$$
for all $t >0$. For $\tau\ge 0$,  replacing $t$ by $t+\tau$, we obtain
$$
\lim_{\varepsilon\to 0} I_{t+\tau,\phi_{0}}^{\varepsilon } = \pair { B^*D\chi^+(B^*D)  \phi_0}{F_{t+\tau}}.$$ But at the same time, $I_{t+\tau,\phi_{0}}^{\varepsilon } =I_{t, \phi_{\tau}}^\varepsilon$ with   $\phi_{\tau}=e^{-\tau|B^*D|}\phi_{0}$. As  $\phi_{\tau}$ satisfies  the same requirements as $\phi_{0}$ and   
$B^*D\chi^+(B^*D)  \phi_\tau= B^*De^{-\tau B^*D}\chi^+(B^*D)  \phi_0$, we obtain
$$
\lim_{\varepsilon\to 0} I_{t+\tau, \phi_{0}}^{\varepsilon } =  \pair {B^*D \chi^+(B^*D)  \phi_\tau}{F_{t}}= \pair {B^*D e^{-\tau B^*D} \chi^+(B^*D)  \phi_0}{F_{t}}.$$
We have obtained the relation 
\begin{equation}
\label{eq:a2}
 \pair { B^*De^{-\tau B^*D}\chi^+(B^*D)  \phi_0}{F_{t}}=  \pair { B^*D\chi^+(B^*D)  \phi_0}{F_{t+\tau}}.
\end{equation}
 
 Summing \eqref{eq:a1}  at  $t+\tau$ and  \eqref{eq:a2}, we have  for all such $\phi_{0}$, $t>0$ and $\tau\ge 0$,
 \begin{equation}
\label{eq:a3}
 \pair { B^*D  \phi_0}{F_{t+\tau}}= \pair { B^*De^{-\tau B^*D}\chi^+(B^*D)  \phi_0}{F_{t}}.
\end{equation}

With these identities, we next show that $\partial_{t}F_{t}\in L^p$.  Let $\varphi_{0}\in \mS$. By Lemma \ref{lemma3}, the function $\phi_{0}=\IP_{B^*D}\varphi_{0}$ has the required properties to apply \eqref{eq:a3}. Moreover, $B^*D\varphi_{0}=B^*D\phi_{0}$. 
Using the integration by parts on both sides, justified by Lemma \ref{lem:sliceIBP}, and $DBF_{t}=- \pd_{t}F_{t}$, we conclude that
\begin{equation}
\label{eq:a4}
 \pair { \varphi_0}{\partial_{t }F_{t+\tau}}   = \pair { e^{-\tau B^*D}\chi^+(B^*D)  \phi_0}{\partial_{t }F_{t}}.
 \end{equation} 
 
 Now, using this equality with $\tau=t$ and using the $E^{p'}_{t}, E_{t}^p$ duality,  we have
 \begin{align*}
| \pair { \varphi_0}{\partial_{t }F_{2t}}|    & \le     \|e^{-t B^*D}\chi^+(B^*D)  \phi_0\|_{E^{p'}_{t}}\|\partial_{t }F_{t}\|_{E^p_{t}} \\
    &  \lesssim   \|\chi^+(B^*D)  \phi_0\|_{p'}\| \|\partial_{t }F_{t}\|_{E^p_{t}} \\
    & \lesssim  \|  \varphi_0\|_{p'} t^{-1}.
\end{align*}
In the second inequality, we used Theorem 9.3 of \cite{AS}. In the last inequality, we used that $\chi^+(B^*D)$ and $\IP_{B^*D}$ are bounded on $L^{p'}$ for $p_{-}(B^*D)<p'<2$ which is   our range here. For the first operator, this is the functional calculus on $\IH^{p'}_{B^*D}$ and for the second operator, this is because we have the kernel/range decomposition for $B^*D$ in $L^{p'}$ and $p'$ as above. We also used $\|t\partial_{t }F_{t}\|_{E^p_{t}}  \lesssim 
\|t\partial_{t}F\|_{N^p_{2}}\lesssim \|F\|_{N^p_{2}}<\infty$ by hypothesis. 
As the inequality above holds for any Schwartz function, this means that $t{\partial_{t }F_{2t}}$ belongs to $L^p$ with uniform norm with respect to $t$.

The next step is to prove the semigroup representation for $\partial_{t}F_{t}$.  It follows that $t{\partial_{t }F_{t}} \in \clos{\ran_{p}(D)}$ as $\pd_{t}F_{t}$ is a conormal gradient (hence satisfies the curl condition). 
Moreover, we can now use the extension $S_{p}^+(\tau)$  of  $e^{-\tau DB}\chi^+(DB)$ in $\clos{\ran_{p}(D)}=H^p_{DB}$ and for $\varphi_{0}\in \mS$, 
$$
\pair { e^{-\tau B^*D}\chi^+(B^*D)  \phi_0}{\partial_{t }F_{t}}= \pair {   \phi_0}{S_{p}^+(\tau)\partial_{t }F_{t}} = \pair {   \varphi_0}{S_{p}^+(\tau)\partial_{t }F_{t}}.$$
The last equality is because $\varphi_{0}-\phi_{0}$ belongs $\nul_{p'}(B^*D)=\nul_{p'}(D)$ which is the polar set  of $\clos{\ran_{p}(D)}$.
Thus, for all $t>0$ and $\tau\ge 0$, we have 
\begin{equation}
\label{eq:pd2}
\partial_{t }F_{t+\tau} = S_{p}^+(\tau)\partial_{t }F_{t}=S_{p}(\tau)\partial_{t }F_{t}.
\end{equation} 
The last equality is because we know that $ \partial_{t }F_{t} \in H^{p}_{DB}$ and we can deduce from \eqref{eq:a1} that $\partial_{t }F_{t} \in H^{p, +}_{DB}$,  $S_{p}(\tau)$ being the extension of $e^{-\tau |DB|}$ to $H^p_{DB}$.

It remains to integrate  and obtain the trace at $t=0$. 
To do this, we  introduce a distribution 
$G_{t,\tau}$ by defining for $\varphi_{0}\in \mS$,
\begin{equation}
\label{eq:gttau}
 \pair { \varphi_0}{G_{t,\tau}}=  \pair { \varphi_0}{F_{t+\tau}}   - \pair { e^{-\tau B^*D}\chi^+(B^*D)  \IP_{B^*D}\varphi_0}{F_{t}}.
 \end{equation}
 By the same argument as above, 
\begin{align*}
  |\pair { e^{-\tau B^*D}\chi^+(B^*D)  \IP_{B^*D}\varphi_0}{F_{t}}|  & \le    \|e^{-\tau B^*D}\chi^+(B^*D)   \IP_{B^*D}\varphi_0\|_{E^{p'}_{t}}\|F_{t}\|_{E^p_{t}}   \\
    &  \lesssim_{t,\tau}  \|  \varphi_0\|_{p'} \|\tN F\|_{p}.
\end{align*}
   Notice that the implicit constant is uniform in $t$ when $\tau=t$. It follows that there exists an element $f_{t,\tau}\in L^p$ such that for all $\varphi_{0}\in \mS$, 
\begin{equation}
\label{eq:fttau}
\pair {\varphi_{0}}{f_{t,\tau}}=\pair { e^{-\tau B^*D}\chi^+(B^*D)  \IP_{B^*D}\varphi_0}{F_{t}}.
\end{equation}
Both pairings are in fact integrals and the equality  extends to all $\varphi_{0}\in L^{p'}$ by density.  Taking $\varphi_{0}\in \nul_{p'}(D)$ shows  $f_{t,\tau}\in \clos{\ran_{p}(D)}$.  As $L^p\subset E^p_{t}$ for any $t$, we  have that $G_{t,\tau}$ is a well-defined element in $\mS'$ and  $G_{t,\tau}=F_{t+\tau}-f_{t,\tau} \in E^p_{t}$. 

We now show that $G_{t,\tau}$ is constant as a function of $t>0$ and $\tau>0$. It is quite clear using smoothness of $t\mapsto F_{t}$ in any fixed $E^p_{\delta }$ and of $\tau \mapsto e^{-\tau B^*D}\chi^+(B^*D)  \IP_{B^*D}\varphi_0$ in $E^{p'}_{\delta }$  that one can differentiate $\pair{\varphi_0}{G_{t,\tau}}$  in $t>0$ and $\tau>0$ when $\varphi_{0}\in \mS$ and  by \eqref{eq:gttau},
 $$
 \pair { \varphi_0}{\pd_{t}G_{t,\tau}}=  \pair { \varphi_0}{\pd_{t}F_{t+\tau}}   - \pair { e^{-\tau B^*D}\chi^+(B^*D)  \IP_{B^*D}\varphi_0}{\pd_{t}F_{t}}=0
 $$
 and
 $$
 \pair { \varphi_0}{\pd_{\tau}G_{t,\tau}}=  \pair { \varphi_0}{\pd_{\tau}F_{t+\tau}}   + \pair { B^*De^{-\tau B^*D}\chi^+(B^*D)  \IP_{B^*D}\varphi_0}{F_{t}}=0.
 $$
 We used again the integration by parts argument and $\pd_{t}F_{t}=-DBF_{t}$.
 We obtain $\partial_{t}G_{t,\tau}=\pd_{\tau}G_{t,\tau}=0$. 
 Let $G=G_{1,1}=G_{t,\tau}$.
  
 Let us show that $G\in L^p$ and then that  $G\in \clos{\ran_{p}(D)}$. 
 As observed,  $\|f_{t,t}\|_{p}$ is uniformly bounded in $t$.
  In particular, we have $f_{t,t} \in E^p_{t}$ with  $\sup_{t>0}\|f_{t,t}\|_{E^p_{t}}\lesssim \sup_{t>0}\|f_{t,t}\|_{L^p} \lesssim \|\tN F\|_{p}$. By taking the difference with $F_{2t}$, this implies that  
$$
\sup_{t>0} \left( \int_{\R^n} \bigg(\barint_{B(x,t)} |G(y)|^2\, dy\bigg)^{p/2}\, dx\right)^{1/p}< \infty.
$$
Applying Fatou's lemma when $t\to 0$ shows that $G\in L^p$. This implies that $F_{2t}=G+f_{t,t}$ belongs to $L^p$, hence $F_{2t}\in \clos{\ran_{p}(D)}$ (because it satisfies the curl condition) and it follows that $G\in  \clos{\ran_{p}(D)}$.
 
 Remark that as we know that $G\in L^p$ and $F_{t+\tau}\in L^p$, one can extend the definition of $\pair { \varphi_0}{G_{t,\tau}}$  to  any $\varphi_{0} \in L^{p'}$ by density because all pairings make sense.    Taking the derivative in $\tau$ in \eqref{eq:gttau} implies that 
\begin{multline}\label{eq:DBG=0}$$ 0= \pair { \varphi_0}{-DBF_{t+\tau}}   + \pair { B^*De^{-\tau B^*D}\chi^+(B^*D)  \IP_{B^*D}\varphi_0}{F_{t}} \\= - \pair { B^*D\varphi_0}{F_{t+\tau}}   + \pair { e^{-\tau B^*D}\chi^+(B^*D)  \IP_{B^*D}B^*D\varphi_0}{F_{t}}= -\pair {B^*D \varphi_0}{G_{t,\tau}}. $$
 \end{multline}
The last equality holds because $B^*D\varphi_{0}\in L^{p'}$ and we use the extension mentioned above.  Thus $DBG=0$ in the sense of Schwartz distributions.

 We have seen that $DBG=0$ in $\mS'$ and $G\in L^p$. Thus $G\in \nul_{p}({DB})$ and we conclude that $G=0$ from the splitting $L^p=\nul_{p}(DB)\oplus \clos{\ran_{p}(D)}$ which holds in our range of $p$. 
We can now write for $\varphi_{0}\in \mS$ using \eqref{eq:fttau} 
$$
\pair {\varphi_{0}}{f_{t,\tau}}=\pair { e^{-\tau B^*D}\chi^+(B^*D)  \IP_{B^*D}\varphi_0}{F_{t}} = \pair {   \IP_{B^*D} \varphi_0}{S_{p}^+(\tau)F_{t}} = \pair {   \varphi_0}{S_{p}^+(\tau)F_{t}},
$$
and as we have just shown that  $\pair {\varphi_{0}}{f_{t,\tau}}= \pair {\varphi_{0}}{F_{t+\tau}}$, we have obtained
$$
F_{t+\tau}= S_{p}^+(\tau)F_{t}
$$
in $\mS'$ for all $t>0, \tau> 0$  and as both terms are in $L^p$, this also holds in $L^p$.  Using the uniform bound on $F_{t}=f_{t/2,t/2}$ in $L^p$ we can use a weak limit argument as in previous cases to deduce the existence of $F_{0}\in H^{p,+}_{DB}$ such that $F_{\tau}=S_{p}^+(\tau)F_{0}$ for all $\tau>0$.  This concludes the proof of this case.

\section{Proof of Theorem \ref{thm:main1}: (ii) implies (iii)}

We assume (ii) and set $F=\nabla\!_{A}u$. Thus, $t\partial_{t} F\in T^p_{2}$,
and $F$ is a solution of  \eqref{eq:curlfreesystem} in $\reu$.  We also assume that $F_{t}$ converges to 0 in $\mD'$  as $t\to \infty$. 
Recall that $p\in I_{L}$, that is,  $H^p_{DB}=H^p_{D}$. We will use Theorem \ref{thm:IL} repeatedly. 

\

\paragraph{\textbf{Step 1}}   Finding the semigroup equation. 

\

This will be achieved  by  taking limits in  \eqref{eq:step1} with $\pd_{s}F_{s}$ replacing $F_{s}$ by selecting $\chi$, $\eta$ and $\phi_{0}$.

\

\paragraph{\textbf{Step 1a}} Limit in space.    We show  that if $\phi_{0}\in \clos{\ran_{2}(B^*D)}$, with $\phi_{0}\in \IH^{p'}_{B^*D}$ (or, equivalently, $\IP\phi_{0}\in 
\IH^{p'}_{D}$) if $p>1$, then \begin{equation}    \label{eq:step2(ii)}
  \iint_{\reu} \bpaire {\eta'(s) B^*(x)D\varphi_s(x)}{\partial_{s }F(s,x)}\,   dsdx = 0\end{equation} 
  and the integral is defined in the Lebesgue sense. 

\

We replace $\chi$ by $\chi_{R}$ with $\chi_{R}(x)=\chi(x/R)$ where $\chi\equiv 1$ in the unit ball $B(0,1)$, has compact support in the ball $B(0,2)$ and let $R\to \infty$.  As $\chi_{R}$ tends to 1 and $D_{\chi_{R}}$ to 0, it suffices by dominated convergence to show that $|\eta'(s)B^*D \varphi_s \partial_{s} F|$ and $|\eta(s)\partial_{s}\varphi_{s}\partial_{s}F|$ are integrable on $\reu$.  As $s\partial_{s}F\in T^p_{2}$, it is enough to have that $\eta'(s)B^*D \varphi_s$ and $\eta(s)\partial_{s}\varphi_{s}$ belong to $(T^p_{2})'$.  As $\partial_{s}\varphi_{s}=B^*D\varphi_{s}$ on $\supp \eta$, it suffices to invoke Lemma \ref{lem:step1} again.

\

\paragraph {\textbf{Step 1b}}  Limit in time. 

\

For fixed $t>0$,   $0<\varepsilon<\inf(t/4, 1/4, 1/t)$, we obtain by making the same choices of $\eta$ as in the proof of (i) implies (iii), 
\begin{align}\label{eq:limit-}
  \frac{1}{\varepsilon} \iint_{[t+\varepsilon, t+2\varepsilon]\times \R^n} &\bpaire { B^*(x)D\varphi_s(x)}{\partial_{s }F(s,x)}\,   dsdx \\ \nonumber & = 2\varepsilon \iint_{[t+ \frac{1}{2\varepsilon}, t+\frac{1}{\varepsilon}]\times \R^n} \bpaire { B^*(x)D\varphi_s(x)}{\partial_{s }F(s,x)} \, dsdx,
   \end{align}
and
\begin{align}\label{eq:limit+}
\frac{1}{\varepsilon} \iint_{[t-2\varepsilon, t-\varepsilon]\times \R^n} & \bpaire { B^*(x)D\varphi_s(x)}{\partial_{s }F(s,x)}\,   dsdx \\  \nonumber & =\frac{1}{\varepsilon} \iint_{[{\varepsilon}, 2{\varepsilon}]\times \R^n} \bpaire { B^*(x)D\varphi_s(x)}{\partial_{s }F(s,x)} \, dsdx.
   \end{align}

  We  show that the second integral in \eqref{eq:limit-} converges to 0 as $\varepsilon\to 0$ for fixed $t$.  Using the function $\psi_{-}$ defined earlier, set $G(s,x)= G_{s}(x)= 1_{[t+ \frac{1}{2\varepsilon}, t+\frac{1}{\varepsilon}]}(s)(s-t)^{-1} \psi_{-}((s-t)B^*D)  \phi_0(x)$, under the conditions on $\phi_{0}$ in Step 1.  Remark that this integral is bounded by  
  $$
  2\varepsilon \|s\partial_{s}F\|_{T^p_{2}}\|G\|_{(T^p_{2})'}.
  $$
  Assume first $p\le 1$ and let $\alpha=n(\frac{1}{p}-1)$.  As $t\varepsilon<1$ and $s\in [t+ \frac{1}{2\varepsilon}, t+\frac{1}{\varepsilon}]$, we have  $s\in [ \frac{1}{2\varepsilon}, \frac{2}{\varepsilon}]$. Since $\|G_{s}\|_{2}\lesssim \varepsilon\|\phi_{0}\|_{2}$ for those $s$, we have
  $\|C_{\alpha}(G)\|_{\infty}\lesssim \varepsilon^{\alpha+\frac{n}{2}+1}\|\phi_{0}\|_{2}$. 
  
  Next, consider $p>1$. Then one sees that $\|G\|_{T^{p'}_{2}} \lesssim \varepsilon \|\psi_{-}(\sigma B^*D)  \phi_0\|_{T^{p'}_{2}} \lesssim \varepsilon \|\IP \phi_{0}\|_{p'}$. 
  
  From now on,  we distinguish the case $p\le 2$ from $p>2$. 
  
  \

\subparagraph{\textbf{Case $\mathbf{p\le 2}$.}}

Using Lemma \ref{lem:q} with $\pd_{s}F$ instead of $F$, we have $\pd_{s}F \in C^\infty(0,\infty; L^q)$ and we can reinterpret the $dx$-integrals in \eqref{eq:limit-} and \eqref{eq:limit+} with the $L^{q'}-L^q$ duality. Copying  \textit{mutatis mutandi} the argument in the proof of (i) implies (iii), we can take the limit as $\varepsilon\to 0$ and obtain the following information: for all $t>0$, $\pd_{t}F_{t}\in \IH^{2,+}_{DB}$,  and if $\tau\ge 0$
\begin{equation}
\label{eq:sgeq}
 e^{-\tau |DB|}\pd_{t}F_{t} =  e^{-\tau DB}\chi^+(DB) \pd_{t}F_{t} = \pd_{t}F_{t+\tau}.
\end{equation}
Note that this equation can be differentiated as many times as we want in both $t$ and $\tau$. 
More information can be obtained such as  ${t}F_{t}\in \IH^{q,+}_{DB}$ but we do not need this.

 \
 
 The argument used in section 8 was to integrate from \eqref{eq:sgeq},  but it does not work the same here. Instead we look for a candidate $f_{t}$ for $F_{t}$ via Hardy space theory. We first prove that $\partial_{t}f_{t}=\partial_{t}F_{t}$ and then integrate and conclude. 
 
 \
 
 \paragraph{\textbf{Step 2}}   Defining an auxiliary function  $f_{t}$.  
 
 \
 
We begin with an observation, valid whatever $p\in (0,\infty)$.

 \begin{lem} \label{lem:obs} For each $t>0$ and $N>\frac{n+1}{2}$, we have $(s,x)\mapsto s^N\pd_{s}^NF_{t+s}(x)$ belongs to $T^p_{2}$ with uniform bound with respect to $t$. Moreover, it is $C^\infty$ as a  function of $t$ valued in $T^p_{2}$.  
 If $t\to 0$, then it converges to $(s,x)\mapsto s^N\pd_{s}^NF_{s}(x)$ in $T^p_{2}$ and, if $t\to \infty$, it converges to $0$ in $T^p_{2}$.
\end{lem}

\begin{proof}
Using Corollary \ref{cor:repeatCaccio}, we have $\|s^N\pd_{s}^NF\|_{T^p_{2}} \lesssim \|s\pd_{s}F\|_{T^p_{2}}<\infty$. For fixed $t>0$ and $x\in \R^n$ (recall that $\Gamma(x)$ denotes a cone with aperture 1 and vertex $x$), 
\begin{align*}\label{}
     \iint_{\Gamma(x)} s^{2N-n-1}|\pd_{s}^NF_{t+s}(y)|^2 \, {dsdy} 
    & \le \iint_{\Gamma(x)} (s+t)^{2N-n-1}|\pd_{s}^NF_{t+s}(y)|^2 \, {dsdy}\\
    & = \iint_{\Gamma(x)+(t,0)} s^{2N-n-1}|\pd_{s}^NF_{s}(y)|^2 \, {dsdy}.
\end{align*}
We used that $2N-n-1>0$ in the first inequality and the change of variable $s+t\to s$. Using the containment  $\Gamma(x)+(t,0) \subset \Gamma(x)$, we have  
$$\|s^N\pd_{s}^NF_{t+s}\|_{T^p_{2}}\le \|s^N\pd_{s}^NF_{s}\|_{T^p_{2}}.
$$
The same argument plus dominated convergence  shows that
$\|s^N\pd_{s}^NF_{t+s}\|_{T^p_{2}} \to 0$ when $t\to \infty$. 

Let us look at the  limit at 0. 
Let  $\Gamma_{\delta}(x)$ be the  truncation of $\Gamma(x)$ for $s\le\delta $ and $\Gamma^{\delta}(x)= \Gamma(x)\setminus \Gamma_{\delta}(x)$. Arguing as before with a crude estimate and assuming $t \le \delta$, we have 
$$
\iint_{\Gamma_{\delta }(x)} s^{2N-n-1}|\pd_{s}^N(F_{t+s}-F_{s})(y)|^2 \, {dsdy}\le 2 \iint_{\Gamma_{2\delta }(x)} s^{2N-n-1}|\pd_{s}^NF_{s}(y)|^2 \, {dsdy}.$$
Next, by    the mean value inequality and $s\ge \delta $, 
$$
\iint_{\Gamma^{\delta }(x)} s^{2N-n-1}|\pd_{s}^N(F_{t+s}-F_{s})(y)|^2 \, {dsdy} \le \frac{t^2}{\delta ^2} \iint_{\Gamma^{\delta }(x)} s^{2N+2-n-1}|\pd_{s}^{N+1}F_{s}(y)|^2 \, {dsdy}.
$$
 By dominated convergence, the  integral on $\Gamma_{2\delta }(x)$ goes to 0 in $L^{p/2}$ when $\delta \to 0$. Next, having chosen $\delta $ small, one makes the integral on $\Gamma^\delta  (x)$ tend to 0 if $t\to 0$.  This proves the limit at 0.
 
  The $C^\infty$ smoothness at any point $t>0$ can be proved by iterating the second part of the argument. We skip details. 
\end{proof}

We now use the Hardy space theory associated to $DB$. Let $N$ be as above and $M$ another integer also chosen large. Let  $c$ be some constant chosen later. We claim that  for all $t\ge 0$ there is an element $f_{t}$ in $H^{p,+}_{DB}$ such that  for all $\phi_{0}\in \IH^{p'}_{B^*D}$ when $1<p\le 2$ or $\phi_{0}\in \IL^\alpha_{B^*D}$ when $p\le 1$ (We are using the notation of \cite[Section 4.2]{AS}: see in particular Corollary 4.4.)
\begin{align}
\label{eq:ft}
\pair{\phi_{0}}{f_{t}}& =c\iint_{\reu} \big( {\tau^M (B^*D)^Me^{-\frac{\tau}{2} |B^*D|}\phi_{0}(x)}\cdot ({\tau^N\pd_{\tau}^NF_{t+\frac{\tau}{2}})(x)}\big)\, \frac{d\tau dx}{\tau}
\\
\nonumber & = c\int_{0}^\infty \pair {\phi_{0}}{\tau^M(DB)^M e^{-\frac{\tau}{2} |DB|}\tau^N\pd_{\tau}^N F_{t+\frac{\tau}{2}}}\, \frac{d\tau}{\tau}
\end{align}
and
$$
f_{t}=  S_{p}^+(t)f_{0}= S_{p}(t)f_{0}.
 $$
 Here the expression $\pd_{\tau}^N F_{t+\frac{\tau}{2}}$ means that we differentiate $N$ times with respect to $\tau$ the function $F_{t+\frac{\tau}{2}}$.
 We recall that $S_{p}^+(t)$ is the continuous extension of $e^{-t DB}\chi^+(DB)$ to 
 $H^p_{DB}$ and it agrees with the continuous extension $S_{p}(t)$ of $e^{-t |DB|}$ on $H^{p,+}_{DB}$. Both integrals converge  in the Lebesgue sense (one is on $\R_{+}$, the other one on $\reu$).

Let us define  $f_{t}\in H^p_{DB}$ for each $t\ge 0$. Truncating $\tau^N\pd_{\tau}^NF_{t+\frac{\tau}{2}}$ with $\chi_{k}$ being the indicator of $[\frac{1}{k}, k]\times B(0,k)$, we have an element  of $T^2_{2}\cap T^p_{2}$. If $M>|\frac{n}{p}-\frac{n}{2}|$ as $p\le 2$,  then \cite[Corollary 4.4]{AS} for $DB$ with the map $\Tpsi {\varphi}{ DB}$ and $\varphi(z)= cz^{M}e^{-\frac{1}{2}\modz}$  shows that 
\begin{align*}
   f_{t,k}&=c\int_{0}^\infty \tau^M (DB)^Me^{-\frac{\tau}{2} |DB|}(\chi_{k} \tau^N\pd_{\tau}^NF_{t+\frac{\tau}{2}})\, \frac{d\tau}{\tau}    \\
      &=\Tpsi{\varphi}{DB}(\chi_{k} \tau^N\pd_{\tau}^NF_{t+\frac{\tau}{2}})
\end{align*}
belongs to $\IH^p_{DB}$. Clearly, $(f_{t,k})$ is a Cauchy sequence in $\IH^p_{DB}$ and, as $k\to \infty$,  it converges to some element in $H^p_{DB}$. This defines $f_{t}$. 

Let us prove the integral formulae \eqref{eq:ft}.  This is in fact the dual argument to the one above. 
 By the $L^2$ theory,    for all integers $k$,  we have 
\begin{align*}
\pair{\phi_{0}}{f_{t,k}}&=c\iint_{\reu} \big( {\tau^M (B^*D)^Me^{-\frac{\tau}{2} |B^*D|}\phi_{0}(x)}\cdot ({\chi_{k}\tau^N\pd_{\tau}^NF_{t+\frac{\tau}{2}})(x)}\big)\, \frac{d\tau dx}{\tau} .
\end{align*} 
As $M$ is large enough, we may apply Theorem 5.7 in \cite{AS} to conclude that $(\tau,x)\mapsto  \tau^M (B^*D)^M e^{-\frac{\tau}{2} |B^*D|}\phi_{0}(x) \in (T^{p}_{2})'$ with  $(T^{p}_{2})'$ norm bounded by $\|\IP\phi_{0}\|_{p'}$ when $p>1$ and $\|\IP\phi_{0}\|_{\dot \Lambda^\alpha}$ and $\alpha= n(\frac{1}{p}-1)$ when $p\le 1$. Then we use Theorem 4.20 in \cite{AS} to see that 
$\|\IP\phi_{0}\|_{p'}\sim \|\phi_{0}\|_{\IH^{p'}_{B^*D}}$ and $\|\IP\phi_{0}\|_{\dot \Lambda^\alpha}\sim \|\phi_{0}\|_{\IL^\alpha_{B^*D}}$ in the allowed range of $p$. 
Thus the above integral on $\reu$  converges (absolutely)  as $k\to \infty$ by dominated  convergence and we obtain the first line in \eqref{eq:ft}.   To see the  second line, we use Fubini's theorem  and  rewrite the $dx$ integral as a pairing for the $L^2$ duality for fixed $\tau$ because both entries in the written pairing are $L^2$ functions.  
Thus we also obtain
$$
\pair{\phi_{0}}{f_{t}} = c\int_{0}^\infty \pair {\tau^M (B^*D)^Me^{-\frac{\tau}{2} |B^*D|}\phi_{0}}{\tau^N\pd_{\tau}^NF_{t+\frac{\tau}{2}}}\, \frac{d\tau}{\tau}
$$
which is the same thing as the second line of \eqref{eq:ft} by taking adjoints (again for each $\tau$, only the $L^2$ theory is used). 
This means that we have proved the formulae \eqref{eq:ft} for such $\phi_{0}$. 

This will allow us to prove the semigroup formula for $f_{t}$. Let us see that now. 
 Differentiating \eqref{eq:sgeq} (reverse the roles of $t$ and $\tau$) shows that as $N\ge 1$,  
$$
\tau^N\pd_{\tau}^NF_{t+\frac{\tau}{2}}= e^{-t DB}\chi^+(DB)\tau^N\pd_{\tau}^NF_{\frac{\tau}{2}}.
$$
Thus, we obtain 
$$
\pair {\tau^M (B^*D)^Me^{-\frac{\tau}{2} |B^*D|}\phi_{0}}{\tau^N\pd_{\tau}^NF_{t+\frac{\tau}{2}}}= \pair {\tau^M (B^*D)^Me^{-\frac{\tau}{2} |B^*D|}\phi_{t}}{\tau^N\pd_{\tau}^NF_{\frac{\tau}{2}}}
$$
with $\phi_{t}= e^{-tB^*D}\chi^+(B^*D)\phi_{0}$.  As $\phi_{t}$ satisfies the same requirements as $\phi_{0}$  above, we can use the definition of $f_{t}$ and obtain
$$
\pair{\phi_{0}}{f_{t}}= \pair{\phi_{t}}{f_{0}}= \pair{e^{-tB^*D}\chi^+(B^*D)\phi_{0}}{f_{0}}= 
\pair{\phi_{0}}{S_{p}^+(t)f_{0}}.
$$
Note that by   
Theorem \ref{thm:IL} and $p\le 2$, this implies that $f_{t}\in \IH^2_{DB}=\clos{\ran_{2}(D)}$ for all $t>0$ (qualitatively). 

\

The next claim is that $\partial_{t}f_{t}= \partial_{t}F_{t}$ in $ \mS'$ upon choosing the constant $c$ appropriately. To do this,  we have to make the connection between   calculus in $\mS'$ and calculus  in $H^p_{DB}$ by showing that the second line of \eqref{eq:ft} holds for $\varphi_{0}\in \mS$ instead of $\phi_{0}$, still with  absolutely convergent integral in $\tau$.

Let  $\varphi_{0}\in \mS$. As $\mS\subset L^2$,  define $\phi_{0}= \IP_{B^*D}\varphi_{0} \in \IH^2_{B^*D}$. We have that $\IP\phi_{0}= \IP \varphi_{0}$ and since $\IP$ preserves $L^{p'}$ and $\dot \Lambda^\alpha$, it follows that $\phi_{0}$ has the required property  to compute $ \pair{\phi_{0}}{f_{t}}$ with the definition given above. 
Now, $\phi_{0}-\varphi_{0}\in \nul_{2}(D)$, thus by $L^2$ theory,
$$
 \pair{\varphi_{0}}{f_{t}}= \pair{\phi_{0}}{f_{t}}
 $$
 and also for each $\tau>0$, 
 $$
 \pair {\varphi_{0}}{\tau^M(DB)^M e^{-\frac{\tau}{2} |DB|}\tau^N\pd_{\tau}^N F_{t+\frac{\tau}{2}}}= \pair {\phi_{0}}{\tau^M(DB)^M e^{-\frac{\tau}{2} |DB|}\tau^N\pd_{\tau}^N F_{t+\frac{\tau}{2}}}.
 $$
 This proves our claim upon applying the second line of \eqref{eq:ft} for $\phi_{0}$. 

We next show how to select  the constant $c$ appropriately  to show our claim, that is  $ \pair{\varphi_{0}}{\pd_{t}f_{t}}=  \pair{\varphi_{0}}{\pd_{t}F_{t}}$ for all $\varphi_{0}\in \mS$ and all $t>0$. Indeed, by dominated convergence and Lemma \ref{lem:obs},  we can differentiate the first formula in \eqref{eq:ft} with respect to $t$ with $\phi_{0}= \IP_{B^*D}\varphi_{0}$ as above. With the same arguments, we obtain the second line replacing $F_{t+\frac{\tau}{2}}$ by $\pd_{t}F_{t+\frac{\tau}{2}}$.  After this, we reexpress the inner product inside in terms of  $\varphi_{0}$, so that  we have obtained 
$$ \pair{\varphi_{0}}{\pd_{t}f_{t}} = c\int_{0}^\infty \pair {\varphi_{0}}{\tau^M(DB)^M e^{-\frac{\tau}{2} |DB|}\tau^N\pd_{\tau}^N \pd_{t}F_{t+\frac{\tau}{2}}}\, \frac{d\tau}{\tau}.
 $$
 Now, again with the $L^2$ theory, 
\begin{multline*}  $$
\pd_{\tau}^{M+N}\pd_{\tau}F_{t+\tau}= \pd_{t}^{M+N+1}F_{t+\tau} = e^{-\frac{\tau}{2} DB}\chi^+(DB)(-DB)^M \pd_{t}^N\pd_{t}F_{t+\frac{\tau}{2}}\\= (-1)^M2^N (DB)^M e^{-\frac{\tau}{2} DB}\chi^+(DB) \pd_{\tau}^N\pd_{t}F_{t+\frac{\tau}{2}} = (-1)^M2^N (DB)^M e^{-\frac{\tau}{2} |DB|} \pd_{\tau}^N\pd_{t}F_{t+\frac{\tau}{2}}.
$$
\end{multline*}
Hence,  with $c=  \frac{(-2)^{N}}{(N+M-1)!}$, we have 
\begin{align*}
   \pair{\varphi_{0}}{\pd_{t}f_{t}} &= \frac{(-1)^{N+M}}{(N+M-1)!}\int_{0}^\infty \pair {\varphi_{0}}{\pd_{\tau}^{N+M}\pd_{\tau}F_{t+\tau}}\, \tau^{M+N-1} \,{d\tau}\\
   &= \frac{(-1)^{N+M-2}}{(N+M-1)!}\int_{t}^\infty \pair {\varphi_{0}}{\pd_{\tau}^{N+M}(\pd_{\tau}F_{\tau})}\, (\tau-t)^{M+N-1} \,{d\tau}. \end{align*}
 It follows from  \eqref{eq:sgeq} and   theory of analytic semigroups  that $\tau^k\pd_{\tau}^k(\pd_{\tau}F_{\tau})$ converges to 0 in $\IH^2_{DB}= \clos{\ran_{2}(DB)}$ for all $k\ge 0$ as $\tau\to \infty$. As $\IH^2_{DB}=\IH^2_{D}$, this implies convergence in $\mS'$.  Thus, we can apply the following elementary lemma (whose proof is left to the reader)   to $g(\tau):= \pair {\varphi_{0}}{\pd_{\tau}F_{\tau}}$ with $k=M+N$,  and we see that the previous integral is equal to $ \pair{\varphi_{0}}{\pd_{t}F_{t}}$.
 
 \begin{lem}\label{lem:technical} Let $k\in \N, k\ge 1$ and  $g: \R^+ \to \C$ be a function of class $C^k$ with $x^jg^{(j)}(x)\to 0$ as $x\to \infty$ when $j=0,1, \ldots , k-1$. Then 
$$
g(x)= \frac{(-1)^{k-2}}{(k-1)!}\int_{x}^\infty g^{(k)}(y) (y-x)^{k-1}\, dy.
$$
\end{lem}

\

Conclusion: we show that $f_{t}=F_{t}$  for  all $t>0$. From what precedes, we have that $F_{t}=f_{t}+G$ where $G\in \mS'$ is a constant distribution. 
We assumed that $F_{t}\to 0$ in $\mD'$ as $t\to \infty$. By the properties of $S_{p}(t)$, we have that $f_{t}\to 0$ in $H^p_{DB}$ when $t\to \infty$ (See \cite{AS}, Section 6) thus in $\mD'$, as $H^p_{DB}=H^p_{D}$. It follows that $G=0$ and we are done.

   \
   
\subparagraph{\textbf{Case $\mathbf{p> 2}$.}} 

Again $p>2$ means here $2<p<p_{+}(DB)$. The beginning is exactly the same as the similar  case when proving (i) implies (iii) with   $\pd_{t}F_{t}$ replacing $F_{t}$ and  we obtain  (compare \eqref{eq:a4}) 
\begin{equation}
\label{eq:a5}
 \pair { \varphi_0}{\partial_{t }^2F_{t+\tau}}   = \pair { e^{-\tau B^*D}\chi^+(B^*D)  \phi_0}{\partial_{t }^2F_{t}}
 \end{equation} 
 for all  $\varphi_{0}\in \mS $,  $t>0$ and $\tau\ge 0$, with $\phi_{0}=\IP_{B^*D}\varphi_{0}$. As before, taking $t=\tau$ leads to 
\begin{align*}
| \pair { \varphi_0}{\partial_{t }^2F_{2t}}|    & \le     \|e^{-t B^*D}\chi^+(B^*D)  \phi_0\|_{E^{p'}_{t}}\|\partial_{t }^2F_{t}\|_{E^p_{t}} \\
    &  \lesssim   \|\chi^+(B^*D)  \phi_0\|_{p'}\| \|\partial_{t }^2F_{t}\|_{E^p_{t}} \\
    & \lesssim  \|  \varphi_0\|_{p'} t^{-2}.
\end{align*} 
    In the last line we used \eqref{eq:ept2} with $\pd_{t}F_{t}$ being the conormal  gradient of a solution and then Corollary \ref{cor:repeatCaccio}: 
    $$
 \|t^2\partial_{t }^2F_{t}\|_{E^p_{t}}    \lesssim  \|t^2\pd_{t}^2F\|_{T^p_{2}}. $$
 Following the same line of argument as before, we conclude that 
 $\partial_{t }^2F_{t} \in H^{p, +}_{DB}$ and  for all $t>0$ and $\tau\ge 0$,
 $$
\partial_{t }^2F_{t+\tau} = S_{p}^+(\tau)\partial_{t }^2F_{t}=S_{p}(\tau)\partial_{t }^2F_{t}. 
$$
As in step 2 of the case $p\le 2$ above, we exhibit for all $t\ge 0$ a function $f_{t}\in H^{p,+}_{DB}$  such that 
$$
f_{t}= S_{p}^+(t)f_{0}= S_{p}(t)f_{0}
 $$
 and such  that $\pd_{t}^2f_{t}=\pd_{t}^2F_{t}$ in $\mS'$ when $t>0$. We give details below as the justifications are not the same. Admitting this,  we have $F_{t}=f_{t}+G+tH$ for two distributions $G,H$ in $\mS'$.  Using again that $F_{t}\to 0$ in $\mD'$ by hypothesis and $f_{t}\to 0$ in $H^p_{DB}$, hence in $\mD'$ because of the value of $p$, we conclude that $G=H=0$ and $F_{t}=f_{t}$ as desired. 
 
 It remains to justify our claims on $f_{t}$.  We construct $f_{t,k}$  similarly by
 \begin{align*}
   f_{t,k}&=c\int_{0}^\infty \tau^M (DB)^Me^{-\frac{\tau}{2} |DB|}(\chi_{k} \tau^N\pd_{\tau}^NF_{t+\frac{\tau}{2}})\, \frac{d\tau}{\tau}    \\
      &=\Tpsi{\varphi}{DB}(\chi_{k} \tau^N\pd_{\tau}^NF_{t+\frac{\tau}{2}}).
\end{align*}
Because $p>2$,   the  \cite[Corollary 4.4]{AS} for $DB$ shows that if $M>0$ then 
$f_{t,k}\in \IH^p_{DB}$ and converges in $H^p_{DB}$ to $f_{t}$. Now if $\phi_{0}\in \IH^{p'}_{B^*D}$, then we obtain
\begin{align}
\label{eq:ft2}
\pair{\phi_{0}}{f_{t}} &=c\iint_{\reu} \big( {\tau^M (B^*D)^Me^{-\frac{\tau}{2} |B^*D|}\phi_{0}(x)}\cdot ({\tau^N\pd_{\tau}^NF_{t+\frac{\tau}{2}})(x)}\big)\, \frac{d\tau dx}{\tau}\\
\nonumber& = c\int_{0}^\infty  \pair {\tau^M (B^*D)^Me^{-\frac{\tau}{2} |B^*D|}\phi_{0}}{\tau^N\pd_{\tau}^NF_{t+\frac{\tau}{2}}}\, \frac{d\tau}{\tau}.
\end{align}
Here, the  pairings hold in the sense of $L^{p'}, L^p$ duality if $N\ge 2$ (because all derivatives of $F_{t}$ with order exceeding 2 are in $L^p$ so far).  Next, the argument that  $f_{t}$ satisfies the semigroup equation is the same argument involving  $L^{p'}, L^p$ duality pairings.
 
 Then, 
one can replace 
such $\phi_{0}$  in the second line of \eqref{eq:ft2} by any $\varphi_{0}\in \mS$ (argue with $\varphi_{0}-\phi_{0}\in \nul_{p'}(D)$ when $\phi_{0}=\IP_{B^*D}\varphi_{0}$ and recall that $\IP_{B^*D}$ is bounded on $L^{p'}$ for $p_{-}(B^*D)<p'<2$) and we have
$$\pair{\varphi_{0}}{f_{t}}  = c\int_{0}^\infty  \pair {\tau^M (B^*D)^Me^{-\frac{\tau}{2} |B^*D|}\varphi_{0}}{\tau^N\pd_{\tau}^NF_{t+\frac{\tau}{2}}}\, \frac{d\tau}{\tau}.$$
The rest of the argument is similar, choosing the same constant $c$, and we have to differentiate twice here,  systematically computing with the $L^{p'}, L^p$ duality, using that  $\tau^k\pd_{\tau}^k(\pd_{\tau}^2F_{\tau})$ converges to 0 in $H^p_{DB}$ for all $k\ge 0$ as $\tau\to \infty$. As $H^p_{DB}=\clos{\ran_{p}(D)}$, this implies convergence in $\mS'$.  Thus, by Lemma  \ref{lem:technical}  applied to $g(\tau)= \pair {\varphi_{0}}{\pd_{\tau}^2F_{\tau}}$ with $k=M+N$, we see that  $\pair{\varphi_{0}}{\pd_{t}^2f_{t}}= \pair{\varphi_{0}}{\pd_{t}^2F_{t}}$.

\section{Proof of Theorem \ref{thm:main2}}

Recall that $q\in I_{L^*}$, that is,  $H^q_{D\wt B}=H^q_{D}$. This is the same as $H^q_{DB^*}=H^q_{D}$ (See \cite{AS}, Section 12.2).
We distinguish the two cases  $q>1$ or $q\le 1$ of  the statement.

\

\subparagraph{\textbf{Case $\mathbf{q>1 : (\alpha) \Longrightarrow (\beta)}$}} 

We set $p=q'$.  Thus we assume  that  $t F\in T^p_{2}$,
where $F$ is a solution of  \eqref{eq:curlfreesystem} in $\reu$, that is, $F=\nabla\!_{A}u$.  We also assume that $u(t,\cdot)$ converges to $0$ in $\mD'$ modulo constants as $t\to \infty$. We will show at the end how this assumption can be removed in some cases.  In terms of $F$, this means that the tangential part of $F_{t}$ converges to $0$ in $\mD'$ as $t\to \infty$. Indeed,  if $\varphi_{0}$ is a test function, then
$$
\pair{\varphi_{0}}{F_{t\ta}}= \pair{\varphi_{0}}{\nabla_{x}u(t,\cdot)}= -  \pair{\divv\varphi_{0}}{u(t,\cdot)}.
$$

We follow the same paths as before. We point out  the differences when necessary. 

\

  \paragraph{\textbf{Step 1}}   Finding the semigroup equation. 

\

This is achieved  by  taking limits in  \eqref{eq:step1} by selecting $\chi$, $\eta$ and $\phi_{0}$ in the definition of $\varphi_{s}$ in \eqref{eq:phis}. 

An observation is in order.  To define $\varphi_{s}$ we took $\phi_{0}\in \IH^2_{B^*D}=\clos{\ran_{2}(B^*D)}$. But  equation  \eqref{eq:step1} involves only $B^*D \varphi_s$ and $\partial_{s}\varphi_{s}$ which  are equal. Now, if we further assume $\phi_{0}\in \dom_{2}(B^*D)= \dom_{2}(D)$, then one can use the similarity relation between the functional calculi of $DB^*$ and $B^*D$ via $D$, so that
\begin{equation}
\label{eq:phi0}
\eta(s)B^*D\varphi_{s}=\begin{cases}
 {\eta(s)}B^*e^{-(t-s)DB^*}\chi^+(DB^*)D\phi_{0}     & \text{if } s<t,  \\
    - {\eta(s)}B^*e^{-(s-t)DB^*}\chi^-(DB^*)D\phi_{0}     & \text{if } t<s. \
\end{cases} 
\end{equation}
But these expressions are defined for all $\phi_{0}\in \dom_{2}(D)$ and,  as $D$ annihilates  $\nul_{2}(B^*D)$, this means that  \eqref{eq:step1} is valid for any $\phi_{0}\in \dom_{2}(D)$ (up to changing $\phi_{0}$ to $\IP_{B^*D}\phi_{0}$ in the definition of $\varphi_{s}$).  

\

\paragraph{\textbf{Step 1a}} Limit in space.    We show that if $\phi_{0}\in {\dom_{2}(D)}$ with $D\phi_{0}\in \IH^{q}_{D}$,  then \begin{equation}    \label{eq:step2th1.2}
  \iint_{\reu} \bpaire {\eta'(s) B^*(x)D\varphi_s(x)}{F(s,x)}\,   dsdx = 0\end{equation}
  with $\eta'(s)B^*(x)D\varphi_s(x)$ as in  \eqref{eq:phi0}
  and the integral is defined in the Lebesgue sense.

\

We replace  $\chi$ in \eqref{eq:step1}  by $\chi_{R}$ with $\chi_{R}(x)=\chi(x/R)$ where $\chi\equiv 1$ in the unit ball $B(0,1)$, has compact support in the ball $B(0,2)$ and let $R\to \infty$.  As $\chi_{R}$ tends to 1 and $D_{\chi_{R}}$ to 0, it suffices by dominated convergence to show that $|\eta'(s)B^*D \varphi_s  F|$ and $|\eta(s)\partial_{s}\varphi_{s} F|$ are integrable on $\reu$.  As $sF\in T^p_{2}$, it is enough to have that $\eta'(s)s^{-1}B^*D \varphi_s$ and $\eta(s)s^{-1}\partial_{s}\varphi_{s}$ belong to $(T^p_{2})'=T^{q}_{2}$.  As $\partial_{s}\varphi_{s}=B^*D\varphi_{s}$ on $\supp \eta$, it suffices to invoke the following lemma.

\begin{lem}\label{lem:step1th1.2} Let $q$ with $H^q_{DB^*}=H^q_{D}$. Assume $\phi_{0}\in \dom_{2}(D) $ and  $D\phi_{0}\in \IH^{q}_{DB^*}$. Then $\eta(s)B^*D\varphi_{s}\in T^q_{2}$ for all $\eta$ bounded and compactly  supported  away from $t$. 

\end{lem} 

\begin{proof}

We use \eqref{eq:phi0}. 
By geometric considerations as  $t-s$ is bounded away from 0, we see that if $(s,y)$ belongs to a cone $\Gamma(x)$ with $s\in \supp(\eta) \cap (0,t)$ then $(t-s,y)$ belongs to a  cone $\wt \Gamma(x)$ with (bad) finite aperture depending on the support of $\eta$. Thus, setting $t-s=\sigma$ and using also that $s$ is bounded below and $t-s$ bounded above on $\supp(\eta)$,  we obtain from elementary estimates based on these considerations and $L^\infty$ boundedness of $B^*$, and adjusting the parameters $c_{0},c_{1}$ in $\tN$,
$$
\iint_{\Gamma{(x)}, s<t} |\eta(s)B^*(y)D\varphi_{s}(y)|^2 \,\frac{dsdy}{s^{n+1}} \lesssim_{\eta} \tN(e^{-\sigma DB^*}\chi^+(DB^*)D\phi_{0})^2(x).
$$
By Theorem 9.1 and Theorem 4.19  of \cite{AS},  we have for $q$ in our range,
$$
\|\tN(e^{-\sigma |DB^*|}\chi^+(DB^*)D\phi_{0})\|_{q}\lesssim  \|\chi^+(DB^*)D\phi_{0}\|_{H^q} \lesssim \|D\phi_{0}\|_{H^q}.
$$

The argument is the same when $t<s$.  \end{proof}

  \

\paragraph {\textbf{Step 1b}}  Limit in time. 

\

For fixed $t>0$,   $0<\varepsilon<\inf(t/4, 1/4,1/t)$, we obtain by making the same choices of $\eta$ as the proof of (i) implies (iii) of Theorem \ref{thm:main1}, 
\begin{align}   \label{eq:limit-th1.2} 
  \frac{1}{\varepsilon} \iint_{[t+\varepsilon, t+2\varepsilon]\times \R^n}& \bpaire { B^*(x)D\varphi_s(x)}{F(s,x)}\,   dsdx \\\nonumber & = 2\varepsilon \iint_{[t+ \frac{1}{2\varepsilon}, t+\frac{1}{\varepsilon}]\times \R^n} \bpaire { B^*(x)D\varphi_s(x)}{F(s,x)} \, dsdx
  \end{align}
  and
  \begin{align}\label{eq:limit+th1.2}    
  \frac{1}{\varepsilon} \iint_{[t-2\varepsilon, t-\varepsilon]\times \R^n}& \bpaire { B^*(x)D\varphi_s(x)}{F(s,x)}\,   dsdx \\\nonumber &=\frac{1}{\varepsilon} \iint_{[{\varepsilon}, 2{\varepsilon}]\times \R^n} \bpaire { B^*(x)D\varphi_s(x)}{F(s,x)} \, dsdx.
  \end{align}
  
   The second integral in \eqref{eq:limit-th1.2} converges to 0 as $\varepsilon\to 0$ for fixed $t$.
 Indeed, let 
 \begin{multline*}$$G(s,x)= G_{s}(x)=1_{[t+ \frac{1}{2\varepsilon}, t+\frac{1}{\varepsilon}]}(s)B^*(x)D\varphi_s(x)\\=  1_{[t+ \frac{1}{2\varepsilon}, t+\frac{1}{\varepsilon}]}(s)(B^*e^{-(s-t)DB^*}\chi^-(DB^*)D\phi_{0})(x),$$
 \end{multline*}
  under the conditions on $\phi_{0}$ in Step 1a.  This integral is bounded by  
  $
  2\varepsilon \|sF\|_{T^p_{2}}\|G\|_{T^q_{2}}.$ 
 As $t\varepsilon<1$ and $s\in [t+ \frac{1}{2\varepsilon}, t+\frac{1}{\varepsilon}]$, we have $\sigma= s-t\in [ \frac{1}{2\varepsilon}, \frac{1}{\varepsilon}]$, so that  as in the proof of the lemma above, 
 $$
 \|G\|_{T^q_{2}} \lesssim \|\tN(e^{-\sigma |DB^*|}\chi^-(DB^*)D\phi_{0})\|_{q}
 $$
 and examination shows uniformity with respect to $\varepsilon$. 
  
Next, to obtain the expression of the limits of the other terms in \eqref{eq:limit-th1.2} and \eqref{eq:limit+th1.2}, we shall impose further that $\chi^\pm(DB^*)D\phi_{0} \in E^q_{\delta }$ for some (any) $\delta >0$. It is convenient to introduce   $\D_{q}$ as the space  of  $\phi_{0}\in \dom_{2}(D) $ with $D\phi_{0}\in  \IH^{q}_{D}$ and $ \chi^\pm(DB^*) D\phi_{0 } \in E^{q}_{\delta }$ for some (any) $\delta >0$. Here this definition makes sense  for any $q\in I_{L^*}$ since in that case $\IH^{q}_{D}= \IH^{q}_{DB^*}$ and $\chi^\pm(DB^*)$ are bounded operators on $\IH^{q}_{DB^*}$.   The following  lemmas hold for any such $q$.

\begin{lem}\label{lemma1th1.2} Let $q\in I_{L^*}$. Let $h\in \IH^{q, \pm}_{DB^*}$. For all $\delta >0$, 
$e^{-\delta |DB^*|}h\in E^{q}_{\delta }$ with uniform bound with respect to $\delta $. More precisely, 
$\sup_{\delta >0}\|e^{-\delta |DB^*|}h\|_{E^{q}_{\delta }}\lesssim \|h\|_{q}$. 
\end{lem}

\begin{proof} Direct consequence of Theorem 9.1 of \cite{AS} together with the method of proof of  Lemma \ref{lemma1}.  \end{proof}

\begin{lem}\label{cor:stability} Let $q\in I_{L^*}$. If $\phi_{0}\in \dom_{2}(D)$ with $D\phi_{0}\in  \IH^{q}_{DB^*}$, then for all $\delta >0$ and $M\in \N$, $(B^*D)^Me^{-\delta |B^*D|}\chi^\pm(B^*D)\phi_{0}\in \D_{q}$. 
\end{lem}

\begin{proof}
For this proof, set  $h= (B^*D)^Me^{-\delta |B^*D|}\chi^\pm(B^*D)\phi_{0}$. As $\phi_{0}\in \dom_{2}(D)$, then $h\in \dom_{2}(D)$ and 
$$
Dh= (DB^*)^Me^{-\delta |DB^*|}\chi^\pm(DB^*)D\phi_{0}.
$$
As $(DB^*)^Me^{-\delta |DB^*|}\chi^\pm(DB^*)$ is bounded on $\IH^q_{DB^*}=\IH^q_{D}$ by $q\in I_{L^*}$, we have $Dh\in \IH^q_{DB^*}$.

Finally,  $\chi^\pm(DB^*)Dh= (DB^*)^Me^{-\delta |DB^*|}\chi^\pm(DB^*)(\chi^+(DB^*)D\phi_{0})$
which belongs to  $E^q_{\delta }$. For $M=0$, this is in Lemma \ref{lemma1th1.2}. For $M>0$, one can take derivatives and apply Corollaries  \ref{cor:repeatCaccio} and \ref{cor:uniformept}. We skip details. 
\end{proof}

\begin{lem}\label{lemma2th1.2} Fix $\delta >0$. Let $q\in I_{L^*}$. Let $h\in \IH^{q,\pm}_{DB^*}\cap E^{q}_{\delta }$.  Then $e^{-s |DB^*|}h$ converges to $h$ in $E^{q}_{\delta }$ when $s\to 0$. 
\end{lem}

\begin{proof} The proof is exactly as that of Lemma \ref{lemma2}. 
\end{proof}

\begin{lem}\label{lemma3th1.2} Let  $\varphi_{0}\in \mS$.  Then $\varphi_{0}\in \D_{q}$ for any $q\in I_{L^*}$.
\end{lem}

We prove this lemma in Section \ref{sec:technical}.

\

Armed with these lemmas and reexpressing \eqref{eq:limit-th1.2} and \eqref{eq:limit+th1.2} with the $E^q_{t},  E^p_{t}$ duality pairings  and taking limits we obtain exactly as in the previous arguments for all  $\phi_{0}\in \D_{q}$, $t>0$ and $\tau\ge 0$,
\begin{equation}
\label{eq:a1th1.2}
 \pair { B^*\chi^-(DB^*)D  \phi_0}{F_{t}}= 0,
\end{equation}
\begin{equation}
\label{eq:a2th1.2}
 \pair { B^*e^{-\tau DB^*}\chi^+(DB^*) D \phi_0}{F_{t}}=  \pair { B^*\chi^+(DB^*)D  \phi_0}{F_{t+\tau}},
\end{equation}
so that summing with the first equation at $t+\tau$, we get
\begin{equation}
\label{eq:a3th1.2}
 \pair {B^*D  \phi_0}{F_{t+\tau}}= \pair { B^*e^{-\tau DB^*}\chi^+(B^*D)  D\phi_0}{F_{t}}.
\end{equation} 

\

\paragraph{\textbf{Step 2}} We show that $\partial_{t}^kF_{t}\in \dot W^{-1, p}$, for $k\ge 1$.
  By Lemma \ref{lemma3th1.2}, any  $\varphi_{0}\in \mS$ has the required properties to apply \eqref{eq:a3th1.2} and using integration by parts in slice-spaces (Lemma \ref{lem:sliceIBP}), we have obtained
\begin{equation}
\label{eq:a4th1.2}
- \pair { \varphi_0}{\partial_{t }F_{t+\tau}}   = \pair {B^* e^{-\tau DB^*}\chi^+(DB^*)  D\varphi_0}{F_{t}},
 \end{equation} 
 where both pairings can be expressed as Lebesgue integrals. 
 
 Now, using this equality with $\tau=t$,   the $E^{q}_{t}, E_{t}^p$ duality and Lemma \ref{lemma1th1.2},  we have
 \begin{align*} 
| \pair { \varphi_0}{\partial_{t }F_{2t}}|    & \le     \|B^*e^{-t DB^*}\chi^+(DB^*)  D\varphi_0\|_{E^{q}_{t}}\|F_{t}\|_{E^p_{t}} \\
    &  \lesssim   \|D  \varphi_0\|_{q} \|F_{t}\|_{E^p_{t}} \\
    & \lesssim  \|  \varphi_0\|_{\dot W^{1,q}} t^{-1}.
\end{align*} 
As this is true for all $\varphi_{0}\in \mS$, this shows that $\partial_{t }F_{2t}$ is an element in $\dot W^{-1,p}$ with norm controlled by $t^{-1}$.  
As usual, as $\pd_{t}F_{t}$ is a conormal gradient, this automatically implies that it belongs to $\dot W^{-1,p}_{D}$. Remark that one can differentiate \eqref{eq:a4th1.2} with respect to $t$ (using dominated convergence)  and then make $\tau=t$ to obtain
$$
| \pair { \varphi_0}{\partial_{t }^kF_{2t}}| \lesssim \|  \varphi_0\|_{\dot W^{1,q}} t^{-k}
$$
for all integer $k\ge 1$. 

\

\paragraph{\textbf{Step 3}} Now, for all $t>0$, we create an element $f_{t}\in \dot W^{-1,p, +}_{DB}$ which, in the end, will be $F_{t}$.  Recall that $p=q'$.
We require the following lemma.   

\begin{lem} \label{lem:obs1} For each $t>0$ and $N>\frac{n}{2}$, we have $(s,x)\mapsto s^{N+1}\pd_{s}^NF_{t+s}(x)$ belongs to $T^p_{2}$ with uniform bound with respect to $t$. Moreover, it is $C^\infty$ as a  function of $t$ valued in $T^p_{2}$.  
 If $t\to 0$, then it converges to $(s,x)\mapsto s^{N+1}\pd_{s}^NF_{s}(x)$ in $T^p_{2}$ and, if $t\to \infty$, it converges to $0$ in $T^p_{2}$.
\end{lem}

The proof is identical to that of Lemma \ref{lem:obs}.  

We need a little bit of Sobolev theory for $DB$  (only evoked at the end of \cite{AS} but working the same way as the Hardy space theory). 
For $1<p<\infty$, let  $\dot \W^{-1,p}_{DB}$ be the space of functions of the form $\Tpsi{\psi}{DB}F$ with $F\in T^2_{2}$ and $\tau  F\in T^p_{2}$ with norm $\inf \|\tau  F\|_{T^p_{2}}<\infty$. This space does not depend on the particular choice of  $\psi$ among bounded holomorphic functions  in bisectors $S_{\mu}$ with enough decay at 0 and $\infty$.  It coincides with the space of functions $h\in \clos{\ran_{2}(DB)}= \IH^2_{DB}$ such that  $\|\tau \Qpsi{\psi}{DB}h\|_{T^p_{2}} $ again for any  choice of  $\psi$ as above. We let $\dot W^{-1,p}_{DB}$ be its completion. 

Also for $1<q<\infty$, 
let $\dot \W^{1,q}_{B^*D}$ be the set of functions $h\in \clos{\ran_{2}(B^*D)}=\IH^2_{B^*D}$ with 
$\tau^{-1} \Qpsi{\psi}{B^*D}h\in T^q_{2}$, equipped with norm $\|\tau^{-1} \Qpsi{\psi}{B^*D}h\|_{T^q_{2}}$. Again, this space does not depend on the particular choice of  $\psi$ as above  and can be characterized as well by the $\Tpsi{\psi}{B^*D}$ maps. Let $\dot W^{1,q}_{B^*D}$ be its completion.  

When $q=p'$, both  spaces $\dot \W^{1,q}_{B^*D}, \dot \W^{-1,p}_{DB}$ are in duality for the $L^2$ duality. This duality extends to the completed spaces. 

For $q\in I_{L^*}$, $q>1$ and $p'=q$, we have that $\dot W^{-1,p}_{DB}= \dot W^{-1,p}_{D}$ with equivalence of norms and $\IP$ extends to an isomorphism from $\dot W^{1,q}_{B^*D}$  onto $\dot W^{1,q}_{D}$. 

\begin{lem}\label{lem:density} Let $q\in I_{L^*}$ with $q>1$. Then $\D_{q}\cap \IH^2_{B^*D}$ is a dense subspace of 
$\dot \W^{1,q}_{B^*D}$. 
\end{lem}

\begin{proof}
See Section \ref{sec:technical}.
\end{proof}

Let $N$ be as above with $N\ge 1$, and $M$ be another integer also chosen large. Let  $c$ be some constant  to be chosen later and $\varphi(z)= cz^{M}e^{-\frac{1}{2}\modz}$. For all integers $k$ and $t\ge 0$, 
set 
\begin{align*}
   f_{t,k}&=c\int_{0}^\infty \tau^{M} (DB)^Me^{-\frac{\tau}{2} |DB|}(\chi_{k} \tau^N\pd_{\tau}^NF_{t+\frac{\tau}{2}})\, \frac{d\tau}{\tau}    \\
      &=\Tpsi{\varphi}{DB}(\frac{1}{\tau}\chi_{k} \tau^{N+1}\pd_{\tau}^{N}F_{t+\frac{\tau}{2}}),
\end{align*}
with $\chi_{k}$ being the indicator function of $[\frac{1}{k}, k ]\times B(0, k)$ as before.
The $\dot W^{-1,p}_{DB}$ theory shows that if $M$ is large enough,  then  $f_{t,k}\in \dot \W^{-1,p}_{DB}$ with 
$\|f_{t,k}\|_{ \dot \W^{-1,p}_{DB}} \lesssim \|\chi_{k} \tau^{N+1}\pd_{\tau}^{N}F_{t+\frac{\tau}{2}}\|_{T^p_{2}}$ uniformly in $t$ and $k$,  and, by Lemma \ref{lem:obs1}, $f_{t,k}$ converges to some $f_{t}\in \dot W^{-1,p}_{DB}$ as $k\to \infty$.    In particular, for all $\phi_{0}\in \D_{q}\cap \IH^2_{B^*D}$, we have
\begin{align*}
\pair{\phi_{0}}{f_{t}}&=c\iint_{\reu} \big( { \tau^{M-1}  (B^*D)^{M}e^{-\frac{\tau}{2} |B^*D|}\phi_{0}(x)}\cdot ({\tau^{N+1}\pd_{\tau}^NF_{t+\frac{\tau}{2}})(x)}\big)\, \frac{d\tau dx}{\tau} \\
&= c\int_{0}^\infty \pair { \tau^{M-1} B^* (DB^*)^{M-1}e^{-\frac{\tau}{2} |DB^*|}D\phi_{0}
}{\tau^{N+1}\pd_{\tau}^NF_{t+\frac{\tau}{2}}}\, \frac{d\tau}{\tau}.
\end{align*} 
The first integral converges in the Lebesgue sense and $ \pair{\phi_{0}}{f_{t}} $ is interpreted using the duality of the spaces  $\dot  \W^{1,q}_{B^*D}, \dot W^{-1,p}_{DB}$ extending the $L^2$ duality on 
$\dot  \W^{1,q}_{B^*D}, \dot \W^{-1,p}_{DB}$ as explained above. One can see that the integral with respect to $x$      can be interpreted in the  $E^q_{\delta} , E^p_{\delta }$ duality for each $\tau$ and the second equality follows from Fubini's theorem.  

This  interpretation  allows us  to show that 
\begin{equation}
\label{eq:ftsgth1.2}
f_{t} = \wt S_{p}^+(t)f_{0}=\wt S_{p}(t)f_{0},
\end{equation} 
where $\wt S_{p}^+(\tau)$ is the extension of $e^{-\tau DB}\chi^+(DB)$ on $\dot W^{-1,p}_{D}$ described in Lemma 14.4 of \cite{AS}, which agrees with the extension  $\wt S_{p}(\tau)$  of $e^{-\tau |DB|}$ on $\dot W^{-1,p,+}_{DB}$. 
Indeed, by  Lemma \ref{lem:density}, it suffices to test against some $\phi_{0}\in \D_{q}\cap \IH^2_{B^*D}$. First, note that we have $\phi_{t}=e^{-tB^*D}\chi^+(B^*D)\phi_{0}\in  \D_{q}\cap \IH^2_{B^*D}$ as well, hence  
$
\pair{\phi_{0}}{\wt S_{p}^+(t)f_{0}}= \pair{\phi_{t}}{f_{0}}
$
by definition of the semigroup $\wt S_{p}^+(t)$. 
Secondly, we may compute $ \pair{\phi_{t}}{f_{0}}$ using the defining representation of $f_{0}$ with $\phi_{t}$ replacing $\phi_{0}$ and obtain 
\begin{align*}
   \pair{\phi_{t}}{f_{0}}= 
c\int_{0}^\infty \pair { \tau^{M-1} B^* (DB^*)^{M-1}e^{-\frac{\tau}{2} |DB^*|}D\phi_{t}
}{\tau^{N+1}\pd_{\tau}^NF_{\frac{\tau}{2}}}\, \frac{d\tau}{\tau}.   
\end{align*}
 Observe that 
$$B^* (DB^*)^{M-1}e^{-\frac{\tau}{2} |DB^*|}D\phi_{t}
= B^* e^{-tDB^*} \chi^+(DB^*) D (B^*D)^{M-1} e^{-\frac{\tau}{2} |B^*D|}\phi_{0},
$$
so by \eqref{eq:a2th1.2} with $(B^*D)^{M-1} e^{-\frac{\tau}{2} |B^*D|}\phi_{0}\in \D_{q}$ replacing $\phi_{0}$ and  $M\ge 1$, 
\begin{multline*}$$
\pair { \tau^{M-1} B^* (DB^*)^{M-1}e^{-\frac{\tau}{2} |DB^*|}D\phi_{t}
}{\tau^{N+1}\pd_{\tau}^NF_{\frac{\tau}{2}}}\\ = \pair { \tau^{M-1} B^* (DB^*)^{M-1}e^{-\frac{\tau}{2} |DB^*|}D\phi_{0}
}{\tau^{N+1}\pd_{\tau}^NF_{t+\frac{\tau}{2}}}.
$$
\end{multline*}
Inserting this equality in the  $\tau$-integral,  we find $ \pair{\phi_{t}}{f_{0}}= \pair{\phi_{0}}{f_{t}}$. 
This and a density argument  prove \eqref{eq:ftsgth1.2}. 

\

\paragraph{\textbf{Step 4}}  Now, we connect this to the calculus in $\mS'$ and show that 
${\pd_{t}f_{t}}= {\pd_{t}F_{t}}$ in $\mS'$ if $f_{t}$ is appropriately normalized.  Since $\dot W^{-1,p}_{DB}= \dot W^{-1,p}_{D}
$ with equivalence of norms and  $f_{t,k}\to f_{t}$ in $\dot W^{-1,p}_{DB}$, we have for all $\varphi_{0}\in \mS$,  
$ \pair{\varphi_{0}}{f_{t}}= \lim_{k}  \pair{\varphi_{0}}{f_{t,k}}$. For each $k$, we have by the $L^2$ theory, we can express $\pair{\varphi_{0}}{f_{t,k}}$ as
$$
\pair{\varphi_{0}}{f_{t,k}}=c\iint_{\reu} \big( { \tau^{M-1} B^*  (DB^*)^{M-1}e^{-\frac{\tau}{2} |DB^*|}D\varphi_{0}(x)}\cdot (\chi_{k}{\tau^{N+1}\pd_{\tau}^NF_{t+\frac{\tau}{2}})(x)}\big)\, \frac{d\tau dx}{\tau}.
$$
By our assumption $H^q_{DB^*}=H^q_{D}$ implies   
$$
\|\tau^{M-1}  B^*  (DB^*)^{M-1}e^{-\frac{\tau}{2} |DB^*|}D\varphi_{0}\|_{T^q_{2}}\lesssim  \|D\varphi_{0}\|_{q}  \lesssim  \|\varphi_{0}\|_{\dot W^{1,q}}.
$$
We can take the limit in the integral above as $k\to \infty$ and obtain first the integral on $\reu$ and then by Fubini's theorem,
$$
\pair{\varphi_{0}}{f_{t}} = c\int_{0}^\infty \pair { \tau^{M-1} B^* (DB^*)^{M-1}e^{-\frac{\tau}{2} |DB^*|}D\varphi_{0}
}{\tau^{N+1}\pd_{\tau}^NF_{t+\frac{\tau}{2}}}\, \frac{d\tau}{\tau}.
$$
The same argument having first  differentiated $\pair{\varphi_{0}}{f_{t}}$  in  $t$ using Lemma \ref{lem:obs1} and the double integral representation  leads to 
$$
\pair{\varphi_{0}}{\pd_{t}f_{t}} = c\int_{0}^\infty \pair { \tau^{M-1} B^* (DB^*)^{M-1}e^{-\frac{\tau}{2} |DB^*|}D\varphi_{0}
}{\tau^{N+1}\pd_{\tau}^N\pd_{t}F_{t+\frac{\tau}{2}}}\, \frac{d\tau}{\tau}.
$$
The meaning of the pairing for the integrand  is as before with respect to   $E^q_{\delta }, E^p_{\delta }$ duality. We have, 
\begin{align*}
&\pair { \tau^{M-1} B^* (DB^*)^{M-1}e^{-\frac{\tau}{2} |DB^*|}D\varphi_{0}
}{\tau^{N+1}\pd_{\tau}^N\pd_{t}F_{t+\frac{\tau}{2}}}      \\
&= \pair { \tau^{M-1} B^* D\big[(B^*D)^{M-1}e^{-\frac{\tau}{2} |B^*D|}\varphi_{0}\big]
}{\tau^{N+1}\pd_{\tau}^N\pd_{t}F_{t+\frac{\tau}{2}}}\\
    & = \pair { \tau^{M-1} (B^* D)^{M-1}B^*e^{-\frac{\tau}{2} |DB^*|}\chi^+(DB^*)D\varphi_{0}
}{\tau^{N+1}\pd_{\tau}^N\pd_{t}F_{t+\frac{\tau}{2}}} 
\\
&= \tau^{M+N}2^{-N}\pair {  B^* e^{-\frac{\tau}{2} |DB^*|}\chi^+(DB^*)D\varphi_{0}
}{(DB)^{M-1}\pd_{t}^N\pd_{t}F_{t+\frac{\tau}{2}}} 
\\
&= \tau^{M+N}2^{-N}(-1)^{M-1}\pair {  B^* e^{-\frac{\tau}{2} |DB^*|}\chi^+(DB^*)D\varphi_{0}
}{\pd_{t}^{N+M-1}\pd_{t}F_{t+\frac{\tau}{2}}} 
\\
&= \tau^{M+N}2^{-N}(-1)^{M-1}\pair {  B^* D\varphi_{0}
}{\pd_{t}^{N+M-1}\pd_{t}F_{t+{\tau}}} 
\\
&= \tau^{M+N}2^{-N}(-1)^{M}\pair {  \varphi_{0}
}{\pd_{\tau}^{N+M}\pd_{\tau}F_{t+{\tau}}}.
\end{align*}
In the first equality, we used that $\varphi_{0}\in \dom_{2}(D)$. In the second equality, we used \eqref{eq:a1th1.2} to insert $\chi^+(DB^*)$ as $(B^* D)^{M-1}e^{-\frac{\tau}{2} |B^*D|}\varphi_{0} \in E^q_{\delta }$ combining Lemma \ref{cor:stability} and Lemma \ref{lemma3th1.2}. In the third equality, we used the integration by parts iteratively in slice-spaces (Lemma \ref{lem:sliceIBP}). In the fourth we used the differential equation $\pd_{t}F_{t}=-DBF_{t}$ repeatedly,  in the fifth \eqref{eq:a3th1.2} and in the sixth,   integration by parts in slice-spaces again together with the differential equation. 
Thus choosing $c= \frac{(-2)^N}{(N+M-1)!}$ as before we obtain
\begin{align*}
   \pair{\varphi_{0}}{\pd_{t}f_{t}} &= \frac{(-1)^{N+M}}{(N+M-1)!}\int_{0}^\infty \pair {\varphi_{0}}{\pd_{\tau}^{N+M}\pd_{\tau}F_{t+\tau}}\, \tau^{M+N-1} \,{d\tau}\\
   &= \frac{(-1)^{N+M-2}}{(N+M-1)!}\int_{t}^\infty \pair {\varphi_{0}}{\pd_{\tau}^{N+M}(\pd_{\tau}F_{\tau})}\, (\tau-t)^{M+N-1} \,{d\tau}. \end{align*}
   Using again Lemma \ref{lem:technical} applied to $g(\tau)=  \pair {\varphi_{0}}{\pd_{\tau}F_{\tau}}$
   since $\tau^kg^{(k)}(\tau) $ is controlled by $\tau^{-1}$ for all $k\ge 0$, we conclude that 
   $\pair{\varphi_{0}}{\pd_{t}f_{t}}= \pair{\varphi_{0}}{\pd_{t}F_{t}}$. 
   
   \

\paragraph{\textbf{Step 5}} Conclusion: $f_{t}=F_{t}$ in $\mS'$.
   From Step 4,  there exists $G\in \mS'$, such that $F_{t}=f_{t}+G$ for all $t>0$. Recall that  $F_{t,\ta}\to 0 $ in $\mD'$ by our assumption.  Also, $f_{t}\to 0$ in $\dot W^{-1,p}_{D}$ from the semigroup equation, hence in $\mD'$. Thus we have $G_{\ta}=0$. As $G=\nabla\!_{A}w$ for some solution, we have $G_{\ta}= \nabla_{x}w$  and $G_{\no}=a\pd_{t}w+ b\cdot \nabla_{x}w$. It follows that $w(t,x)=w(t)$ and, as $G$ is independent of $t$,  $w'(t)=\alpha$ is a constant in $\C^m$. 
   Here, $a$ is the first diagonal block of $A$ in \eqref{eq:A}. In the construction of the matrix $B=\hat A$, it is proved (and used) that $a$ is invertible in $L^\infty$ with $\lambda|\alpha| \le |a(x) \alpha|$ almost everywhere, where $\lambda$ is the ellipticity constant for $A$. But at the same time we must have $tG\in E^p_{t}$ for all $t>0$. This implies that  $a\alpha\in E^p_{1}$ and the only possibility is $\alpha=0$.

\

\paragraph{\textbf{Step 6}} Elimination of  the condition $u(t,\cdot)\to 0$  as $t\to \infty$ in $\mD'$ modulo constants when $p< \frac{2n}{n-2}$. Recall that we have used the consequence that  $F_{t,\ta}\to 0$ as $t\to \infty$ in $\mD'$ when $F=\nabla_{A}u$.     [Actually, the converse holds using that any test function with mean value 0 is the divergence of some test function.]  However this limit always hold  when $tF \in T^p_{2}$ and $p< \frac{2n}{n-2}$ [which means that $u(t,\cdot)\to 0$  as $t\to \infty$ in $\mD'$ modulo constants when $t\nabla u \in T^p_{2}$]. Indeed,  if $p\le 2$, then we know that $tF_{t}\in E^p_{t}$ with uniform bound thus $tF_{t}\in L^p$ uniformly. This gives the desired limit. If $2<p$, then that $tF_{t}\in E^p_{t}$ uniformly implies 
$t^{n/p-n/2} t \|F_{t}\|_{L^2(B_{t/2})} \le C$ uniformly as a simple consequence of H\"older inequality (we leave the proof to the reader). In particular on any fixed ball $B$ we obtain $\|F_{t}\|_{L^2(B)}\to 0$ as $t\to \infty$  when  $p< \frac{2n}{n-2}$ and we are done. [Remark that using Meyers' improvement of  local $L^2$ estimates for the gradient of $u$ to local  $L^{p_{0}}$ estimates for some $p_{0}>2$, one can improve the upper bound to $p<\frac{p_{0}n}{n-p_{0}}$.]

  \

\subparagraph{\textbf{Case $\mathbf{q\le 1 : (a) \Longrightarrow (b)}$}}

We assume $q\in I_{L^*}$ and  $q\le 1$. For $\alpha=n(\frac{1}{q}-1)$, we assume  that  $t F\in T^\infty_{2,\alpha}$,
where $F$ is a solution of  \eqref{eq:curlfreesystem} in $\reu$, that is, $F=\nabla\!_{A}u$.  We also assume that $u(t,\cdot)$ converges to $0$ in $\mD'$ modulo constants as $t\to \infty$. In terms of $F$, this means that the tangential part of $F_{t}$ converges to $0$ in $\mD'$ as $t\to \infty$.

The proof is almost identical to the previous one: Step 1a is valid for $q\le 1$.  Next,  the stated  series of lemmas in Step 1b is  valid for $q\in I_{L^*}$, $ q\le 1$, and the conclusion is the same. In Step 2, we deduce that $\partial_{t}^kF_{t}\in \dot \Lambda^{\alpha-1}$ for $k\ge1$. In step 3, Lemma \ref{lem:obs1} is replaced by Lemma \ref{lem:obs2} below and  Lemma \ref{lem:density}  holds for $q\le 1$ with appropriate replacement of $\dot \W^{1,q}_{B^*D}$ explained below.  This allows us to define a candidate $f_{t}\in \dot \Lambda^{\alpha-1}$ which enjoys the desired semigroup formula and which will be $F_{t}$ in the end. We shall give some detail of the steps 4 and 5  to show this is the case. 

\begin{lem} \label{lem:obs2} For each $t>0$ and $2(N+1)-1>n+2\alpha$, we have $(s,x)\mapsto s^{N+1}\pd_{s}^NF_{t+s}(x)$ belongs to $T^\infty_{2,\alpha}$ with uniform bound with respect to $t$. Moreover, it is $C^\infty$ as a  function of $t$ valued in $T^\infty_{2,\alpha}$ equipped with the weak-star topology.  
 If $t\to 0$, then it converges for the weak-star topology to $(s,x)\mapsto s^{N+1}\pd_{s}^NF_{s}(x)$ in $T^\infty_{2,\alpha}$.
\end{lem} 

\begin{proof} Taking derivatives,  $\|t^{N+1} \pd_{t}^NF\|_{T^\infty_{2,\alpha}}<\infty$ for each integer $N\ge 1$ by the assumption for $N=0$ and Corollary \ref{cor:repeatCaccio}.

We take a Carleson box $(0,R)\times B_{R}$ where $B_{R}$ is a ball of radius $R$ in $\R^n$. For $t\ge 0$, let
$I_{t,R}= \int_{B_{R}}\int_{0}^R |s^{N+1}\pd_{s}^NF_{t+s}(x)|^2\, \frac{dsdx}{s}$.

If $t\le R$, we write \begin{align*}
 I_{t,R}   & = \int_{B_{R}}\int_{0}^R s^{2(N+1)-1}|\pd_{s}^NF_{t+s}(x)|^2\, {dsdx}    \\
    &  \le \int_{B_{R}}\int_{0}^R (t+s)^{2(N+1)-1}|\pd_{s}^NF_{t+s}(x)|^2\, {dsdx}
    \\
    &= \int_{B_{R}}\int_{t}^{t+R} s^{2(N+1)-1}|\pd_{s}^NF_{s}(x)|^2\, {dsdx}
    \\
    & \le \int_{B_{2R}}\int_{0}^{2R} s^{2(N+1)-1}|\pd_{s}^NF_{s}(x)|^2\, {dsdx}\\
    & \le (2R)^{n+2\alpha} \|t^{N+1} \pd_{t}^NF\|_{T^\infty_{2,\alpha}}^2.
\end{align*} 

If  $t\ge R$, 
\begin{align*}
 I_{t,R}   & = \int_{B_{R}}\int_{t}^{t+R} (s-t)^{2(N+1)-1}|\pd_{s}^NF_{s}(x)|^2\, {dsdx}    \\
    &  \le \frac{R^{2(N+1)-1}}{t^{2(N+1)-1}} \int_{B_{R}}\int_{t}^{t+R} s^{2(N+1)-1}|\pd_{s}^NF_{s}(x)|^2\, {dsdx}
    \\
    & \le (2t)^{n+2\alpha} \frac{R^{2(N+1)-1}}{t^{2(N+1)-1}} \|t^{N+1} \pd_{t}^NF\|_{T^\infty_{2,\alpha}}^2
    \\
    & \le (2R)^{n+2\alpha}\|t^{N+1} \pd_{t}^NF\|_{T^\infty_{2,\alpha}}^2,
\end{align*}
where the condition on $N$ is used.  

To show the weak-star continuity and convergence, it suffices to test against    $L^2$ functions
$H$ supported in  compacta of $\reu$, as such functions form a dense subspace of $T^q_{2}$. Then, as a function of $t\ge 0$, the integral
$$\iint_{K} \paire {H(s,x)} {s^{N+1}\pd_{s}^NF_{t+s}(x)} \, \frac{dsdx}{s},$$  is clearly continuous and even $C^\infty$ because $F\in C^\infty(0,\infty; L^2_{loc})$ for example. 
\end{proof}

The second tool  we need is the adapted H\"older theory to define $f_{t}$. For $q\le 1$, we begin with the space  $\dot \IH^{1,q}_{B^*D}$ as the set of functions $h\in \clos{\ran_{2}(B^*D)}=\IH^2_{B^*D}$ with 
$\tau^{-1} \Qpsi{\psi}{B^*D}h\in T^q_{2}$, equipped with quasi-norm $\|\tau^{-1} \Qpsi{\psi}{B^*D}h\|_{T^q_{2}}$. Again, this space does not depend on the particular choice of  $\psi$ among  bounded holomorphic functions  in bisectors $S_{\mu}$ with enough decay at 0 and $\infty$ and can be characterized as well by the $\Tpsi{\psi}{B^*D}$ maps. 

Then, we define  $\dot \Lambda^{\alpha-1}_{DB}$ as the dual space of $\dot \IH^{1,q}_{B^*D}$.
Arguing as in Proposition 4.7 of \cite{AS} and using $(T^q_{2})'= T^\infty_{2,\alpha}$, we see that 
given any $\psi$ as above, any linear functional $\ell$ on $\dot \IH^{1,q}_{B^*D}$ can be written as 
$$
\ell(\phi_{0})= \iint_{\reu} \paire{\psi(sB^*D)\phi_{0}(x)} {G(s,x)}\, \frac{dsdx}{s}$$
for some  $G$ with $sG\in T^\infty_{2,\alpha}$ and  $\|\ell\| \sim \|s G\|_{T^\infty_{2,\alpha}}$.  
Letting $\chi_{k}$ be the cut-off function  as before, replacing $G$ by $\chi_{k}G$ leads to a $T^2_{2}$ function, hence 
$$
\iint_{\reu} \paire{\psi(sB^*D)\phi_{0}(x)} {\chi_{k}(s,x)G(s,x)}\, \frac{dsdx}{s} = \pair {\phi_{0}}{h_{k}}
$$
for some $h_{k}\in \IH^2_{DB}$. Arguing as in \cite{AS}, we see that $h_{k}$ belongs in fact to 
$\dot \IL^{\alpha-1}_{DB}$, the pre-H\"older space whose weak-star completion is $\dot \Lambda^{\alpha-1}_{DB}$,  and that 
$\dot \IL^{\alpha-1}_{DB}=\dot \IL^{\alpha-1}_{D}$ with equivalent topology because of the assumption on $q$. As $s\chi_{k}G$  weakly-star converges to $sG$ in $T^\infty_{2,\alpha}$, this means that $h_{k}$ weakly-star converges to some $h\in \dot \Lambda^{\alpha-1}_{D}$. Thus we may write 
$\ell(\phi_{0})=\pair {\phi_{0}}{h}$ and observe that  one can replace $\phi_{0}$ by any $\varphi_{0}\in \mS$ in the formula.

With this in hand, we can define $f_{t}$, and using the $E^q_{\delta }, E^{\infty,\alpha}_{\delta }$ duality, run the same  computation as above to prove the semigroup formula for $f_{t}$ in $ \dot \Lambda^{\alpha-1}_{D}$ and   $\pd_{t}f_{t}=\pd_{t}F_{t}$ in $\mS'$.

Thus,  $F_{t}=f_{t}+G$ for some distribution $G\in \mS'$. As before, we have $F_{t,\ta}\to 0$ in $\mD'$. Now if $\varphi_{0}\in \mS$ and $\phi_{0}=\IP_{B^*D}\varphi_{0}$,  the proof of the semigroup formula contains the equalities
$$
\pair{\varphi_{0}}{f_{t}}= \pair{\phi_{0}}{f_{t}} = \pair{e^{-tB^*D}\chi^+(B^*D)\phi_{0}}{f_{0}}.
$$
The left entry of the last pairing goes to $0$ in $\IH^{1,q}_{B^*D}$ (the proof is completely analogous to Proposition 4.6 of \cite{AS} using the square function characterization of the space $\IH^{1,q}_{B^*D}$). Thus,   $\pair{\varphi_{0}}{f_{t}}$ must tend to 0 as $t\to \infty$. 

We obtain $G_{\ta}=0$ and, as before,   $G_{\no}= a \sigma$, where $\sigma$ is some constant in $\C^m$. Note that since $f_{0}\in  \dot \Lambda^{\alpha-1, +}_{DB}$ we have $tf_{t}\in T^\infty_{2,\alpha}$ (see \cite{AS}, Lemma 14.4 with $\wt B$ replaced by $B$) where we use the assumption on $q$.    Using  $sG= s(F_{s}-f_{s}) \in T^\infty_{2,\alpha}$, we have 
$$\int_{B_{R}}\int_{0}^R |sG(x)|^2\, \frac{dsdx}{s}   \le C R^{n+2\alpha}.
$$
But using $\lambda|\sigma| \le |a(x) \sigma|$ almost everywhere,  hence $|G(x)| \ge \lambda |\sigma|$, 
$$\int_{B_{R}}\int_{0}^R |sG(x)|^2\, \frac{dsdx}{s} \ge \lambda^2 |\sigma|^2 \frac{R^{n+2}}{2}.$$
It follows that $ \lambda^2 |\sigma|^2 R^{2-2\alpha} \le 2 C$. 
Taking $R\to \infty$ forces $\sigma=0$ as $\alpha<1$, and this yields $G=0$ as desired.

\section{Proof  of Corollary \ref{cor:main1}}

Let $u$ be a weak solution to $Lu=0$ on $\reu$ with $\|\tN(\nabla u )\|_{p}<\infty$. By (iii) in Theorem \ref{thm:main1}, we have  $\nabla\!_{A} u(t,\, .\,) \in H^p_{D}$ and
$\nabla\!_{A} u(t,\, .\,)= S_{p}(t)( \nabla\!_{A} u|_{t=0})$ for all $t\ge 0$. As $S_{p}(t)$ is a continuous semigroup on $H^{p}_{D}$, we obtain the continuity for $t\ge 0$. The limit at $\infty$ can be seen 
from Proposition 4.5 in \cite{AS} and taking bounded extension. 
The $C^\infty$ regularity can be obtained as in this same proposition. We skip details.

As in the proof of Theorem \ref{thm:main2},   $D: \dot \IH^{1,p}_{BD}\to \IH^p_{DB}$, where $\dot \IH^{1,p}_{BD}=\dot \W^{1,p}_{BD}$ is $p>1$, is an isomorphism. Thus it extends to an isomorphism, still denoted by $D$,   $D: \dot H^{1,p}_{BD}\to H^p_{DB}$. We also know that $\IP: \dot \IH^{1,p}_{BD} \to \dot \IH^{1,p}_{D}$ is an isomorphism for $p\in I_{L}$, thus its extension to the completed spaces is also an isomorphism. Set  $v(t,\cdot)=- D^{-1}  \nabla\!_{A} u(t,\cdot) \in \dot H^{1,p,+}_{BD} $ for all $t\ge 0$.
Then $t\mapsto v(t,\cdot)$ is uniformly bounded in $\dot H^{1,p,+}_{BD}$. Applying $\IP$,  
$$\sup_{t\ge0}\|\IP v(t,\cdot)\|_{\dot H^{1,p}_{D}}  \sim \|\IP v(0,\cdot)\|_{\dot H^{1,p}_{D}} \sim   \| \nabla\!_{A} u|_{t=0}\|_{H^p}.$$
In particular, $v_{\no}(t,\cdot) \in \dot H^{1,p}$ and $\nabla_{x}v_{\no}(t,\cdot)= \nabla_{x}u(t,\cdot)$, which gives a meaning to $u(t,\cdot) \in \dot H^{1,p}$ for all $t\ge 0$ with the above estimate.
Note also that $\IP v(t,\cdot) \to \IP v(0,\cdot)$ as $t\to 0$ in $\dot H^{1,p}_{D}$. In particular,  $u(t,\cdot) \to u(0,\cdot)$ in $\dot H^{1,p}$.
  
We next show if $p<n$,  that one can select a constant such that $ u(t,\cdot)-c\in L^{p^*}$ for all $t\ge 0$.  
One can write $u|_{t=0}= f+c \in  L^{p^*}+\C^m$. At the same time, it is shown in \cite{KP}, Theorem 3.2, p.462, that 
\begin{equation}
\label{eq:KP}
\left|\bariint_{W(t,x)}  u(s,y) \, dsdy-u(0,x)\right| \lesssim t \tN(\nabla u)(x)
\end{equation}
 almost everywhere (the proof done for $1<p$ extends without change to $\frac{n}{n+1}<p$ and it does not use the specificity of real symmetric equations  if we replace pointwise values of $u(t,x)$ by Whitney averages as here). Finally, there exists $\tilde u(t,\cdot) \in L^{p^*}$ and $c(t)\in \C^m$ such that $u(t,\cdot)= \tilde u(t,\cdot) +c(t)$ almost everywhere ($t$ fixed). Thus, we obtain
\begin{align*} \left|\barint_{[c_{0}^{-1} t, c_{0}t]} c(s)\, ds -c\right|&  = \left|\bariint_{W(t,x)}c(s)\,  dsdy -c\right|  \\ 
&\lesssim \left|\bariint_{W(t,x)} \tilde u(s,y)dsdy-f(x)\right|+  t \tN(\nabla u)(x)
\end{align*}
almost everywhere. As the right hand side belongs to $L^{p^*}+L^p$, this implies that $\barint_{[c_{0}^{-1} t, c_{0}t]} c(s)\, ds -c=0$ for all $t>0$, hence $c(t)=c$ for $t>0$.  
Having this at hand we have 
$$
\sup_{t\ge 0} \|\tilde u(t,\cdot) \|_{p^*} \lesssim  \sup_{t\ge 0} \|\nabla_{x} u(t,\cdot) \|_{H^p} \lesssim 
\| \nabla\!_{A} u|_{t=0}\|_{H^p}.
$$
(Again, $H^p=L^p$ if $p>1$).
On the other hand, as  $\tilde u(t,\cdot)$  belong to $L^{p^*}$ for $t\ge 0$, Sobolev inequality implies 
$$
\|\tilde u(t,\cdot) -\tilde u(s,\cdot)\|_{p^*} \lesssim \|\nabla_{x} u(t,\cdot) - \nabla_{x} u(s,\cdot)\|_{H^p}
$$
which shows continuity on $t\ge 0$ in $L^{p*}$ topology. That the limit is 0 when $t\to \infty$  follows from
$$
\|\tilde u(t,\cdot) \|_{p^*} \lesssim \|\nabla_{x} u(t,\cdot) \|_{H^p}.
$$
The $C^\infty$ regularity follows from repeated use of the Sobolev inequality for $t$-derivatives of $\pd_{t}\tilde u$.
We have proved all the stated properties for $u$ in the case $p<n$.  

 In the case $p\ge n$, we do not have to wonder about the constant and barely use 
 the Sobolev embedding $\dot H^{1,p} \subset \dot \Lambda^s$. Thus the topology on $\dot \Lambda^s$ in this argument is the strong topology. Details are easier and left to the reader.

 To conclude this proof, we turn to almost everywhere convergence. We begin with the one for $u$, namely \eqref{eq:CVaeu1}.  That 
 $$\lim_{t\to 0}\ \bariint_{W(t,x)}  u(s,y) \, dsdy   = u|_{t=0}(x)
 $$
 follows directly from  \eqref{eq:KP}. As 
  $$
\left |  \barint_{B(x,c_{1}t)}  u(t,y)  \, dy -  \bariint_{W(t,x)}  u(s,y) \, dsdy \right|  \lesssim  t\tN(\nabla u)(x),
$$
we also obtain the almost everywhere convergence of slice averages. 

We turn to the almost everywhere convergence result  for the conormal gradient, namely \eqref{eq:CVaegradA}. It is shown in \cite{AS}, Theorem 9.9, that for all $h\in \clos{\ran_{2}(DB)}=H^2_{D}$,  for  almost every $x_{0}\in \R^n$,
\begin{equation}
\label{eq:CVaeh}
\lim_{t\to 0}\ \bariint_{W(t,x_{0})} |e^{-s|DB|}h(y)-h(x_{0})|^2\, dsdy =0.
\end{equation}
This holds in particular for any $h\in \IH^{p,+}_{DB}$, which is by construction a dense class in $H^{p,+}_{DB}$. Next, we know that $h\mapsto \tN(e^{-s|DB|}h)$  is bounded from $\IH^{p,+}_{DB}$ into $L^p$ when $p\in I_{L}$ by \cite[Theorem 9.1]{AS} and this extends by density. Remark that for $p\ge 1$,  $H^p$ embeds in $L^p$, thus  the elements in $H^{p,+}_{DB}\subset H^p_{D}$ are measurable $L^p$ functions. Hence, the classical  density argument for almost everywhere convergence allows us to show that \eqref{eq:CVaeh} extends to all $h\in H^{p,+}_{DB}$ provided we replace $ e^{-s|DB|}$ by its extension $S_{p}(s)$ as usual. This yields \eqref{eq:CVaegradA} applied to   $h=\nabla_{A}u|_{t=0}\in H^{p,+}_{DB}$. 

We have the same argument for the slice averages. We skip further details.  

\begin{rem}
When $p<1$, starting from $h\in H^p_{DB}$, the above argument shows that almost everywhere limit of Whitney averages of $S_{p}(s)h$ exists and defines  a measurable function $h_{0}$ at almost every $x_{0}\in \R^n$. However, this function could be not related to the distribution $h\in H^p_{DB}$. This fact is well known in classical Hardy space theory (see \cite{Stein}, p. 127).
\end{rem}

\section{Proof  of Corollary \ref{cor:main2}}

First assume $p=q'$ with $q\in I_{L^*}$ and $q>1$.  The regularity properties and \eqref{eq:supw-1p} are a consequence of semigroup theory on Banach spaces. We skip details. 

Next, we want  to show  $u=\tilde u + c$ on $\reu$ with $t\mapsto  \tilde u(t,\, .\,) \in C_{0}([0,\infty);L^{p})\cap C^{\infty}(0,\infty; L^p)$ and $c\in \C^m$. Let $h= \nabla\!_{A} u|_{t=0}\in \dot W^{-1,p,+}_{DB}$. There exists 
$\tilde h\in H^{p,+}_{BD}$ such that $D\tilde h= h$. Here again, we use the extension of the isomorphism $D: \IH^p_{BD} \to \dot \W^{-1,p}_{DB}$. Then, $v(t,\cdot):=  S_{p, BD}(t) \tilde h$, where $S_{p, BD}(t)$ is the extension of the semigroup $e^{-t|BD|}$ from $\IH^p_{BD}$ to $H^p_{BD}$, satisfies $Dv(t,\cdot)= \nabla\!_{A}u(t,\cdot)$.  As $\IP: H^{p}_{BD} \to H^p_{D}$ is an isomorphism (\cite{AS}, Theorem 4.20) we have that $t\mapsto v_{\no}(t,\cdot) \in C_{0}([0,\infty);L^{p})\cap C^{\infty}(0,\infty; L^p)$. Following the proof of Theorem 9.3 in \cite{AA1}, we obtain that $u+v_{\no}$ is constant on $\reu$. So our claim holds with $\tilde u=-v_{\no}$ (Another possible argument is sketched in \cite{AS}, Section 14.1). 

Let us see the non-tangential maximal estimate \eqref{eq:ntmax} for $\tilde u$. In fact, let  $\tilde h\in H^p_{BD}$ and approximate by $\tilde h_{k}\in \IH^{p}_{BD}$. This implies $\IP \tilde h_{k} \to \IP \tilde h$ in $L^p$ by the isomorphism property above. 
But 
$$ (e^{-t|BD|}\tilde h_{k})_{\no} = (\IP e^{-t|BD|}\tilde h_{k})_{\no}=  (\IP e^{-t|BD|}\IP \tilde h_{k})_{\no}= ( e^{-t|BD|}\IP\tilde h_{k})_{\no}.
$$
Thus \cite{AS}, Theorem 9.3, yields in our range of $p$, 
$$
\|\tN (e^{-t|BD|}\tilde h_{k})_{\no} \|_{p}\lesssim  \|\IP \tilde h_{k}\|_{p}. 
$$
But $e^{-t|BD|}\tilde h_{k} \to S_{p, BD}(t) \tilde h$ in $L^2_{loc}(\reu)$ and $ \|\IP \tilde h_{k}\|_{p} \to \|\IP \tilde h \|_{p}. $
Thus using Fatou's lemma for arbitrary  linearisations of the left hand side, we obtain
$$
\|\tN (S_{p,BD}(t)\tilde h)_{\no} \|_{p}\lesssim  \|\IP \tilde h\|_{p}. 
$$
Now taking the element $\tilde h\in H^{p,+}_{BD}$ associated to the solution $u$ as above, we have  $(S_{p,BD}(t)\tilde h)_{\no}=-\tilde u$ and  
$$
\|\tN \tilde u \|_{p} \lesssim \|\IP \tilde h\|_{p} \sim \|\nabla\!_{A} u(t,\, .\,)\|_{\dot W^{-1,p}} \sim \|S(t\nabla u)\|_{p}
$$
where the last comparison is Theorem \ref{thm:main1}. Thus, \eqref{eq:ntmax} is proved.

Next, let us see the almost everywhere convergence \eqref{eq:CVaeu}. It suffices to show it for the solution $\tilde u$ that we exhibited. By Theorem 9.9 of \cite{AS},  we have for  $\tilde h\in \IH^2_{BD}$ and almost every $x_{0}\in \R^n$, 
\begin{equation}
\label{eq:CVae}
\lim_{t\to 0}\ \bariint_{W(t,x_{0})} |(e^{-s|BD|}\tilde h-\tilde h)_{\no}(x_{0})|^2 =0.
\end{equation}
As we have the non-tangential maximal estimate $\|\tN(S_{p,BD}(t)\tilde h)_{\no}\|_{p} \lesssim \|\IP\tilde h\|_{p}$ the usual density argument gives the almost everywhere convergence \eqref{eq:CVae} for any $\tilde h\in H^p_{BD}$.  For slice averages, we also have  the maximal estimate and   the almost everywhere convergence  on a dense class (see Remark 9.7 in \cite{AS}) and we conclude as above.  Specializing again to $\tilde h\in H^{p,+}_{BD}$ we obtain \eqref{eq:CVaeu}.

\

We turn to the case $q\le 1$ and $\alpha=n(\frac{1}{q}-1)\in [0,1)$. The regularity  for $t\mapsto \nabla\!_{A}u(t,\cdot)$ follows from semigroup theory and Theorem \ref{thm:main2}, where the topology on $\dot  \Lambda^{\alpha-1}$ is the weak-star topology of dual of $\dot H^{1,q}$. 

Next, in particular, we have $t\mapsto \nabla_{x}u(t,\cdot) \in C_{0}([0,\infty);\dot\Lambda^{\alpha-1})\cap C^\infty(0,\infty; \dot  \Lambda^{\alpha-1}) $, thus  $t\mapsto   u(t,\, .\,) \in C_{0}([0,\infty);\dot\Lambda^\alpha)\cap C^\infty(0,\infty; \dot  \Lambda^{\alpha}) $, with continuity in the sense of the weak-star topology of dual of $H^q$ and \eqref{eq:suplambdasdir} holds.

We finish with the proof of \eqref{eq:global}. Let $u$ be a solution with $t\nabla u\in T^{\infty}_{2,\alpha}$ and $u(t,\cdot)$ converges to 0 in $\mD'$ modulo constant as $t\to \infty$. Let $R>0$ and $B_{R}$ a ball of radius $R$ in $\R^n$. Set $T_{R}=(0,R]\times B_{R} $. Fix $t\ge 0$.  Applying  the classical inequality (see \cite{BS}) 
$$
\bariint_{T_{R}} \bigg| f - \bariint_{T_{R}}f\, \bigg|^2\, dsdy\le C \barint_{B_{R}}\int_{0}^R |\nabla f(s,y)|^2\, sdsdy
$$
 to $f(s,y)= u(t+s,y)$ and using translation invariance, we have that  
 $$
\bariint_{(t,0)+ T_{R}} \bigg| u - \bariint_{(t,0)+T_{R}}u\,\bigg|^2\, dsdy \le C \barint_{B_{R}}\int_{0}^R |\nabla u(t+ s,y)|^2\, sdsdy. $$
By  \eqref{eq:t2alpha} for $(s,y)\mapsto\nabla_{A} u(t+ s,y)$ and    the uniform boundedness of the semigroup $\nabla\!_{A} u|_{t=0} \to \nabla\!_{A} u(t,\cdot)$, we have
$$
\barint_{B_{R}}\int_{0}^R |\nabla u(t+ s,y)|^2\, sdsdy \lesssim R^{2\alpha} \|\nabla\!_{A} u(t,\cdot)\|_{\dot  \Lambda^{\alpha-1}}^2 \lesssim R^{2\alpha}\|\nabla\!_{A} u|_{t=0}\|_{\dot  \Lambda^{\alpha-1}}^2.
$$
Hence,   
we have obtained that 
$$
\bariint_{(t,0)+ T_{R}} \bigg| u - \bariint_{(t,0)+T_{R}}u\,\bigg|^2\, dsdy \lesssim R^{2\alpha}\|\nabla\!_{A} u|_{t=0}\|_{\dot  \Lambda^{\alpha-1}}^2.
$$

As $t$ and $T_{R}$ are arbitrary, this is the $BMO(\reu)$ property for $u$ if $\alpha=0$ and  the $\dot  \Lambda^{\alpha}(\overline{\reu})$ property for $u$ by the  Morrey-Campanato characterization of 
$\dot  \Lambda^{\alpha}(\overline{\reu})$ and  \eqref{eq:global} is proved.

\section{Solvability and well-posedness results}

To make the game a little more symmetric, we make precise the formal fact that Dirichlet and Regularity problems are the same with different topologies. 

\begin{lem}  We have $(D)_{Y}^{L^*}=(R)_{Y^{-1}}^{L^*}$,
 where the regularity problem  $(R)_{Y^{-1}}^{L^*}$ means  $L^*u=0$,  $\nabla_{x}u|_{t=0}\in \dot Y^{-1}$, $t\nabla u\in \wt \mT$.
\end{lem}

\begin{proof} This is an easy consequence of Corollary \ref{cor:main2}. Any solution in $\wt \mT$ satisfies 
$u=\tilde u+c$, where $t\mapsto  \tilde u(t,\, .\,) \in C_{0}([0,\infty);L^{p})$ and $c\in \C^m$.
Assume $(R)_{Y^{-1}}^{L^*}$ is well-posed and let $f\in L^p$. Then 
we can solve $(R)_{Y^{-1}}^{L^*}$ for $g=\nabla_{x}f\in \dot W^{-1,p}$  as boundary data. Then, choose the solution that belongs to $C_{0}([0,\infty);L^{p})$. It solves the Dirichlet problem with convergence in $L^p$ to $f$. 
Conversely,  assume $(D)_{Y}^{L^*}$ is well-posed and let $g\in \dot W^{-1,p}$. Then, there exists $f\in L^p$ such that $g=\nabla_{x}f$ in $\mS'$. Solve the Dirichlet problem $(D)_{Y}^{L^*}$ with data $f$. Then $\nabla_{x}u(t,\cdot)$ converges to $g$ in $\dot W^{-1,p}$.
\end{proof}

\begin{proof}[Proof of Theorem \ref{thm:wpequiv}] 
We prove the first line as the others are the same. By the lemma above, $(D)_{Y}^{L^*}$ is well-posed if and only if the map $\wt\mT/\C^m$ to $\dot Y^{-1}_{\ta}$ 
$: [u+c] \mapsto  \nabla_{x}u|_{t=0}$ is an isomorphism (Here, we use the abuse of notation $u\in \wt\mT$ to mean $t\nabla u \in \wt\mT$ of the introduction). But this map is composed of 
 $\wt \mT/\C^m$ onto $\dot Y^{-1,+}_{D\wt B}$ given by $[u+c] \mapsto \nabla_{A^*}u|_{t=0}$, which is an isomorphism by Theorem \ref{thm:main2},  followed by   $N_{\ta}: \dot Y^{-1,+}_{D\wt B} \to \dot Y^{-1}_{\ta}$. The equivalence follows: existence is the same as ontoness of $N_{\ta}$ and uniqueness is the same as injectivity of $N_{\ta}$. 
 \end{proof}

The following is an extension of a result in \cite{AAH} for pairs of projections in a non-Hilbert context.  

\begin{lem}\label{lem:projections} Assume that $E$ is a quasi-Banach space having two pairs on complementary bounded projections $P_{1}^\pm$ and $P_{2}^\pm$. Call $E_{i}^\pm$ the images of $E$ under $P_{i}^\pm$. 
Then \begin{enumerate}
  \item  $P_{1}^-: E_{2}^+ \to E_{1}^-$ has  \textit{a priori} estimates if and only if $P_{2}^-: E_{1}^+ \to E_{2}^-$ has  \textit{a priori} estimates.
  \item $P_{1}^-: E_{2}^+ \to E_{1}^-$ is onto if and only if $P_{2}^-: E_{1}^+ \to E_{2}^-$ is onto.  \item $P_{1}^-: E_{2}^+ \to E_{1}^-$ is an isomorphism if and only if $P_{2}^-: E_{1}^+ \to E_{2}^-$ is an isomorphism.
\end{enumerate}
 A bounded  operator $T$ has  \textit{a priori} estimates if  $\|Tf\| \gtrsim \|f\|$ for all $f$, in other words, it is  injective with closed range. One can change $(P_{2}^-, P_{1}^-)$ to the other
pair $(P_{2}^+, P_{1}^+)$.
\end{lem}

In the proof and later on, it will be convenient   to consider that  if  $P,Q$ are two projectors then $PQ$ is automatically restricted to the range of $Q$ and maps into the range of $P$: this is how we mean \textit{a priori} estimates, ontoness and isomorphism. With this convention, $P:\ran(Q) \to \ran(P)$ is an isomorphism if and only if  $PQ$ is an isomorphism. 

\begin{proof} The proof is identical to the one in \cite{AAH}, Proposition 2.52.  We give it for completeness. The final claim is just a symmetry observation. We are left with proving the equivalences. We begin with (1).  In fact,  $P_{1}^-: E_{2}^+ \to E_{1}^-$ with  \textit{a priori} estimates  is equivalent to the transversality of the spaces $E_{1}^+$ and $ E_{2}^+$, that is $\|x_{1}+x_{2}\| \gtrsim \|x_{1}\| +\|x_{2}\|$ when $x_{i}\in E_{i}^+$, which is symmetric in the indices $i=1,2$. 
Indeed, assume $P_{1}^-: E_{2}^+ \to E_{1}^-$ has  \textit{a priori} estimates. If $x_{2}\in E_{2}^+$, then 
$$\|x_{2}\| \lesssim \|P_{1}^-x_{2}\|  = \|P_{1}^-(x_{2}+x_{1})\| \lesssim \|x_{1}+x_{2}\|$$
for any $x_{1}\in E_{1}^+$. The transversality follows from the quasi-triangle inequality.
Conversely, assume $E_{1}^+$ and $E_{2}^+$ are transversal. Let $x_{2}\in E_{2}^+$. Set $x_{1}=x_{2}-P_{1}^-x_{2} \in E_{1}^+$.  Thus 
$$
\|P_{1}^-x_{2}\|  = \|x_{2}-x_{1}\| \gtrsim \|x_{2}\|+\|x_{1}\| \ge \|x_{2}\|.
$$
This proves the first  equivalence.

For (2), by symmetry again, it is enough to show that ontoness of    $P_{1}^-: E_{2}^+ \to E_{1}^-$   is equivalent to  $E_{1}^++ E_{2}^+=E$. Let us see that. 
Assume $P_{1}^-: E_{2}^+ \to E_{1}^-$ onto. Let $x\in E$.  There exists $x_{2}\in E_{2}^+$ such that $P_{1}^-x_{2}= P_{1}^-x$. Hence, 
$$
P_{1}^+(x-x_{2})= P_{1}^+x-P_{1}^+x_{2} = x-P_{1}^-x- P_{1}^+x_{2}= x- P_{1}^-x_{2} - P_{1}^+x_{2}= x-{x_{2}}.
$$
Thus $x-x_{2}\in E_{1}^+$. Conversely, assume $E_{1}^++ E_{2}^+=E$. Let $x\in E_{1}^-$. Decompose $x=x_{1}+x_{2}$ with $x_{i}\in E_{i}^+$. Then 
$$ x= P_{1}^-x= P_{1}^-x_{1}+ P_{1}^-x_{2}= P_{1}^-x_{2}.$$ 

The proof of (3) is the conjunction of (1) and (2). 
\end{proof}

\begin{lem}\label{thm:duality}[Simultaneous duality] Assume $q\in I_{L}$. There is a pairing between the spaces  $X_{D}$ and $\dot Y^{-1}_{D}$ for which  $ \dot Y^{-1,\mp}_{D\wt B}$ realizes as the  dual space of $X^\pm_{DB}$,  
$\dot Y^{-1}_{\no}$ as the dual of $X_{\ta}$ and $\dot Y^{-1}_{\ta}$ as the dual of $X_{\no}$.
 If $q>1$, the situation is reflexive so that the pairing is a duality. 
\end{lem}

\begin{proof}  In this proof, $q\in I_{L}$ all the time. We identify $X_{\no}$ with $[X_{\no}, 0]^T$ and $X_{\ta}$ with $[0,X_{\ta}]^T$ so that they becomes subspaces of $X_{D}$ and form a splitting of $X_{D}$.  
Do the same for $\dot Y^{-1}_{D}$. So the goal is to build this simultaneous decomposition. 

Recall that $\wt B$ is the matrix corresponding to $A^*$. It is given by $\wt B= NB^*N$ with $N= \begin{bmatrix} I & 0  \\ 
    0 & -I \end{bmatrix}= N_{\no}-N_{\ta}$ according to the identifications. We have to use it here if we want to put in duality the second order operators $L$ and $L^*$. Hence $N$ will appear in the pairing. 

The pairing is given by the following scheme:    the article \cite{AS} defines spaces $H^{q'}_{\wt BD}$ and $\dot \Lambda^{\alpha}_{\wt BD}$ and a consequence of the theory there is that $D$ (restricted to   $\IH^2_{\wt BD}$ as an unbounded operator) extends to  isomorphisms $H^{q'}_{\wt BD} \to \dot W^{-1, q'}_{D\wt B}=\dot W^{-1, q'}_{D}$ and $\dot \Lambda^{\alpha}_{\wt BD} \to \dot \Lambda^{\alpha-1}_{D\wt B}=\dot \Lambda^{\alpha-1}_{D}$. Let us keep calling $D$ these maps and write $D^{-1}: \dot Y^{-1}_{D}= \dot Y^{-1}_{D\wt B} \to \dot Y_{\wt BD}$. Recall also (\cite{AS}, Section 12.2) that $H^{q'}_{\wt BD}$ is the reflexive dual space of $H^q_{DB}=H^q_{D}$ when $q>1$ and $\dot \Lambda^{\alpha}_{\wt BD}$ is the dual space of $H^q_{DB}=H^q_{D}$ when $q\le 1$ for the duality pairing  $\pair {f}{g}_{N}= \pair {f}{Ng}$ where the pairing $ \pair {f}{g}$ extends the standard $L^2$ inner product.   We can define the desired pairing 
on $X_{D}$,  $\dot Y^{-1}_{D}$ by $\pair {h}{D^{-1}g}_{N}$ where $D^{-1}$ is defined just above. Let us check the duality statements in the lemma for this pairing. 

Call $\chi^\pm_{X}$ the extension of $\chi^\pm(DB)$ on $X_{D}$ and $\chi^\pm_{Y^{-1}}$ the extension of $\chi^\pm(D\wt B)$ on $\dot Y^{-1}_{D}$. We claim that $\chi^\mp_{Y^{-1}}$ is the adjoint of $\chi^\pm_{X}$ in this pairing. Indeed, working with appropriate functions $h$, $g$ in dense classes, 
\begin{multline*}
$$
\pair {\chi^\pm(DB) h}{D^{-1}g}_{N}= \pair {\chi^\pm(DB) h}{ND^{-1}g}= \pair { h}{\chi^\pm(B^*D)ND^{-1}g} =\\
 \pair {h}{N\chi^\mp(\wt BD)D^{-1}g} =  \pair {h}{ND^{-1}\chi^\mp(D\wt B)g}= \pair {h}{D^{-1}\chi^\mp(D\wt B)g}_{N}.
$$
\end{multline*}

The change from $\pm$ to $\mp$ comes from the anti-commutation $ND=-DN$ (see \cite{AS}, Section 12.2). 
It follows that the splitting  $\dot Y^{-1}_{D}= \dot Y^{-1, -}_{D\wt B}\oplus \dot Y^{-1, +}_{D\wt B}$
is adjoint of $X_{D}= X^+_{DB}\oplus X^-_{DB}$ for this pairing.

At the same time, it was proved that  $\IP: \IH^2_{\wt BD}\to \IH^2_{D}$ (where we recall that $\IP$ is the orthogonal projection onto  $\IH^2_{D}$)  extends to an isomorphism $ \dot Y_{\wt BD}\to \dot Y_{D}$, which we denote again by $\IP$.  As (the extension of) $\IP$ preserves $X_{D}$ and commutes with $N$, we have 
$$\pair {h}{D^{-1}g}_{N}= \pair {h}{\IP D^{-1}g}_{N}.$$
 As $N$ preserves scalar and tangential spaces, while $D$ swaps them, using the
 pairing  $\pair {h}{\IP D^{-1}g}_{N}$, it is easy to see  $\dot Y^{-1}_{\no}$ as the dual of $X_{\ta}$ and $\dot Y^{-1}_{\ta}$ as the dual of $X_{\no}$. Also the splitting 
 $\dot Y^{-1}_{D}= \dot Y^{-1}_{\ta}\oplus \dot Y^{-1}_{\no}$  is adjoint of 
 $X_{D}= X_{\no} \oplus  X_{\ta}$ for this pairing.
\end{proof}

\begin{proof}[Proof of Theorem \ref{cor:dualityprinciple}] Let us prove the direction from $(R)_{X}^L$ is well-posed to $(D)_{Y}^{L^*}$ is well-posed.  Because of  Theorem \ref{thm:wpequiv}, it is enough to argue on the maps $N_{\ta}$ in each context. 
We have the situation of Lemma \ref{lem:projections} for $E=X_{D}$ with $P_{2}^\pm=\chi_{X}^\pm$, $P_{1}^-=N_{\ta}$ and $P_{1}^+=N_{\no}$. We assume here that $P_{1}^-: E_{2}^+ \to E_{1}^-$ is an isomorphism.  Thus, $P_{2}^-: E_{1}^+ \to E_{2}^-$ is an isomorphism, which concretely means 
$\chi^-_{X}: X_{\no}\to \dot X^{ -}_{DB}$ is an isomorphism, or equivalently that $\chi^-_{X}N_{\no}$ is an isomorphism (recall that it is understood that this is from the range of $N_{\no}$ onto the range of $\chi^-_{X}$). By taking adjoint for the pairing of the lemma above, we have that $N_{\ta}\chi^+_{\dot Y^{-1}}= ( \chi^-_{X}N_{\no})^*$ is an isomorphism. This precisely means that $N_{\ta}: \dot Y^{-1,+}_{D\wt B} \to \dot Y^{-1}_{\ta}$ is an isomorphism. 
The other implications all have the  same proof. 
\end{proof}

As said in the Introduction, we can recover the Green's formula from such abstract considerations: in the pairing of $\pair {f}{g}_{N}$ of Lemma \ref{thm:duality}, one can prove that  the polar set of $X^+_{DB}$ is precisely $\dot Y^{-1,+}_{D\wt B}$. The orthogonality equation expressing this fact is the Green's formula.   It is also possible to prove it directly, which we do for convenience. 

\begin{proof}[Proof of Theorem \ref{thm:green}]
By Theorems \ref{thm:main1} and \ref{thm:main2}, we know that $h=\nabla\!_{A}u|_{t=0}\in H^{q,+}_{DB}\subset  H^{q}_{D} $ and $g=\nabla_{A^*}w|_{t=0}\in \dot W^{-1,q',+}_{D\wt B} \subset \dot W^{-1,q'}_{D}$ or $\dot \Lambda^{\alpha-1,+}_{D\wt B} \subset \dot \Lambda^{\alpha-1}_{D}$.  Using standard approximation theory, we can  approximate  $h$ in $H^q$ topology by functions $\tilde h_{k}$   in the Schwartz class with compactly supported Fourier transform away from the origin. Then, applying the bounded projection $\IP$,  $\IP \tilde h_{k}\in H^q_{D}$ approximate $h=\IP h$ as well. Note that $\IP \tilde h_{k}\in \IH^2_{D}$ as well as $\dot \mH^{-1/2}_{D}$. Applying the projection $\chi^+(DB)$, we have obtained an approximation $h_{k}\in \IH^{q,+}_{DB} \cap \dot \mH^{-1/2,+}_{DB}$, with $\|h_{k}-h\|_{H^q}\to 0$. 

Similarly, if $q>1$, we can find an approximation $g_{k}\in \dot \W^{-1,q',+}_{D\wt B} \cap \dot \mH^{-1/2,+}_{D\wt B}$ with $\|g_{k}-g\|_{ W^{-1,q'}}\to 0$. If $q\le 1$, $g_{k}\in \dot \IL^{\alpha-1,+}_{D\wt B} \cap \dot \mH^{-1/2,+}_{D\wt B}$ with $g_{k} \to g$ for the weak-star topology on $\dot \Lambda^{\alpha-1}$.  Remark that $h_{k}=\nabla\!_{A}u_{k}|_{t=0}$ where $u_{k}$ is  an  energy solution of   $L$ and $g_{k}= \nabla_{A^*}w_{k}|_{t=0}$ where $w_{k}$ is an energy solution of $L^*$. Thus, we have the Green's formula for energy solutions (see \cite{AM}, Lemma 2.1, in this context),
$$
\pair {\partial_{\nu_{A}}u_{k}|_{t=0}} {w_{k}|_{t=0}} = \pair {u_{k}|_{t=0}}{\partial_{\nu_{A^*}}w_{k}|_{t=0}}.  
$$
The pairings are for the homogeneous Sobolev spaces $\dot H^{-1/2}, \dot H^{1/2}$ and 
$\dot H^{1/2}, \dot H^{-1/2}$. But as said, the different notions of conormal gradients are compatible. Assume $q>1$ first.  We have 
$$\partial_{\nu_{A}}u_{k}|_{t=0} \to \partial_{\nu_{A}}u|_{t=0} \ \mathrm{in}\  L^q,$$
$$
\nabla_{x}w_{k}|_{t=0} \to \nabla_{x}w|_{t=0} \ \mathrm{in}\ \dot W^{-1,p}.$$
We can arrange (see Corollary \ref{cor:main2}) that both $w_{k}|_{t=0}$ and $w|_{t=0}$ are in $L^p$ and it follows that 
$$
w_{k}|_{t=0} \to w|_{t=0}  \ \mathrm{in}\  L^p.$$
Having fixed the constant, the first pairing can be reinterpreted in the $L^q, L^p$ duality and converges to 
$\pair {\partial_{\nu_{A}}u|_{t=0}} {w|_{t=0}}$. 
For the second one, we have
$$
u_{k}|_{t=0} \to u|_{t=0}  \ \mathrm{in}\  \dot W^{1,q}, $$ 
$$
\partial_{\nu_{A^*}}w_{k}|_{t=0} \to \partial_{\nu_{A^*}}w|_{t=0}  \ \mathrm{in}\  \dot W^{-1,p},
$$
and the pairing, reinterpreted in the $\dot W^{1,q}, \dot W^{-1,p}$ duality,  converges to $\pair {u|_{t=0}}{\partial_{\nu_{A^*}}w|_{t=0}}$. 

If $q\le 1$, then
 
$$\partial_{\nu_{A}}u_{k}|_{t=0} \to \partial_{\nu_{A}}u|_{t=0} \ \mathrm{in}\  H^q,$$
$$
w_{k}|_{t=0} \to w|_{t=0} \ \mathrm{weakly-star\ in}\ \dot \Lambda^\alpha.$$
 (see Corollary \ref{cor:main2}.)

Thus the first pairing can be reinterpreted in the $H^q, \dot \Lambda^\alpha$ duality and converges to  $\pair {\partial_{\nu_{A}}u|_{t=0}} {w|_{t=0}}$. 
For the second one, we have
$$
\nabla_{x}u_{k}|_{t=0} \to \nabla_{x}u|_{t=0}  \ \mathrm{in}\  H^{q}, $$ 
$$
\partial_{\nu_{A^*}}w_{k}|_{t=0} \to \partial_{\nu_{A^*}}w|_{t=0}  \ \mathrm{weakly-star\ in}\  \dot \Lambda^{\alpha-1},
$$
and the pairing, reinterpreted in the $\dot H^{1,q}, \dot \Lambda^{\alpha-1}$ duality,  converges to $\pair {u|_{t=0}}{\partial_{\nu_{A^*}}w|_{t=0}}$. 
\end{proof}

\begin{proof}[Proof of Theorem \ref{thm:solvimplieswp}]
It is shown in \cite{AS} (Lemma 14.2, 14.3, 14.5, 14.6) that for each of the four problems, solvability for the energy class implies that the attached map  $N_{\ta}$ or $N_{\no}$ is an isomorphism from the corresponding trace space onto 
the  space of boundary data. By  Theorem \ref{thm:wpequiv}, this implies well-posedness. When the data is an ``energy data'',  our assumption is that the energy solution also solves the same problem. By the just obtained uniqueness, it is the only one. This yields compatible well-posedness. \end{proof}

\begin{proof}[Proof of Theorem \ref{thm:solvvsuniq}] Let us prove for example the second item as any other one is the same. Assume $(D)_{Y}^{L^*}$ is solvable  for the energy class. 

We fist show solvability for the energy class of $(R)_{X}^L$.  Let $g\in X^+_{\ta}\cap \dot \mH^{-1/2}_{\ta}$ and consider the energy solution $u$ attached to $g$ (Dirichlet problem). We  want to show that $\|\pd_{\nu_{A}}u|_{t=0}\|_{X_{\no}} \lesssim  \|g\|_{X_{\ta}}$.  Consider a test function $\varphi\in \mS$ and the energy solution of the Dirichlet problem $(D)_{Y}^{L^*}$ for this data. This means that if $w$ denotes this energy solution, we have $\|\pd_{\nu_{A^*}}w|_{t=0}\|_{\dot Y^{-1}_{\no}} \lesssim  \|\nabla \varphi\|_{\dot Y^{-1}_{\ta}}\lesssim \|\varphi\|_{\dot Y_{{\no}}}$. Applying  
\eqref{eq:green}, we have
$$
\pair {\partial_{\nu_{A}}u|_{t=0}} {\varphi}= \pair {\partial_{\nu_{A}}u|_{t=0}} {w|_{t=0}} = \pair {u|_{t=0}}{\partial_{\nu_{A^*}}w|_{t=0}}
$$
hence
$$
|\pair {\partial_{\nu_{A}}u|_{t=0}} {\varphi}| \le \|g\|_{X_{\ta}}\|\|\pd_{\nu_{A^*}}w|_{t=0}\|_{\dot Y^{-1}_{\no}} \lesssim \|g\|_{X_{\ta}} \|\varphi\|_{\dot Y_{{\no}}}.
$$
Thus, by density of test functions in $L^{q'}$ if $q>1$,  and in $VMO$ (the closure of $C_{0}^\infty$ for the $BMO$ norm) if $q=1$, and dualities $L^{q'}-L^q$ and $VMO-H^1$, we have $\partial_{\nu_{A}}u|_{t=0}\in X_{\no}$ as desired. 

Let us now prove the uniqueness of $(R)_{X}^L$. Assume  $u$ is any solution with $\nabla u \in \mN$ and $\nabla_{x}u|_{t=0}=0$. We know that $\partial_{\nu_{A}}u|_{t=0}\in X_{\no}$ and want to show it is 0, so that $u=0$.  We may apply \eqref{eq:green} against the same $w$ as before. This time, we obtain $\pair {\partial_{\nu_{A}}u|_{t=0}} {\varphi}=0$ for all $\varphi\in \mS$, hence $\partial_{\nu_{A}}u|_{t=0}=0$. 
\end{proof}

\begin{rem}
In the previous argument, if $q<1$, then $\partial_{\nu_{A}}u|_{t=0} \in \dot \mH^{-1/2}$, hence it is a Schwartz distribution and it belongs to the dual of the closure of test functions in $\dot\Lambda^{s}$. 
\end{rem}

\begin{rem}
As in \cite{AM} or \cite{HKMP2}, this argument only uses solvability for the energy class \textbf{with compactly supported smooth data}. Thus combining (1) and (2) of Theorem \ref{thm:solvvsuniq} when $q=1$ shows that 
if $(D)_{Y}^{L^*}$ is solvable  for the energy class for any $VMO$ data implies compatible well-posedness (and in particular existence of a solution) for any $BMO$ data. 
\end{rem}

We continue with the 

\begin{proof}[Proof of Corollary \ref{cor:real}] As mentioned in Introduction, \cite{HKMP1} implies 
$(D)_{L^p}^{L^*}$ is solvable for the energy class for some $p\ge 2$ when the coefficients are real. Thus $(R)_{L^{p'}}^{L}$  is solvable for the energy class.  
In the case of real equations,  we have the De Giorgi-Nash-Moser assumptions on $L$ and its adjoint and also on the reflected operator  and its adjoint across $\R^n$ because all four have real coefficients. Thus we may use Theorem 10.1 in \cite{AM} so that $(R)_{L^{q}}^{L}$ for $1<q<p'$ and $(R)_{H^q}^{L}$ are solvable for the energy class and the lower bound on $q$ is determined by the common De Giorgi exponent of all four operators. Comparison of assumptions show  that  this exponent is larger than or equal to the exponent  $p_{\ta}$ found in \cite{AS}, Section 13. In particular, all the exponents $q$ here belong to $I_{L}$.  Thus, we may apply Theorem \ref{thm:solvimplieswp} and compatible well-posedness follows. 

Applying now Theorem \ref{thm:solvvsuniq} settles the Dirichlet problem for $L$. 

When we perturb this situation, we may apply Theorem 14.8 in \cite{AS}, whose proof gives in fact invertibility of the $N_{\ta}$ map and Theorem \ref{thm:wpequiv}, (2), yields well-posedness of the regularity problem for $L^*$  in the same range of exponents $q$. Applying Theorem \ref{cor:dualityprinciple} gives us the dual range of the  Dirichlet problem for $L$. 
\end{proof}

Let us finish with the proof of Theorem \ref{thm:layerpot}. There is a standard argument that uses jump relations that we have here. For example, adapt the proof of \cite{HKMP2}. Instead, we exploit the pairs of projections with a complementary result. 

\begin{lem}\label{lem:projections2} Assume that $E$ is a quasi-Banach space having two pairs on complementary bounded projections $P_{1}^\pm$ and $P_{2}^\pm$.  Then   $P_{1}^-P_{2}^+ P_{1}^+$ is an isomorphism if and only if 
the two  operators  $P_{1}^-P_{2}^+$ and $P_{1}^- P_{2}^-$ are isomorphisms.
\end{lem}

We adopt here as well the convention that product $PQR$ of projectors are restricted to the range of $R$ into the range of $P$. 

\begin{proof}  Remark that $P_{1}^-P_{2}^- P_{1}^+= P_{1}^-(I-P_{2}^+) P_{1}^+= - P_{1}^-P_{2}^+ P_{1}^+$. So the first assertion is equivalent to the two operators  $P_{1}^-P_{2}^\pm P_{1}^+$ are isomorphisms.

Assume that the two operators $P_{1}^-P_{2}^\pm P_{1}^+$ are isomorphisms. First
$P_{1}^-P_{2}^+$ is onto and $P_{2}^-P_{1}^+$ has \textit{a priori} estimates. By  Lemma \ref{lem:projections}, (a), this means that $P_{1}^-P_{2}^+$ also has \textit{a priori} estimates. Thus, it is an isomorphism. One does analogously with  $P_{1}^-P_{2}^-$.

Conversely, assume the two  operators  $P_{1}^-P_{2}^+$ and $P_{1}^- P_{2}^-$ are isomorphisms. By Lemma \ref{lem:projections}, (c), $P_{2}^+P_{1}^+$ is also an isomorphism. 
  Thus,  $P_{1}^-P_{2}^+P_{1}^+=P_{1}^-P_{2}^+P_{2}^+P_{1}^+ $ is an isomorphism.  
  \end{proof}

\begin{proof}[Proof of Theorem \ref{thm:layerpot}] 
Let us prove (1).  Exactly as for the upper half-space, well-posedness of  $(R)_{X}^L$ on the lower half-space is equivalent to $N_{\ta}: X^{-}_{DB} \to X_{\ta}$  is an isomorphism.  Thus simultaneous well-posedness is equivalent  to both $N_{\ta}: X^{\pm}_{DB} \to X_{\ta}$ being isomorphisms, thus  to  $N_{\ta} \chi^{\pm}_{X}$ being  isomorphisms  (identifying again the $\no$ and $\ta$ spaces as subspaces and using our convention for product of projections), where $\chi^{\pm}_{X}$ are the continuous extensions of $\chi^\pm(DB)$ to $X_{DB}$. By the previous lemma, this is equivalent to 
$N_{\ta} \chi^{+}_{X}N_{\no}$ being an isomorphism. We claim this is exactly  saying that the single layer potential $\mS_{0}^L$ is invertible from $X$ onto $\dot X^{1}$. 

Indeed, the single layer  $\mS_{0\pm}^Lf$ using the DB method is defined in \cite{R1} and is shown to coincide with the usual definition on appropriate functions $f$ and when $Lu=0$ is a real equation for instance. See \cite{AS}, Section 12.3. In fact, by density (see  Theorem 2.6 in \cite{AS}), one can see  that for  $f\in X$, as $ \begin{bmatrix} f \\ 0\end{bmatrix}\in X_{D}=X_{DB}$,
$$
\nabla\!_{A}\mS_{0\pm}^Lf = \pm\chi^{\pm}_{X}\begin{bmatrix} f \\ 0\end{bmatrix}.
$$
Thus, 
 $$\begin{bmatrix} 0 \\ \nabla_{x}\mS_{0\pm}^Lf \end{bmatrix}=   \pm N_{\ta}\chi^{\pm}_{X}N_{\no}h
 $$
 for any $h\in X_{DB}$ such that  $h_{\no}=f$, for example $\begin{bmatrix} f \\ 0\end{bmatrix}$. 
 Note that the coincidence of $\nabla_{x}\mS_{0+}^L$ and $\nabla_{x}\mS_{0-}^L$ is exactly the formula $N_{\ta}\chi^{+}_{X}N_{\no}= - N_{\ta}\chi^{-}_{X}N_{\no}$, so that we may remove the signs symbols. 
 
 Thus $N_{\ta}\chi^{+}_{X}N_{\no}$ being an isomorphism  is equivalent to $\nabla_{x}\mS_{0}^L$ being an isomorphism from $X_{\no}$ onto $X_{\ta}$,  that is,  $\mS_{0}^L$ is an isomorphism from $X$ onto $\dot X^1$.

 Let us turn to the proof of (2). As above, the well-posedness of $(D)_{Y}^{L^*}$ is equivalent to 
 $N_{\ta}\chi^{+}_{Y^{-1}}N_{\no}$ being an isomorphism, with obvious notation. As $N_{\ta}\chi^{+}_{Y^{-1}}N_{\no}= - N_{\ta}\chi^{-}_{Y^{-1}}N_{\no}$ (jump relation), this is equivalent to $N_{\ta}\chi^{-}_{Y^{-1}}N_{\no}$ being an isomorphism. We claim this is exactly  saying that the single layer potential $\mS_{0}^{L^*}$ is invertible from $\dot Y^{-1}$ onto $Y$.
 
 In fact, by density from  Lemma 8.1  in \cite{AM}, one can see  that for  $f\in  \dot Y^{-1}$, as $ \begin{bmatrix} f \\0\end{bmatrix}\in \dot Y^{-1}_{D}=\dot Y^{-1}_{DB}$,
$$
\nabla_{A^*}\mS_{0\pm}^{L^*}f = \pm\chi^{\pm}_{\dot Y^{-1}}\begin{bmatrix} f \\ 0\end{bmatrix}.
$$
Thus, 
 $$\begin{bmatrix} 0 \\ \nabla_{x}\mS_{0\pm}^{L^*}f \end{bmatrix}=  \pm N_{\ta}\chi^{\pm}_{Y^{-1}}N_{\no}h
 $$
 for any $h\in \dot Y^{-1}_{DB}$ such that  $h_{\no}=f$, for example $\begin{bmatrix} f \\ 0\end{bmatrix}$.  Again, $\nabla_{x}\mS_{0+}^{L^*}$  and $\nabla_{x}\mS_{0-}^{L^*}$ agree, and  $N_{\ta}\chi^{+}_{Y^{-1}}N_{\no}$being an  isomorphism  is equivalent to $\mS_{0}^{L^*}$ being an isomorphism from $\dot Y^{-1}$ onto $Y$. 
 
 We continue with the proof of (3).  As for (1), well-posedness of $(N)_{X}^L$ on the lower half-space is equivalent to $N_{\no}: X^{-}_{DB} \to X_{\no}$  being an isomorphism.  Thus simultaneous well-posedness on both half-spaces is equivalent   to 
$N_{\no} \chi^{+}_{X}N_{\ta}$ being an isomorphism. We claim this is exactly  saying that the operator $\pd_{\nu_{A}}\mD_{0}^L$ is invertible from  $\dot X^{1}$ onto $X$. 

  In fact, by density from  Lemma 8.1  in \cite{AM}, one can see  that for  $f\in \dot X^1$, as $ \begin{bmatrix} 0 \\ \nabla_{x}f\end{bmatrix}\in X_{D}=X_{DB}$,
$$
\nabla\!_{A}\mD_{0\pm}^Lf = \mp\chi^{\pm}_{X}\begin{bmatrix} 0 \\ \nabla_{x}f \end{bmatrix}.
$$
Thus, 
 $$\begin{bmatrix} \pd_{\nu_{A}}\mD_{0\pm}^Lf \\  0 \end{bmatrix}=   \mp N_{\no}\chi^{\pm}_{X}N_{\ta}h
 $$
 for any $h\in X_{DB}$ such that  $h_{\no}=f$, for example $\begin{bmatrix} 0 \\ \nabla_{x}f\end{bmatrix}$. 
 Note that the coincidence of $\pd_{\nu_{A}}\mD_{0+}^L$ and $\pd_{\nu_{A}}\mD_{0-}^L$ is exactly the formula $N_{\no}\chi^{+}_{X}N_{\ta}= - N_{\no}\chi^{-}_{X}N_{\ta}$, so that we may remove the signs symbols. 
 
 Thus $N_{\no}\chi^{+}_{X}N_{\ta}$ being an  isomorphism  is equivalent to $\pd_{\nu_{A}}\mD_{0}^L$ being an isomorphism from $\dot X^1$ onto $X$.
 
 The proof of (4) is exactly similar to the other ones and we skip it. 
\end{proof}

\section{Specific situations}

We want to comment and illustrate our results in two cases. 

\subsection{Constant coefficients}

We assume that $L$ has constant coefficients given by $A$. We observe that ther are may different choices for $A$ but this does not affect the results in this section. We know from \cite{AS} that $I_{L}=(\frac{n}{n+1},\infty)$. A reformulation  in the current setup of the  results  in \cite{AAMc}, Section 4.3, establishes the invertibility of  the four  maps   of Theorem \ref{thm:wpequiv} when $X=L^2$ and $Y=L^2$. Thus, 
Dirichlet problem, regularity problem and Neumann problem (with any given choice of $A$ representing $L$) are well-posed  (with our current understanding of it) in $L^2$.  Applying this result using  (2) and then  (4)  in Theorem \ref{thm:wpequiv}, we have that the map $N_{\no}N_{\ta}^{-1}: L^2_{\ta} \to L^2_{\no}$ is  bounded and invertible. But by taking the Fourier transform, this map and its inverse are  Fourier multipliers with homogeneous  symbols of degree 0, smooth away from the origin. Thus they are classical singular integral operators and, as a consequence,  bounded on any of the spaces $X$ or $\dot Y^{-1}$ corresponding to $I_{L}$, and also for the space $\dot W^{-1/2,2}$ corresponding to energy solutions in which they coincide with  the Dirichlet to Neumann map and its inverse (another way is to observe that examination of the proof in \cite{AAMc} shows that directly). This implies  from our general results  (Theorem \ref{thm:solvimplieswp}) that all possible problems are compatibly well-posed on the upper half-space (and, of course, on the lower half-space by a similar argument).

\subsection{Block case} 

In the block case, we have $L=-\divv A \nabla$ with 
 $$
A= \begin{bmatrix} a & 0 \\ 0 & d \end{bmatrix},$$
that is,  $A$ is block diagonal. In this case, $B$ is also block diagonal with
$$
B= \begin{bmatrix} a^{-1}& 0 \\ 0 & d \end{bmatrix}.$$
The Hardy  space theory for  $DB$ was described in \cite{AS} in terms of the operator $E=-\divv d \nabla a^{-1}$ (we change notation because $L$ is here  our second order operator)  and it is proved that $\|\sqrt E (af)\|_{H^p}\sim \|\nabla_{x}f\|_{H^p} $ for any $p\in I_{L}$ and $f  \in \dot W^{1,2}$, where $H^p=L^p$ if $p>1$.  The case $p=2$, is the consequence of \cite{AKMc}).  
A little more work gives the description of the spectral Hardy spaces by 
$$
H^{2,\pm}_{DB}= \{ [\mp\sqrt E \, (af), \nabla_{x}f]^T;  f\in \dot W^{1,2}\}. 
$$
  In this   case, we thus know the \textit{a priori} inequalities for the energy solution which correspond to $\nabla_{x}f\in \dot W^{-1/2,2}$ or $ \sqrt E \, (af)\in \dot W^{-1/2,2}$. These two conditions are equivalent because one can interpolate  the isomorphisms $\sqrt E \,a : \dot W^{1, 2} \to  L^2$ and  $\sqrt E\, a: L^2\to \dot W^{-1,2}$, where the second is  the dual statement of the first one for $E^*$ and similarity.  Thus,  for all energy solutions  $u$, we have \textit{a priori} in the range of $p$ above, 
  $$
  \|\pd_{\nu_{A}}u|_{t=0}\|_{p}\sim \|\nabla_{x}u|_{t=0}\|_{p}.
  $$
This means that the Neumann and regularity problems are solvable in $L^p/H^p$ for the energy class  for any $p\in I_{L}$ (see \cite{AM} for statements), thus compatibly well-posed in this range by Theorem \ref{thm:solvimplieswp}. By Theorem \ref{thm:solvvsuniq}, we have the dual compatibly well-posed Dirichlet and Neumann problems for $L^*$. We leave details to the reader. 

\begin{rem}
 An interesting observation concerns Theorem 4.1 in \cite{May}. There, it is assumed that $a=1$ and it it is proved a non-tangential maximal estimate for  $u(t,\cdot) = e^{-t\sqrt E} f$ with $f\in \dot W^{1,2}$,
$$
\|\tN (\nabla u)\|_{p} \lesssim \|\nabla_{x}f\|_{p},
$$
for values of $p$ smaller that the lower bound of $I_{L}$ in case $\inf I_{L}>1$. Looking at the algebra, $\nabla\!_{A}u=\nabla u = e^{-t|DB|}h$ with $h=[\mp\sqrt E \, f, \nabla_{x}f]^T$ so this agrees with our way of constructing solutions. The maximal estimate uses the fact that $\|\sqrt E f\|_{p} \lesssim \|\nabla_{x}f\|_{p}$ may hold whereas the opposite inequality fails (\cite{Auscher}).     By the results of \cite{KP}, one knows at least weak convergence in $L^p$ of  averages in time  for $\nabla u$ towards $h$ under the non-tangential maximal control. So this is a solution of a slightly modified regularity problem in $L^p$ for these values of $p$. 
 \end{rem}

\section{Proofs of technical lemmas}\label{sec:technical}

The  following is the key lemma. 

\begin{lem}\label{lem:T} Let $T$ be any operator of the form $DB$ or $BD$. For each integer $N$, 
there exists $\phi^\pm \in H^\infty(S_{\mu})$ with the following properties:

\begin{enumerate}
  \item $\phi^\pm(sT)$ is uniformly bounded on $L^2$. 

  \item  $\phi^\pm(sT)$ coincides with $e^{-s |T|}$ on $\IH^{2,\pm}_{T}$.

  \item  $(\phi^\pm(sT))_{s>0}$ has $L^2$ off-diagonal decay of order $N$. 

  \item  $\phi^\pm(sT) \to  I$ strongly in $L^2$ as $s\to 0$.

\end{enumerate}
Here, $S_{\mu}$ is an open bisector  with angle $\omega<\mu<\pi/2$.

\end{lem}

\begin{proof} As in Section 5 in \cite{AS}, one can find coefficients $c_{m}$ such that
$$
\psi^{+}(z):= e^{-z}- \sum_{m=1}^N c_{m}(1+imz)^{-k} = O(z^N)
$$
near 0. For $z\in \C$, set  
$$
\phi^{+}(z):= \sum_{m=1}^M c_{m}(1+imz)^{-k} + \psi^{+}(z)\chi^+(z). 
$$
One can easily see that $\phi^{+}\in  H^\infty(S_{\mu})$. 

Next, the resolvents  $(I+ismT)^{-1}$ are uniformly  bounded on $L^2$ with respect to $s$, while $\psi^{+}\chi^+ \in \Psi(S_{\mu})$, hence $(\psi^{+}\chi^+)(sT)$ is defined on $L^2$ and its operator norm is bounded by $  \|\psi^{+}\chi^+\|_{\infty}$ by the $H^\infty$-calculus of $T$ on $L^2$. This proves (1). 

To see (2), it is enough to show that $\phi^{+}(z) \chi^+(z)= e^{-z}\chi^+(z)$. But this is immediate by construction.

Next, the resolvent of $T$ has  $L^2$ off-diagonal decay of any order (\cite{AS}, Lemma 2.3), while 
$ \psi^{+}\chi^+ \in \Psi_{N}^1(S_{\mu})$, hence it has   $L^2$ off-diagonal decay of  order $N$ (\cite{AS}, Proposition 3.13). Thus (3) holds. 

Finally, $\phi^{+}(sz)$ converges uniformly to $1$ on compact subsets of $S_{\mu}$ as $s\to 0$. By the $H^\infty$-calculus on $\clos{\ran_{2}(T)}$, this implies that $\phi^{+}(sT)$ converges strongly to $I$ as $s\to 0$ on $\clos{\ran_{2}(T)}$. On $\nul_{2}(T)$, we have  $(\psi^{+}\chi^+)(sT)=0$, hence $\phi^{+}(sT)=(\sum c_{m}) I=I$. Thus (4) is proved.  

The proof is the same to get  $\phi^-$, picking (different) coefficients such that $$
\psi^{-}(z)= e^{z}- \sum_{m=1}^N c'_{m}(1+imz)^{-k} = O(z^N)
$$
near 0.  
\end{proof}

\begin{proof}[Proof of Lemma \ref{lemma2}] Let $h\in \IH^{p',+}_{B^*D}\cap E^{p'}_{\delta }$. Fix $\delta >0$.
Then $e^{-s |B^*D|}h  \in E^{p'}_{\delta }$ using Lemma \ref{lemma1}. 
It follows from  Lemma \ref{lem:T} that $e^{-s |B^*D|}h= T_{s}h$, where $T_{s}=\phi^+(sB^*D)$ has 
$L^2$ off-diagonal decay of order $N$ with $N$ at our disposal and $T_{s}\to I$ in $L^2$ strongly as $s\to 0$. By Proposition \ref{prop:cv}, we conclude if $N>\inf (n/p',n/2)$ that $e^{-s |B^*D|}h$ converges to $h$ in $E^{p'}_{\delta }$. 

The argument for $h\in  \IH^{p',-}_{B^*D}\cap E^{p'}_{\delta }$ is the same. 
\end{proof}

\begin{proof}[Proof of Lemma \ref{lemma3}] If $\varphi_{0}\in \mS$, then $\phi_{0}=\IP_{B^*D}\varphi_{0}\in \IH^2_{B^*D}$. Also
$\IP\phi_{0}=\IP\varphi_{0}\in L^{p'}$. Thus $\phi_{0} \in \IH^{p'}_{B^*D}$ by \cite{AS}, Theorem 4.20 and our range for $p'$. Since $\phi_{0}-\varphi_{0}\in \nul_{2}(D)$, $D\phi_{0}= D\varphi_{0}$. 
Thus  $B^*D\phi_{0}=B^*D\varphi_{0}$ and the latter is of course in $\IH^2_{B^*D}= \clos{\ran_{2}(B^*D)}$ and in $L^{p'}$, thus $\IP B^*D\varphi_{0}\in L^{p'}$ and,  again, $B^*D\varphi_{0} \in \IH^{p'}_{B^*D}$. Thus, $\chi^\pm(B^*D)B^*D\phi_{0}\in  \IH^{p',\pm}_{B^*D}$ by the $H^\infty$-calculus on this space for $B^*D$.  

It remains to prove that $\chi^\pm(B^*D)B^*D\varphi_{0}\in   E^{p'}_{\delta }$. As $B^*D\varphi_{0}= \chi^+(B^*D)B^*D\varphi_{0}+ \chi^-(B^*D)B^*D\varphi_{0}$ and the left hand side is clearly in $E^{p'}_{\delta }$ (for example, $D\varphi_{0}\in \mS$ so contained in $E^{p'}_{\delta }$ and multiplication by $B^*$ preserves $E^{p'}_{\delta }$), it suffices to do it for $\chi^+(B^*D)B^*D\varphi_{0}$.  Pick $\psi\in \Psi_{N}^N(S_{\mu})$ with $N$ large enough and 
$\int_{0}^\infty \psi(sz)\, \frac{ds}{s}=1$ for all $z\in S_{\mu}$. Set $\phi(z)= \int_{1}^\infty \psi(sz)\, \frac{ds}{s}$. Thus $\phi\in H^\infty(S_{\mu})$, $\phi(z)- 1=O(z^N)$  and $\phi(z)= O(z^{-N})$. Set $h=\chi^+(B^*D)B^*D\varphi_{0}$. Then by $H^\infty$-calculus on $L^2$, and as $h\in \clos{\ran_{2}(B^*D)}$, we have
$$
h= \int_{0}^\delta (\chi^+\psi)(sB^*D)(B^*D\varphi_{0}) \, \frac{ds}{s} + \phi(\delta B^*D) h.
$$ 
As in Lemma \ref{lemma1}, we have $\phi(\delta B^*D) h \in E^{p'}_{\delta }$.  For the other term, 
we notice that $\chi^+\psi\in \Psi_{N}^N(S_{\mu})$, so $(\chi^+\psi)(sB^*D)$ has $L^2$ off-diagonal estimates of order $N$. Using this with $N$ large, $L^2$ boundedness of the operator $\int_{0}^\delta (\chi^+\psi)(sB^*D) \, \frac{ds}{s}$ and the fact that $B^*D\varphi_{0}$ satisfies 
$$
\int_{\R^n} |B^*D\varphi_{0}(x)|^2  \brac {x/\delta }^{-M}\, dx <\infty
$$
for any $M>0$, it is not hard, though a little tedious, to check that  the function  $g=  \int_{0}^\delta (\chi^+\psi)(B^*D)(B^*D\varphi_{0}) \, \frac{ds}{s}$ satisfies the estimate
$$
\int_{\R^n} |g(x)|^2  \brac {x/\delta }^{-2N}\, dx <\infty.
$$
It follows that for all $x\in \R^n$, 
$$
\bigg(\barint_{B(x,\delta )} |g|^2\, \bigg)^{1/2} \lesssim \delta ^{-n/2}  \brac {x/\delta }^{-N},
 $$
 which is $L^{p'}$ as a function of $x$ when $Np'>1$, in which case $g\in E^{p'}_{\delta }$.  We skip further details. 
\end{proof}

\begin{proof}[Proof of Lemma \ref{lemma3th1.2}] Recall that $q\in I_{L^*}$. 
Clearly $\varphi_{0}\in \mS$ implies $\varphi_{0}\in \dom_{2}(D)$ and $D\varphi_{0}\in \IH^q_{D}=\IH^q_{DB^*}$ given the description of $\IH^q_{D}$ and the value of $q$. That
$ \chi^\pm(DB^*) D\varphi_{0 } \in E^{q}_{\delta }$   is trivial if $q\ge 2$. Indeed,  $ \chi^\pm(DB^*) D\varphi_{0 }  \in L^q \subset E^q_{\delta }$ in that case. If $q<2$, then one adapts the argument just given in the proof of Lemma \ref{lemma3} above. 
\end{proof}

\begin{proof}
[Proof of Lemma \ref{lem:density}] Notice that we stated the lemma for $q>1$ using $\dot \W^{1,q}_{B^*D}$ but the lemma and its proof  are valid with $q\le 1$ using the space $\dot \IH^{1,q}_{B^*D}$ instead. We prove both cases simultaneously. Let us first  see the inclusion.  Let $\phi_{0}\in \D_{q}\cap \IH^2_{B^*D}$. We have to show that $\tau^{-1} \psi(\tau B^*D)\phi_{0} \in T^q_{2}$ where $\psi\in \Psi(S_{\mu})$ has sufficiently large decay at 0 and infinity. Writing  $\psi(z)= z\tilde \psi(z)$, we have $$
\tau^{-1} \psi(\tau B^*D)\phi_{0}= B^*\tilde \psi(\tau DB^*) D\phi_{0}.
$$
As $D\phi_{0}\in \IH^q_{D}=\IH^q_{DB^*}$,  the tent space characterization of $\IH^q_{DB^*}$, yields $\tilde \psi(\tau DB^*) D\phi_{0}\in T^q_{2}$.   The boundedness of $B^*$ allows us to conclude for the inclusion. 

For the density, we use that for any $h\in \dot \W^{1,q}_{B^*D}$ (or $\dot \IH^{1,q}_{B^*D}$), we have $e^{-s|B^*D|}h\to h$ as $s\to 0$ in that space. The proof is exactly the same as the one of \cite{AS}, Proposition 4.5.  Now, by Lemma \ref{cor:stability}, we have that $e^{-s|B^*D|}h \in \D_{q}$ and it also belongs to $ \IH^2_{B^*D}$, as $h$ belongs to $\IH^2_{B^*D}$. 
 \end{proof}

\end{document}